\documentclass[letterpaper, english, 11pt]{smfart}

\usepackage[all]{xy}

\usepackage{setspace}
\usepackage[british]{babel}

\usepackage[OT2,T1]{fontenc}

\usepackage{amssymb}

\usepackage{stmaryrd}
\usepackage{amsthm}

\usepackage{amsmath}
\usepackage[latin1]{inputenc}

\allowdisplaybreaks[1]

\usepackage{graphicx}

\usepackage[margin=3.3cm]{geometry}

\usepackage{smfthm}

\usepackage{enumerate}

\usepackage[pagebackref=true, colorlinks=true, linkcolor=black, citecolor=blue, urlcolor=black]{hyperref}

\usepackage[msc-links, lite]{amsrefs}

\DeclareSymbolFont{cyrletters}{OT2}{wncyr}{m}{n}
\DeclareMathSymbol{\Sha}{\mathalpha}{cyrletters}{"58}

\newtheorem{theoA}{Theorem}

\newtheorem*{coro*}{Corollary}
\newtheorem*{conj*}{Conjecture}
\newtheorem*{lemm*}{Lemma}

\providecommand{\twomat}[4]{\left(\begin{array}{cc}#1&#2\\#3&#4\end{array}\right)}
\providecommand{\smalltwomat}[4]{\left(\begin{smallmatrix}#1&#2\\#3&#4\end{smallmatrix}\right)}

\theoremstyle{definition}

\theoremstyle{remark}
\newtheorem{remark*}{Remark}

\NumberTheoremsIn{subsection}

\numberwithin{equation}{subsection}

\newcommand{\fin}{\mathrm{fin}}
\newcommand{\Ar}{\mathrm{Ar}}
\newcommand{\Div}{\mathrm{Div}}

\newcommand{\divisor}{\mathrm{div}}
\newcommand{\one}{\mathbf{1}}
\newcommand{\Q}{\mathbf{Q}}
\newcommand{\GL}{\mathbf{GL}}
\newcommand{\SO}{\mathbf{SO}}
\newcommand{\GO}{\mathbf{GO}}

\newcommand{\Lf}{l_{f_\alpha}}
\newcommand{\X}{\mathcal{X}}
\newcommand{\calL}{\mathcal{L}}
\newcommand{\hatL}{\widehat{\mathcal{L}}}

\renewcommand{\mod}{\ \mathrm{mod}\,}
\newcommand{\R}{\mathbf{R}}
\newcommand{\Z}{\mathbf{Z}}

\newcommand{\frakp}{\mathfrak{p}}

\newcommand{\frakN}{\mathfrak{N}}
\newcommand{\frakh}{\mathfrak{H}}

\newcommand{\frakd}{\mathfrak{d}}
\newcommand{\frakD}{\mathfrak{D}}

\newcommand{\calA}{\mathcal{A}}
\newcommand{\calG}{\mathcal{G}}

\newcommand{\calR}{\mathcal{R}}
\newcommand{\calS}{\mathcal{S}}
\newcommand{\calM}{\mathcal{M}}

\newcommand{\bcalS}{\baar{\mathcal{S}}}

\newcommand{\C}{\mathbf{C}}

\newcommand{\N}{\mathbf{N}}
\newcommand{\OO}{\mathcal{O}}
\newcommand{\A}{\mathbf{A}}

\newcommand{\bks}{\backslash}
\newcommand{\baar}{\overline}

\newcommand{\eps}{\varepsilon}
\newcommand{\e}{\mathbf{e}}

\newcommand{\wtil}{\widetilde}
\newcommand{\Res}{\mathrm{Res}}
\newcommand{\B}{\mathbf{B}}

\newcommand{\W}{\mathcal{W}}
\newcommand{\ord}{\mathrm{ord}}
\newcommand{\vol}{\mathrm{vol}}
\newcommand{\Tr}{\mathrm{Tr}}

\newcommand{\ex}{\mathbf{e}}

\newcommand{\Gal}{\mathrm{Gal}}

\newcommand{\Pic}{\mathrm{Pic}}

\newcommand{\CH}{\mathrm{CH}}
\newcommand{\Ker}{\mathrm{Ker}\,}

\newcommand{\Hom}{\mathrm{Hom}\,}
\newcommand{\End}{\mathrm{End}\,}
\newcommand{\Lie}{\mathrm{Lie}\,}

\newcommand{\Spec}{\mathrm{Spec}\,}
\newcommand{\can}{\mathrm{can}}
\newcommand{\id}{\mathrm{id}}

\title[$p$-adic Heights of Heegner points  on Shimura curves]{$p$-adic Heights of Heegner points   \\  on Shimura curves}
\author{Daniel Disegni}
\address{ Department of Mathematics and Statistics\\ McGill University\\
805  Shebrooke St. West\\
Montreal, QC H3A 0B9\\Canada
}
\email{daniel.disegni@mcgill.ca}

\begin{document}
\frontmatter

\begin{abstract}
Let $f$ be a primitive  Hilbert modular form of parallel weight~$2$ and level~$N$ for the totally real field $F$, and let $p$ be  a rational prime coprime to $2N$.  If $f$ is ordinary at $p$ and $E$ is a CM extension of $F$ of relative discriminant~$\Delta$ prime to~$Np$, we give an explicit construction of the $p$-adic Rankin--Selberg $L$-function $L_{p}(f_E,\cdot)$.  When the sign of its functional equation is $-1$, we show, under the assumption that all primes $\wp\vert p$ are principal ideals of $\OO_{F}$ which split in $\OO_{E}$,  that its central derivative is given by the $p$-adic  height of a Heegner point on the abelian variety $A$ associated with~$f$.

This $p$-adic Gross--Zagier formula generalises the result obtained by Perrin-Riou when $F=\Q$ and $(N,E)$ satisfies the so-called Heegner condition. We deduce applications to both the $p$-adic and the classical Birch and Swinnerton-Dyer conjectures for~$A$. 
\end{abstract}

\maketitle

\tableofcontents

\section*{Introduction}

In this work we generalise the  $p$-adic analogue of the Gross--Zagier formula of Perrin-Riou \cite{PR} to totally real fields, in a  generality similar to  the work of Zhang \cite{shouwu, asian, II}. We describe here the main result and its applications.

\subsection*{The $p$-adic Rankin--Selberg $L$-function} Let $f$ be a primitive (that is, a normalised new eigenform) Hilbert modular form of parallel  weight $2$, level $N$ and trivial character for the totally real field $F$ of degree $g$ and discriminant $D_{F}$.  Let $p$ be a rational prime coprime to $2N$. Fix embeddings $\iota_\infty$ and $\iota_p$ of the algebraic closure $\baar{\Q}$ of $F$ into $\C$ and $\baar{\Q}_p$ respectively; we let $v$ denote the valuation on $\baar{\Q}_{p}$, normalised by $v(p)=1$. 

Let $E\subset\baar{\Q}$ be a CM (that is, quadratic and purely imaginary) extension of $F$ of relative discriminant $\Delta$ coprime to $D_{F}Np$, let $$\eps=\eps_{E/F}:F_\A^\times/F^\times\to \{\pm 1\}$$
be the associated Hecke character and $\frakN=N_{E/F}$ be the relative norm. If 
$$\W\colon E_\A^\times/E^\times\to\baar{\Q}^\times$$
is a finite order Hecke character\footnote{We will throughout use the same notation for a Hecke character, the associated ideal character, and the associated Galois character.} of conductor $\mathfrak{f}=\mathfrak{f}(\W)$ prime to $N\Delta$, the Rankin--Selberg $L$-function $L(f_{E},\W,s)$ is the entire function defined for $\Re(s)>3/2$ by
$$L(f_{E},\W,s)=L^{N\Delta(\W)}(\eps\W|_{F_{\A}^{\times}}, 2s-1)\sum_m {a(f,m) r_\W(m)\over\N m^s},$$
where $\Delta(\W)=\Delta\frakN( \mathfrak{f})$,  $r_\W(m)= \sum_{\frakN(\mathfrak{a})=m}\W(\mathfrak{a})$ (the sum running over all nonzero ideals of $\OO_E$) and
$$L^{N\Delta(\W)}(\eps\W|_{\OO_F}, s)=\sum_{(m,  N\Delta(\W))=1}\eps(m)\W(m)\N m^{-s}.$$

This $L$-function admits a $p$-adic analogue (\S\ref{sec:L}). Let $E_{\infty}'$ be the maximal abelian extension of $E$ unramified outside $p$, and $\calG'=\Gal(E'_{\infty}/E)$.\footnote{The reason for the notation is that later in the paper we will denote by $E_{\infty}$ 
the maximal $\Z_{p}$-subextension of $E_{\infty}'$.}
(It has rank $1+\delta+g$ over $\Z_{p}$, where $\delta$ is the Leopoldt defect of $F$.) 
For each prime $\wp$ of $\OO_F$ dividing $p$, let 
$$P_{\wp,f}(X)=X^2-a(f,\wp) X+\N \wp$$
be the  $\wp^{\mathrm{th}}$ Hecke polynomial of $f$, and assume that $v(\iota_{p}(a(f, \wp)))=0$; in this case $f$ is said to be  \emph{ordinary}, and  there is a unique root $\alpha_{\wp}\in \baar{Q}$ of $P_{\wp, f}(X)$  such that $\iota_{p}(\alpha_{\wp})$ is a $p$-adic unit.  Let $L\subset\baar{\Q}_{p}$ be the finite  extension of $\Q_{p}$ generated by the Fourier coefficients $a(f, m)$ of $f$ and by the $\alpha_{\wp}$ for $\wp\vert p$.

\begin{theoA}\label{theoremA} There exists a unique element  $L_{p}(f_{E})$ of $\OO_{L}\llbracket \calG'\rrbracket\otimes_{\OO_{L}} L$
satisfying the interpolation property
$$L_{p}(f_E)(\W)= \frac{ \W(d_{F}^{(p)})\tau(\baar{\W})   \N(\Delta(\W))^{1/2}   V_{p}(f,\W) \baar{\W}(\Delta)}
{  \alpha_{{\frak f}}  \Omega_{f}} 
 L(f_{E} , \baar{\W},1),
$$
for all  finite order characters $\W$  of $\calG'$ of conductor ${\frak f}(\W)$. Here both sides are algebraic numbers,\footnote{By a well-known theorem of Shimura \cite{shimura}. They are compared via $\iota_{p}^{-1}$ and $\iota_{\infty}^{-1}$.}  $\baar{\W}=\W^{-1}$ and
\begin{align*}
\Omega_{f}=(8\pi^{2})^{g}\langle f,f \rangle_{N}
\end{align*}
with $\langle\ , \ \rangle_{N}$ the Petersson inner product \eqref{pet}; $\tau(\baar{\W})$ is a normalized Gau\ss\ sum; $V_{p}(f,\W)$ is a product of partial Euler factors at $p$; and finally $\alpha_{{\frak f}}=\prod_{\wp\vert p}\alpha_{\wp}^{v_{\wp}(\frakN({\frak f}))}$.
\end{theoA}

This is essentially a special case of results of Panchishkin \cite{Pa} and Hida \cite{Hi}; we reprove it entirely here (see \S\ref{sec:L}, especially Theorem \ref{theo:interpolate}) because the precise construction of $L_{p}(f_E)$ will be crucial for us. It is obtained, using a technique of Hida and  Perrin-Riou,  by applying a $p$-adic analogue of the functional ``Petersson product with $f$'' to a  convolution $\Phi$ of Eisenstein and theta measures  on $\mathcal{G}'$ valued in $p$-adic modular forms (so that $\Phi=\Phi(\W)$ is an analogue of the kernel of the classical Rankin--Selberg convolution).  The approach we follow is adelic; one novelty introduced here is that the theta measure is constructed via the Weil representation, which seems very natural and would generalise well to higher rank cases.

\medskip

On the other hand,  Manin \cite{maninjl}, Dimitrov \cite{dimitrov} and others
have constructed a $p$-adic $L$-function $L_{p}(f,\cdot)\in \OO_{L}\llbracket\calG_{F}\rrbracket$ as an analogue of the standard $L$-function $L(f,s)$, where $\calG'_{F}$ is the Galois group of the maximal abelian extension of $F$ unramified outside $p$;  it is characteriesd by  the interpolation property
$$L_{p}(f,\chi)=\chi(d_{F}^{(p)}){\tau(\baar{\chi})\N(\mathfrak{f}(\chi))^{1/2}\over \alpha_{{\frak f}(\chi)}}
{L(f,\baar{\chi},1)\over \Omega_{f}^{+} }$$
for all  finite order characters $\chi$ of conductor ${\frak f} (\chi)$ which are trivial at infinity and ramified at all primes $v\vert p$. (Here $\Omega_{f}^{+}$ is a suitable period, cf. \S\ref{sec:realperiods}, and $\tau(\chi)$ is again a  normalised Gau\ss\ sum.) The corresponding formula for complex $L$-functions implies a factorisation  \eqref{factoris2}
\begin{gather}\label{factoris}
L_{p}(f_{E},\chi\circ\frakN)
=  \chi(\Delta)^{2}{\Omega_{f}^{+}\Omega_{{f}_{\eps}}^{+}\over D_{E}^{-1/2}\Omega_{f}} 
   L_{p}(f,\chi)L_{p}(f_{\eps},\chi),\end{gather}where $f_{\eps}$ is the form with coefficients $a(f_{\eps},m)=\eps(m)a(f,m)$ and $D_{E}=\N(\Delta)$.

\subsection*{Heegner points on Shimura curves and the main theorem} Suppose that $\eps(N)=(-1)^{g-1}$, where $g=[F:\Q]$. Then for each embedding $\tau\colon  F\to \C$, there is a quaternion algebra $B(\tau)$ over $F$ ramifed exactly at the finite places $v\vert N$ for which $\eps(N_{v})=-1$ and the infinite places different from $\tau$; it admits an embedding $\rho\colon E\hookrightarrow B(\tau)$, and  we can consider an order $R$ of $B(\tau)$ of discriminant $N$ and containing $\rho(\OO_E)$. These data define a \emph{Shimura curve} $X$. It is an algebraic curve over~$F$, whose complex points for any embedding $\tau\colon  F\to \C$ are described by 
$$X(\C_\tau)=B(\tau)^\times\bks \frakh^\pm\times \widehat{B}(\tau)^\times/\widehat{F}^\times\widehat{R}^\times \cup \mathrm{\{cusps\}}.$$
It plays the role of the modular curve $X_0(N)$ in the works of Gross--Zagier \cite{GZ} and Perrin-Riou \cite{PR} who consider the case $F=\Q$ and $\eps(v)=1$ for all $v\vert N$ (it is only in this case that the set of cusps is not empty).

The curve $X$ is connected but not geometrically connected. Let $J(X)$ be its Albanese ($\cong$ Jacobian) variety; it is an abelian variety defined over $F$, geometrically isomorphic to the product of the Albanese varieties of the geometrically connected components of $X$. There is a natural  map $\iota\colon X\to J(X)\otimes \Q$ given by $\iota(x)=[x]-[\xi]$, where $[\xi]\in\mathrm{Cl}(X)\otimes\Q$ is a canonical divisor class constructed in \cite{shouwu} having degree~$1$ in every geometrically connected component of $X$; an integer multiple of $\iota$ gives a morphism $X\to J(X)$ defined over $F$.

As in the modular curve case, the curve $X$ admits a finite collection of \emph{Heegner points} defined over the Hilbert class field $H$ of $E$ and permuted simply transitively by $\Gal(H/E)$. They are the points represented by $({x_0},t)$  for $t\in \widehat{E}^\times/E^\times\widehat{F}^\times\widehat{\OO}_E^\times$ when we use the complex description above and view $E\subset B$ via $\rho$.  We let $y$ be any such Heegner point, and let $[z]$ denote the class
$$[z]=u^{-1}\iota \left(\Tr_{H/E}y\right)\in J(X)(E)\otimes\Q,$$
where $u=[\OO_E^\times:\OO_F^\times]$.

As a consequence of Jacquet--Langlands theory, the Hecke algebra on Hilbert modular forms of level $N$ acts  through its quaternionic quotient on $J(X)$. Let $z_f\in J(X)(E)\otimes\baar{\Q}$ be the $f$-component of $[z]$.

\paragraph{Heights and the formula} On any curve~$X$ over a number field~$E$, there is a notion  (\S\ref{sec:height pairing})  of $p$-adic height  $\langle\ , \ \rangle_\ell$ attached to the auxiliary choices of splittings of the Hodge filtrations on $H^{1}_{\rm dR}(X/E_{w})$ for $w \vert p$ and of  a \emph{$p$-adic logarithm} $\ell\colon E_\A^\times/E^\times\to \Q_p$. It is a symmetric bilinear pairing on the group of degree zero divisors on $X$ modulo rational equivalence, which we can view as a pairing on $J(X)(E)$. More generally, for any abelian variety $A/E$ there is defined a $p$-adic height pairing on $A(E)\times A^{\vee}(E)$. In our case, there is a canonical choice for the Hodge splittings on the $f$-components of the Albanese variety $J(X)$, given by the unit root subspaces, and we choose our height pairing on $J(X)$ to be compatible with this choice.

 Under the assumption $\eps(N)=(-1)^{g-1}$, the value $L_{p}(f_E, \one)$ is zero by  the complex functional equation and the interpolation property; in fact, we have more generally $L_{p}(f_{E}, \W)=0$ for 
any anticyclotomic character $\W$ of $\calG$. We can then consider its derivative in a cyclotomic direction. Let thus $\W$ be a Hecke character of $E$ induced from a Hecke character of $F$ taking values in $1+p\Z_{p}\subset \Z_{p}^{\times}$, and assume $\W$ is ramified at all places dividing~$p$.   The derivative of $L_{p}(f_E)$ in the $\W$-direction is
$$L_{p,\W}'(f_E,\one)=\left.{d\over ds}\right\rvert_{s=0} L_{p}(f_E)(\W^s).$$
\begin{theoA}\label{theoremB} 
Assume that $\Delta_{E/F}$ is totally odd and that every prime $\wp\vert p$ is a principal ideal in $\OO_{F}$ and splits in $\OO_{E}$. Suppose that $\eps_{E/F}(N)=(-1)^{g-1}$.  Then $L_{p}(f_{E},\one)=0$ and 
$$L_{p,\W}'(f_E,\one)= D_{F}^{-2}\prod_{\wp\vert p} \left(1-{1\over\alpha_\wp}\right)^{2}
 \left(1-{1\over \eps(\wp)\alpha_\wp}\right)^{2} \langle z_f, z_f\rangle_{\W}$$
where $\langle \ , \ \rangle_{\W}$ is the height pairing on $J(X)(E)$ associated with the logarithm $\ell=\left.{d\over ds}\right\rvert_{s=0}\W^{s}$.
\end{theoA}

The hypothesis that the primes $\wp\vert p$ are principal is a technical assumption which intervenes only in Proposition  \ref{psifinetc}.\footnote{A somewhat more sophisticated approach to our main result should remove this and other restrictions \cite{pyzz}.}  The assumption that they split in $E$ is essential to the argument but, like the  assumption on $\Delta_{E/F}$, it can be removed \emph{a posteriori} if the left-hand side of the formula below  is nonzero  -- see \S\ref{kobtrick}.

\subsection*{Applications to the conjecture of Birch and Swinnerton-Dyer} 
It is conjectured that to any Hilbert modular newform $f$ one can attach a simple abelian variety $A=A_{f}$ over $F$, characterised uniquely up to isogeny\footnote{Thanks to Faltings's isogeny theorem \cite{faltings}.} by the equality of $L$-functions 
\begin{align*}
L(A,s)=\prod_{\sigma\colon M_{f}\to\C}L(f^{\sigma},s).
\end{align*}
 Here $M=M_{f}$ is the field generated by the Fourier coefficients of $f$; $A$ has dimension $[M:\Q]$ and its endomorphism algebra contains $M$ (we say that $A$ is \emph{ of $GL_{2}(M)$-type}; in fact since $F$ is totally real, $A$ is of \emph{strict} $GL_{2}$-type, that is, its endomorphism algebra {equals} $M$ -- see e.g. \cite[Lemma 3.3]{yzz}). The conjecture is known to be true \cite[Theorem B]{shouwu} when either $[F:\Q]$ is odd or $v(N)$ is odd for some  finite place $v$ (the assumptions of Theorem B above imply that one of these conditions holds); in this case $A$ is a quotient $\phi$ of $J(X)$ for a suitable Shimura curve $X$ of the type described above. Viceversa any abelian variety of $GL_{2}$-type (for some field $M$) over a totally real field $F$ is conjectured to be associated with a Hilbert modular form $f$ as above. This is known to be true for all elliptic curves $A$ over $F$ when $F$ is $\Q$  or a  real quadratic field; and for all but possibly finitely many geometric isomorphism classes if $F$ is a general totally real field (see \cite{bao}, whose result is somewhat stronger result than this, and \cite{six}; the results build on the method of Wiles for $F=\Q$).

In view of known ${\rm Aut}(\C/\Q)$-equivariance properties  of automorphic $L$-functions and the above equality, the order of vanishing of $L(A,s)$ at $s=1$ will be an integer multiple $r[M:\Q]$ of the dimension of $A$. We call $r$ the \emph{$M$-order of vanishing} of $L(A,s)$ or the \emph{analytic $M$-rank} of $A$.

\begin{conj*}[Birch and Swinnerton-Dyer] Let $A$ be an abelian variety of $GL_{2}(M)$-type over a totally real field $F$ of degree $g$. 
\begin{enumerate}
\item The $M$-order of vanishing of $L(A,s)$ at $s=1$ is equal to the dimension of $A(F)_{\Q}$ as $M$-vector space.
\item The Tate-Shafarevich group $\Sha(A/F)$ is finite, and the leading term of $L(A,s)$ at $s=1$ is given by 
$${L^*(A,1)\over \Omega_{A}}={D_{F}^{-d/2} }|\Sha(A/F)| R_{A}\prod_{v\nmid\infty }c_{v} ={\rm BSD}\,(A),$$
where $d=\dim A=[M:\Q]$, the $c_{v}$ are the  Tamagawa numbers of $A$ at finite places (almost all equal to~$1$),
$$\Omega_{A}=\prod_{\tau\colon F\to \R} \int_{A(\R_{\tau})}|\omega_{A}|_{\tau}$$
for a N\'eron differential\footnote{When it exists, which is only guaranteed if $F=\Q$. Otherwise, we take for $\omega_{A}$ any generator of $H^{0}(A, \Omega^{d}_{A/F})$ and to define $\Omega_{A}$ we divide by the product of the indices $[H^{0}(\mathcal{A}_{v}, \Omega^{d}_{{\mathcal A}_{v}/\OO_{F,v}}):\OO_{F,v}\wtil{\omega_{A}}]$ of (the extension of) $\omega_{A}$ in the space of top differentials on the local N\'eron models ${\mathcal A}_{v}/\OO_{F,v}$ of $A$.} 
 $\omega_{A}$, and 
$$R_{A}={\det(\langle x_{i},y_{j}\rangle)\over [A(F):\sum \Z x_{i}][A^{\vee}(F):\sum \Z y_{j}]}$$
is the regulator of the N\'eron-Tate height paring on $A(F)\times A^{\vee}(F)$, defined using any subsets $\{x_{i}\}$, $\{y_{j}\}$ of $A(F)$, $A^{\vee}(F)$ inducing bases of $A(F)_{\Q}$ and $A^{\vee}(F)_{\Q}$.
\end{enumerate}
\end{conj*}

\medskip

By the automorphic description of $L(A,s)$ and results of Shimura \cite{shimura}, we know that $L(A,s)/\prod_{\sigma\colon M_{f}\to\C}\Omega_{f^{\sigma}}^{+}$ is an algebraic number. Comparison with the Birch and Swinnerton-Dyer conjecture suggests the following conjecture.

\begin{conj*}[Period Conjecture] We have
$$\Omega_{A}\sim\prod_{\sigma\colon M_{f}\to\C}\Omega_{f^{\sigma}}^{+}\qquad\textrm{in }\C^{\times}/\baar{\Q}^{\times}.$$
\end{conj*}

The conjecture is known for $F=\Q$ \cite{shimura31} or when $A$ has complex multiplication (over $\baar{\Q}$) \cite{blasius}; see \S\ref{sec:periods} below for a more precise conjecture and some further evidence and motivation.

 Assuming the conjecture,  we can define  a $p$-adic $L$-function $L_{p}(A)$ for~$A$ by 
$$L_{p}(A)={ \prod_{\sigma}\Omega_{f^{\sigma}}^{+} \over \Omega_{A}}     \prod_{\sigma\colon M_{f}\to\C} L_{p}(f^{\sigma}) $$
for any prime $p$ such that $A$ has good ordinary reduction at all primes above $p$.

 Then, fixing a ramified Hecke character $\nu\colon \calG_{F}'\to 1+p\Z_{p}\subset \Z_{p}^{\times}$ which we omit from the notation,  one can formulate a $p$-adic version of the Birch and Swinnerton-Dyer conjecture similarly as above for $L_{p}(A,\nu^{s})$:\footnote{Here $s\in\Z_{p}$ and the central point is $s=0$, corresponding to $\nu^{0}=\one$.} the conjectural formula reads
$$\prod_{\wp|p}(1- \alpha_{\wp}^{-1})^{-2}L_{p}^{*}(A,\one)={\rm BSD}_{p}(A)$$
where ${\rm BSD}_{p}(A)$ differs from  ${\rm BSD}(A)$ only in the regulator term, which is now the regulator of the $p$-adic height pairing on $A(F)\times A^{\vee}(F)$ associated with the $p$-adic logarithm $\ell$ deduced from $\nu$ as in Theorem \ref{theoremB}.
One can also formulate a main conjecture of Iwasawa theory for $L_{p}(A)$, see \cite{schneider}.

\medskip

Then, just as in \cite{PR}, we can deduce the following arithmetic application of Theorem \ref{theoremB}.

\begin{theoA}\label{theoremC} Assume that the Period conjecture holds for   the abelian varitey $A=A_f$. For an ordinary prime $p>2$ decomposing into principal prime ideals in $\OO_{F}$, we have:
\begin{enumerate}
\item The following are equivalent:
\begin{enumerate}
\item\label{prank1} The $p$-adic $L$-function $L_{p}(A,\nu^{s})$ has $M_{f}$-order of vanishing $r\leq 1$ at the central point.
\item\label{rank1} The complex  $L$-function  $L(A,s)$ has $M_{f}$-order of vanishing $r\leq 1$ at the central point and the $p$-adic height pairing associated with $\nu$ is non-vanishing on $A(F)$.
\end{enumerate}
\item If either of the above assumptions holds, the rank parts  of the classical and the $p$-adic Birch and Swinnerton-Dyer conjecture are true for $A$ and the Tate-Shafarevich group of $A$ is finite.
\item  If moreover the cyclotomic Iwasawa main  conjecture is true for $A$, then the classical and the $p$-adic Birch and Swinnerton-Dyer formulas for $A$ are true up to a $p$-adic unit.
\end{enumerate}
\end{theoA}
\begin{proof} In part 1, the statement follows trivially from the construction of $L_p (A)$ if $r=0$; if $r=1$, both conditions are equivalent to the assertion that for a suitable CM extension $E$, the Heegner point $z_f=z_{f,E}$ is nontorsion: this is obvious from our main theorem in case \ref{prank1}; in case \ref{rank1}, by the work of Zhang \cite{shouwu, asian}  (generalising Gross--Zagier \cite{GZ} and Kolyvagin \cite{koly, koly-log}),
the Heegner point
 $$P=\sum_{\sigma}{\rm Tr}_{E/F}\phi(z_{f^{\sigma},E})\in A(F)\otimes\Q$$
(with  $\phi\colon J(X)\to A$) generates $A(F)\otimes\Q$ as $M_{f}$-vector space, so that the $p$-adic height pairing on $A(F)$ is non-vanishing if and only if it is nonzero at $z_f$. Part 2 then follows from 1 and the results of Zhang \cite{shouwu, asian}. 

 Schneider \cite{schneider} proves an ``arithmetic'' version of the $p$-adic Birch and Swinnerton-Dyer formula for (the Iwasawa $L$-function associated with) $A$, which under the assumption of 3 can be compared to the analytic $p$-adic formula as explained in \cite{PR} to deduce  the $p$-adic Birch and Swinnerton-Dyer formula up to a $p$-adic unit.  In the analytic rank 0 case the classical Birch and Swinnerton-Dyer formula follows immediately. In the case of analytic rank~$1$, recall that the main result of \cite{shouwu, II} is, in our normalisation, the formula
$${L'(f_E,1)\over \Omega_{f}}={1\over D_{F}^{2} D_{E}^{1/2}}\langle z_f,z_f\rangle= :D_{E}^{-1/2} {\rm GZ}(f_{E})$$
(where $\langle\ , \ \rangle$ denotes the N\'eron--Tate height); whereas we introduce the notation ${\rm GZ}_{p}(f_{E})$ to write our formula (for any fixed ramified cyclotomic character $\W=\nu\circ\frakN$) as 
$$L'_{p}(f_{E}, \one)=\prod_{\wp\vert p} \left(1-{1\over\alpha_\wp}\right)^{2}
 \left(1-{1\over \eps(\wp)\alpha_\wp}\right)^{2}{\rm GZ}_{p}(f_{E}).$$
Then, after choosing $E$ suitably so that $L
(f_\eps,1)\neq0$ (which can be done by \cite{bfh}, \cite{wald}),  we can argue as in \cite{PR} to compare the $p$-adic and the complex  Birch and Swinnerton-Dyer formulas via the corresponding Gross--Zagier formulas to get the result. Namely, 
 we have
\begin{align*} {L^{*}(A,1)\over \Omega_{A} {\rm BSD} (A)}
&={\prod_{\sigma} \Omega_{f^{\sigma}}^{+}  \over \Omega_{A}}
{1\over {\rm BSD}(A)} 
\prod_{\sigma}
{ L'(f^{\sigma}_{E}, 1)\over \Omega_{f^{\sigma}}} 
{\Omega_{f^{\sigma}}\over \Omega_{f^{\sigma}}^{+}\Omega_{f^{\sigma}_{\eps}}^{+}}
{\Omega_{f_{\eps}^{\sigma}}^{+}\over L(f_{\eps}^{\sigma},1)}  \\
&={\prod_{\sigma} \Omega_{f^{\sigma}}^{+}  \over \Omega_{A}}
{\prod_{\sigma}{\rm GZ}(f_{E}^{\sigma}) \over {\rm BSD}(A)} 
\prod_{\sigma}
{D_{E}^{-1/2}\Omega_{f^{\sigma}}\over \Omega_{f^{\sigma}}^{+}\Omega_{f^{\sigma}_{\eps}}^{+}}
{\Omega_{f_{\eps}^{\sigma}}^{+}\over L(f_{\eps}^{\sigma},1)}
\end{align*}
by the complex Gross--Zagier formula and the factorisation of $L(f_{E},s)$. 
Similarly, 
\begin{align*} \prod_{\wp|p}(1- \alpha_{\wp}^{-1})^{-2} {L^{*}_{p}(A,1)\over {\rm BSD}_{p} (A)}
&={\prod_{\sigma} \Omega_{f^{\sigma}}^{+}  \over \Omega_{A}}
{\prod_{\sigma}{\rm GZ}_{p}(f_{E}^{\sigma}) \over {\rm BSD}_{p}(A)} 
\prod_{\sigma}
{D_{E}^{-1/2}\Omega_{f^{\sigma}}\over \Omega_{f^{\sigma}}^{+}\Omega_{f^{\sigma}_{\eps}}^{+}}
{\Omega_{f_{\eps}^{\sigma}}^{+}\over L(f_{\eps}^{\sigma},1)} 
\end{align*}
by the $p$-adic Gross--Zagier formula, the factorisation of $L_{p}(f_{E})$ and the interpolation property of $L_{p}(f_{\eps})$. Since we are assuming to know that the left-hand side of the last formula is a $p$-adic unit, the result follows from observing the equality
$${\prod_{\sigma}{\rm GZ}(f_{E}^{\sigma}) \over {\rm BSD}(A)} ={\prod_{\sigma}{\rm GZ}_{p}(f_{E}^{\sigma}) \over {\rm BSD}_{p}(A)}$$
 of rational numbers.\footnote{The rationality of the ratios follows from the fact that the $z_{f^{\sigma}}$ 
 essentially belong to $J(X)(F)$ -- that is, they belong to  the $+1$-eigenspace for the action of $\Gal(E/F)$ on $J(X)(E)\otimes\baar{\Q}$ -- and that in this sense, their images $\phi(z_{f^{\sigma}})$ form a $\Gal(\baar{\Q}/\Q)$-invariant basis of $A(F)\otimes\baar{\Q}$, orthogonal for the height pairing.}
\end{proof}

Alternative approaches to the Birch and Swinnerton-Dyer formula in rank one have recently been proposed, at least for the case $F=\Q$, by Wei Zhang \cite{weikol} and Xin Wan.
  
\begin{enonce*}[remark]{Discussion of the assumptions} The conjecture on periods could be dispensed of if one were willing to work with  a ``wrong'' $p$-adic $L$-function for $A$ (namely, one without the period ratio appearing in the definition above). Then at least the rank part of the $p$-adic Birch and Swinnerton-Dyer conjecture makes sense and parts 1 and 2 of the Theorem hold.  The nonvanishing of the $p$-adic height pairing is only known for CM elliptic curves \cite{bertrand}. The Iwasawa main conjecture is known in most cases for ordinary elliptic curves over $\Q$ thanks to the work of Rubin, Kato and Skinner--Urban (see \cite{SU}). For Hilbert modular forms,  one divisibility in the CM case is proved by Hsieh \cite{hsieh}; results on the general case are obtained by Wan \cite{xin wan}. We can then record the following unconditional result, whose assumptions are inherited from \cite{hsieh}.
\end{enonce*}

\begin{theoA} Let $A/F$ be an elliptic curve  with complex multiplication by the ring of integers $\OO_{K}$ of an imaginary quadratic field $K$. Let $K'=FK$ and let $h_{K'}^{-}=h_{K'}/h_{K}$ be the relative class number of $K'/F$. Let $p\nmid 6h_{K'}^{-}D_{F}$ be a prime such that for all primes $\wp\vert p$,  $\wp$ is principal and $A$ has good ordinary reduction at $\wp$. Suppose that ${\rm ord}_{s=1}L(A,s)\leq 1$. Then 
$$v_{p}\left({L^{*}(A,1)\over R_{A} \Omega_{A}}\right)\leq v_{p}\left(|\Sha(A/F)| \prod_{v\nmid\infty }c_{v}\right).$$
\end{theoA} 

Results toward the divisibility  in the opposite direction can be obtained from the method of Kolyvagin, cf. \cite{kolystruct} (for $F=\Q$) and \cite{howgl2} (for general $F$ but excluding the CM case). 

\subsection*{Plan of the proof} The proof of the main formula follows the strategy of Perrin-Riou \cite{PR}. It is enough (see \S\ref{proofmt}) to study the case where $\W$ is cyclotomic ($\W=\W^{c}$), since both sides of the formula are zero when $\W$ is anticyclotomic ($\W\W^{c}=\one$).

 In the first part of this paper, we construct the  measure  $\Phi$ on  $\mathcal{G}$  valued in $p$-adic modular forms such that $L_{p}(f_E)(\W)$ essentially equals $\Lf(\Phi(\W))$, where  $\Lf$ is a $p$-adic analogue of the functional ``Petersson product with $f$'' on $p$-adic modular forms. This  allows us to write
 $$L_{p, \W}'(f_{E},\one)\doteq \Lf(\Phi'_{\W}),$$
 where $\doteq$ denotes equality uo to suitable nonzero factors, and $\Phi_{\W}'={d\over ds}\Phi(\W^{s})|_{s=0}$ is a $p$-adic Hilbert modular form.
 
On the other hand,  there is a modular form $\Psi$ with Fourier coefficients given by $\langle z, T(m)z\rangle_{\W}$, so that  $\Lf(\Psi)\doteq \langle z_f,z_f \rangle_{\W}$. It can be  essentially written as a sum of 
 modular forms $\Psi_\fin+\Psi_{ p}$, where $\Psi_\fin$  encodes the local contributions to the height from places  not dividing $p$ and $\Psi_{p}=\sum\Psi_{\wp}$ the contribution from the places $\wp$ above $p$. Then we can show by explicit computation  that the Fourier coefficients of $\Phi'$ are equal to the Fourier coefficients of $\Psi_\fin$ up to the action of suitable Hecke operators at $p$.  The desired formula then follows once we show that $\Lf(\Psi_{p})$ vanishes. To prove this we examine the effect of the operator $U_{\wp}$ on $\Psi_{\wp}$, and find that, in a suitable quotient space, the ordinary projection of $\Psi_{\wp}$ is zero. The study of $\Psi_{\wp}$ follows the methods of Perrin-Riou.

\medskip
One difficulty in the approach just outlined is that compared to the case of modular curves there are no cusps available, so that in this case the divisors $z$ and $T(m)z$ have intersecting supports and the decomposition of the height pairing into a sum of local pairings is not available. Our solution to this problem, which is inspired from the work of Zhang \cite{shouwu}, is to make use  of $p$-adic Arakelov theory as developed by Besser \cite{besser} (see \S \ref{sec:arakelov}) and work consistently in a suitable quotiented space of Fourier coefficients.

\subsection*{Perspective} The original Gross--Zagier formula has undergone an impressive transformation since its first appearance in 1986, culminating in the recent book of Yuan--Zhang--Zhang \cite{yzz}. Obviously, this work is only a first step in catching up on the $p$-adic side.\footnote{For a more accomplished attempt, see  \cite{pyzz}.} The latter has also seen important developments, with generalisations to the non-ordinary case, by Kobayashi \cite{kobayashi}, and to the case of  higher weights,  by   {Nekov{\'a}{\v{r}} \cite{nekovar} and Shnidman \cite{ari}.  It would certainly be of interest to generalise those results to the setting of the present work.\footnote{Results in the non-ordinary case were presented in a preliminary version of this paper, assuming a suitable construction of $p$-adic $L$-functions generalising \cite{nearly};  I hope to present them  in revised form in a future work.} 

Results similar to those presented here were recently obtained in the thesis of Li Ma (Paris 6).

\subsection*{Acknowledgments} The present paper grew out of  my Columbia thesis. I am grateful to my advisor Professor Shou-Wu Zhang for suggesting this area of research and for his support and encouragement; and to Amnon Besser, Shinichi 
Kobayashi, Luis Garcia Martinez, David Loeffler, Yifeng Liu, Giovanni Rosso,  Eric Urban, Jeanine Van Order and Shou-Wu Zhang for useful conversations or correspondence.

This work builds upon the works of Perrin-Riou \cite{PR} and Zhang \cite{shouwu, asian, II} -- and, of course,  Gross--Zagier \cite{GZ}. My debt to their ideas cannot be overstated and will be obvious to the  reader.

Some  revisions to the manuscript were done while the author was a postdoctoral fellow at MSRI funded under NSF grant 0932078000.

\subsection*{Notation}
Throughout this text we use the following notation and assumptions, unless otherwise noted:

\begin{itemize}
\item $F$ is a totally real field of degree $g$;
\item $\N_{F}$ is the monoid of nonzero ideals of $\OO_{F}$;
\item $|\cdot|_{v}$ is the standard absolute value on $F_{v}$;
\item $\A=\A_F$ is the ad\`{e}le ring of $F$;  if $*$  is a place or a set of places or an ideal of $F$,  the component at $*$ (respectively away from $*$) of an adelic object  $x$ is denoted $x_*$ (respectively $x^*$). For example if $\phi=\prod_{v}\phi_{v}$ is a Hecke character and $\delta$ is an ideal of $\OO_{F}$ we write $\phi_{\delta}(y)=\prod_{v\vert \delta}\phi_{v}(y_{v})$, and $|y|_{\delta}=\prod_{v\vert \delta}|y|_{v}$. We also use the notation 
$$|m|_{v}=|\pi_{m}|_{v},\ |m|_{\delta}=|\pi_{m}|_{\delta}, \quad \phi_{v}(m)=\phi_{v}(\pi_{m}),  \ \phi_{\delta}(m)=\phi_{\delta}(\pi_{m})$$
if $m$ is an ideal of $\OO_{F}$ and $\phi$ is unramified at $\delta$ (here $\pi_{m}$ satisfies $\pi_{m}\OO_{F}=m$).
\item ``$>$''  denotes the partial order on $\A_{F}$ given by $x>0$ if and only if $x_{\infty}$ is totally positive;
\item $R_\A=R\otimes_F\A$ if $R$ is an $F$-algebra;
\item $\N m$ is the absolute norm of an ideal $m$ in a number field (the index of $m$ in the ring of integers: it is a positive natural number); 
\item $d_F$ is the different of $F$;
\item $\pi_{N}$, for $N$ an ideal of $\OO_{F}$, is the id\`ele with components $\pi_{v}^{v(N)}$ for $v\nmid\infty$ and $1$ for $v\vert \infty$.
\item $D_F=\N d_F$ is the discriminant of $F$. 
\item  $m^\times=\{a\in F_\A^\times\, |\, a\OO_F=m\}$ if $m$ is any nonzero fractional ideal of $F$ (this notation will be used with $m=d_F^{-1}$).
\end{itemize}

\begin{itemize}
\item $E$ is a quadratic CM  (that is, totally imaginary) extension of $F$;
\item $\frakD=\frakD_{E/F}$ is the different of $E/F$.
\item $\frakN=N_{E/F}$ is the relaltive norm on $E$ or any $E$-algebra; 
\item $\Delta=\Delta_{E/F}=\frakN(\frakD)$ is the relative discriminant of $E/F$ and we assume $$(\Delta_{E/F}, D_{F}Np)=1;$$
in \S\S\ref{2.5}, \ref{4.5} and part of \S\ref{3.2}  we further assume that
$$(\Delta, 2)=1$$
and in \S\S\ \ref{7.2}, \ref{sec:hp2}, \ref{8.1}, that 
$$(\Delta, 2)=1 \textrm{ and all primes $\wp$ dividing $p$ are split in $E$.}$$
\item $D_{E}=\N(\Delta)$ is the absolute discriminant of $E$.
\end{itemize}

\begin{itemize}
\item $U_F(N)$ is the subgroup of $\widehat{\OO}_F^\times =\prod_v \OO_{ F,v}^\times\subset F_{\A^\infty}^\times$ consisiting of elements $x\equiv 1\mod N\widehat{\OO}_F$, if $N$ is any ideal of $\OO_F$; 
\item $\ex_v(x)=\exp(-2\pi i \{\Tr_{F_v/\Q_p}(x)\}_p)$ for $v\vert p<\infty$ and $\{y\}_p$ the $p$-fractional part of $y\in\Q_p$ is the standard addtive character of $F_v$, with conductor $d_{F,v}^{-1}$; for  $v\vert \infty$, $\ex_v(x)=\exp(2\pi i \Tr_{F_v/\R}(x))$;
\item $\ex(x)=\prod_v\ex_v(x_v)$ is the standard additive character of $\A_F$.
\end{itemize}

\begin{itemize}
\item $\one_Y$ is the characteristic function of the set $Y$;
\item if $\varphi$ is any logical proposition, we define $\one[\varphi]$ to be $1$ when $\varphi$ is true and $0$ when $\varphi$ is false -- e.g. $\one[x\in Y]=\one_Y(x)$.
\end{itemize}

\mainmatter

\part{$p$-adic $L$-function and measures}

This part is dedicated to the construction of the measure giving the $p$-adic Rankin--Selberg $L$-function $L_{p}(f_E)$ and to the computation of its Fourier coefficients.

\section{$p$-adic modular forms}

\subsection{Hilbert modular forms}\label{sec:hmf}
Let us define compact subgroups of $\GL_2(\A^\infty)$ as follows:
\begin{itemize}
\item $K_0(N)=\left\{\begin{pmatrix}a & b\\c&d\end{pmatrix}\in \GL_2(\widehat{\OO}_F)\ |\  c\equiv 0 \mod N\widehat{\OO}_F\right\}$ if $N$ is an ideal of $\OO_F$;
\item $K_1(N)=\left\{\begin{pmatrix}a & b\\c&d\end{pmatrix}\in K_0(N)\ |\ a\equiv 1 \mod N\widehat{\OO}_F \right\}$.
\end{itemize}
Let $k$ be an element of $\Z_{\geq 0}^{\Hom(F, \baar{\Q})}$ and $\psi$ be a character of $F_\A^\times/F^{\times}$ of conductor dividing $N$ satisfying $\psi_v(-1)=(-1)^{k_{v}}$ for $v\vert \infty$.  
A \textbf{Hilbert modular form} of  weight $k$, level $K_1(N)$ and character $\psi$ is a smooth function $$f\colon \GL_2(F)\bks \GL_2(\A_F)\to \C$$ of \emph{moderate growth}\footnote{That is, for every $g$ the function  $\A^{\times}\ni y\mapsto f\left(\twomat y{}{}1 g\right)$ grows at most polynomially in $|y|$ as $|y|\to\infty$.}  
satisfying\footnote{Recall the notation $\psi_{N}=\prod_{v\vert N}\psi_{v}$.} 
$$f\left(\begin{pmatrix}z&\\&z \end{pmatrix} g \begin{pmatrix}a&b\\c&d \end{pmatrix} r(\theta)\right)=\psi(z)\psi_{N}(a) \ex_\infty(k\cdot \theta)f(g)$$
for each $z\in F_{\A}^\times$, $\begin{pmatrix} a&b\\c&d\end{pmatrix}\in K_0(N)$ and $\theta=(\theta_v)_{v\vert \infty}\in F_\infty$,
with $r(\theta)=\prod_{v\vert \infty} r(\theta_v)$ and
$r(\theta_v)=\begin{pmatrix} \cos\theta_v&\sin\theta_v\\-\sin\theta_v&\cos\theta_v\end{pmatrix} \in \SO_2(F_v)$. If $k $ is constant we say that $f$ has parallel weight; in this work we will be almost exclusively concerned with forms of parallel weight, and we will assume that we are in this situation for the rest of this section.

We call $f$  holomorphic if for each $x^{\infty}\in \A^{\infty}$, $y^{\infty}\in F_{\A^{\infty}}^{\times}$, the function on  $\frakh^{\Hom(F,\baar{\Q})}=\{x_\infty+iy_{\infty}\in F\otimes\C\,\ |\, y_{\infty}>0\}$ 
$$x_{\infty}+iy_{\infty}\mapsto \psi^{-1}(y) |y|^{-k/2} f\left(\twomat y x {}1\right)$$
is holomorphic; in this case such  function determines $f$.

\subsubsection{Petersson inner product} We define a Haar measure $dg$ on $Z(\A_{F})\bks \GL_{2}(\A_{F})$ (where $Z\cong\mathbf{G}_{m}$ denotes the center of $\GL_{2}$) as follows. Recall the Iwasawa decomposition  
\begin{align}\label{iwadec}
\GL_2(\A_F)=B(\A_F)K_0(1)K_\infty
\end{align}
where $K_\infty=\prod_{v\vert \infty} \SO_2(F_v)$. Let $dk=\otimes_{v}dk_{v}$ be the Haar measure on $K=K_{0}(1)K_{\infty}$ with volume $1$ on each component. Let  $dx=\otimes_{v}dx_{v}$ be the Haar  measure such that $dx_{v}$ is the usual Lebesgue measure on $\R$ if $v\vert \infty$, and $\OO_{F,v}$ has volume $1$ if $v\nmid \infty$. Finally let $d^{\times}x=\otimes_{v}d^{\times}x_{v}$ on $F_{\A}^{\times}$ be the product of the measures given by $d^{\times}x_{v}=|dx_{v}/x_{v}|$ if $v\vert\infty$ and by the condition that $\OO_{F,v}^{\times}$ has volume $1$ if $v\vert\infty$. Then we can use the Iwasawa decomposition  $g=\twomat z{}{}z\twomat y x{}1 k$ to define
$$\int_{Z(\A)\bks \GL_{2}(\A)} f(g)\, dg=\int_{F_{\A}^{\times}}\int_{\A}\int_{K} f\left(\twomat y x{}1 k \right) \,dk\,dx\,{d^{\times}y\over |y|}.$$
The Petersson inner product of two forms $f_{1}$, $f_{2}$ on $GL_{2}(F)\bks GL_{2}(\A)$ such that $f_{1}f_{2}$ is invariant under $Z(\A)$ is defined by 
\begin{align*}
\langle f_{1}, f_{2}\rangle_{\rm Pet} =\int_{Z(\A)\bks \GL_{2}(\A)} \baar{f_{1}(g)}f_{2} (g)\, dg
\end{align*}
whenever this converges (this is ensured if either $f_{1}$ or $f_{2}$ is a cuspform as defined below). It will be convenient to introduce a level-specific inner product on forms $f$, $g$ of level $N$:
\begin{align}\label{pet}
\langle f, g\rangle_{N}=\frac{\langle f, g\rangle_{\rm Pet}}{\mu(N)}
\end{align}
where $\mu(N)$ is the measure of $K_{0}(N)$.

\subsection{Fourier expansion}\label{sec:fourier}

Let $f$ be a (not necessarily holomorphic) Hilbert modular form. We can expand it as
$$f(g)=C_{f}(g)+\sum_{\alpha\in F^{\times}} W_{f}\left(\twomat \alpha {}{}1 g\right)$$
where 
\begin{align*} C_{f}(g)&=D_{F}^{-1/2}\int_{\A/F} f\left(\twomat 1 x {} 1 g\right) \, dx,\\
 W_{f}(g)&=D_{F}^{-1/2}\int_{\A/F} f\left(\twomat 1 x {} 1 g\right)\ex(-x) \, dx
\end{align*}
are called the \emph{constant term} and the \emph{Whittaker function} of $f$ respectively. The form $f$ is called \emph{cuspidal} if its  constant term $C_{f}$ is identically zero. The functions of $y$ obtained by restricting the constant term and the Whittaker function to the elements  $\smalltwomat y {}{} 1$ are called the \emph{Whittaker coefficients} of $f$.  When $f$ is holomorphic, they  vanish unless $y_{\infty}>0$ and otherwise have the simple form
\begin{align*}
C_{f}\left(\twomat y{x}{}1\right)&=\wtil{a}^{0}(f,y)=\psi(y)|y|^{k/2}a(f,0),\\
W_{f}\left(\twomat y{x}{}1\right)&=\wtil{a}_{}(f,y)\ex_{\infty}(iy_{\infty})\ex(x)=\psi(y)|y|^{k/2}a(f,y^{\infty}d_{F})\ex_{\infty}(iy_{\infty})\ex(x)
\end{align*}
for functions $\wtil{a}^{0}(f,y)$, $\wtil{a}(f,y)$ of $y\in F_{\A}^{\infty,\times}$  which we call the \emph{Whittaker-Fourier coefficients} of~$f$, and a function $a(f,m)$ of  the fractional ideals $m$ of ${F}$ which vanishes on nonintegral ideals whose values are called the \textbf{Fourier coefficients} of~$f$.

For any $\Z$-submodule $A$ of $\C$,  we denote by ${M}_k(K_1(N),\psi,A)$ the space of holomorphic Hilbert modular forms with Fourier coefficients in $A$ of weight $k$,  level $K_1(N)$, and  character  $\psi$; and by $S_{k}(K_1(N),\psi,A)$ its subspace of cuspidal forms. 
 When the character $\psi$ is trivial we denote those spaces simply by $M_k(K_0(N), A)$ and $S_k(K_0(N),A)$, whereas linear combinations of forms of level $K_{1}(N)$ with different characters form the space $M_{k}(K_{1}(N), A)$. The notion of Whittaker-Fourier coefficients extends by linearity to the spaces $M_{k}(K_{1}(N),\C)$.

We can allow more general coefficients: if $A$ is a $\Z[1/N]$-algebra, we define $S_{k}(K_{0}(N),A)= S_{k}(K_{0}(N),\Z[1/N])\otimes A$; this is well-defined thanks to the $q$-expansion principle \cite{AG}.

\subsection{$p$-adic  modular forms}\label{p-over} Let $N$, $P$ be coprime ideals of $\OO_{F}$, $\psi$ a character of conductor dividing $N$. If $f$ is a holomorphic form  of weight $k$, level $K_{1}(NP)$ and prime-to-$P$ character $\psi$ (that is, $f$ is a linear combination of forms of level $NP$ and character $\psi\psi'$ with $\psi'$ a character of conductor dividing $P$), we associate to it the \emph{formal $q$-expansion coefficients}
$$ a_{p}(f,y^{\infty})=\psi^{-1}(y)|y|^{-k/2}\wtil{a}(f,y).$$
If $\psi'$ is trivial we set $a_{p}(f,m)=a_{p}(f,y^{\infty})$ if $m$ is the ideal $m=y^{\infty}d_{F}$.

Let $N$ be an ideal prime to $p$, $\psi$ a character of level dividing $N$. Consider the space of classical modular forms $M_k(K_1(Np^\infty),\baar{\Q})$  with character whose prime-to-$p$ part is equal  to $\psi$, and endow it with the norm given by the maximum of the $p$-adic absolute values (for the chosen embedding $\baar{\Q}\hookrightarrow \C_{p}$) of the Fourier coefficients. Its completion
 $$\mathbf{M}_k(K_1(N),\psi, \C_{p})$$  
 of this space is a $p$-adic Banach space called the space of  \textbf{$p$-adic modular forms} of weight $k$,  \emph{tame level} $K_1(N)$ and tame character $\psi$. We shall view $M_{k}(K_{1}(Np^{r},\psi\psi',\calA)$ (for any character $\psi'$ of conductor divisible only by primes above $p$)   as a subset of $\mathbf{M}_k(K_1(N),\psi,\calA)$ via the $q$-expansion map.
 
If  $\calA$ is  a complete $\Z_p$-submodule of $\C_p$, we also use the notation  $\mathbf{M}_k(K_1(N),\psi,\calA)$ with obvious meaning, and  $\mathbf{S}_k(K_1(N),\psi,\calA)$ or   $\mathbf{S}_k(K_0(N),\calA)$ (in the case of trivial tame character) for cuspforms; when $k=2$ we  write simply 
$$\mathbf{S}_N(\calA)=\mathbf{S}_2(K_0(N),\calA)$$
or just $\mathbf{S}_N$ if $\calA=\Q_{p}$ or $\calA=\C_{p}$ (as understood from context). 

\subsection{Operators acting on modular forms}\label{sec:hecke}
There is a natural action of the group algebra $\Q[\GL_{2}(\A^{\infty})]$  on modular forms induced by right translation. Here we describe several interesting operators arising from this action.

Let $m$ be an ideal of $\OO_{F}$, $\pi_{m}\in F_{\A^{\infty}}^{\times}$ a generator of $m\widehat{\OO}_{F}$ which is trivial at places not dividing $m$.

The operator $[m]\colon M_{k}(K_{1}(N),\psi)\to M_{k}(K_{1}(Nm),\psi)$ is defined by
\begin{align}\label{[d]}
[m]f(g)=\N (m)^{-k/2}  f\left(g\twomat {1}{}{}{\pi_{m}}\right).
\end{align}
It acts on Fourier coefficients by 
$$a([m]f, n)=a(f,m^{-1}n).$$
If $\chi$ is a Hecke character of $F$, we denote by $f|\chi$ the form with coefficients
$$a(f|\chi,n)=\chi(n)a(f,n).$$

For any double coset decomposition $$K_{1}(N)\twomat {\pi_{m}}{}{}1 K_{1}(N)=\coprod_{i} \gamma_{i} K_{1}(N),$$
the \textbf{Hecke operator} $T(m)$  is defined by the following level-preserving action on forms $f$ in $M_{k}(K_{1}(N))$:
$$T(m)f(g)=\N(m)^{k/2-1}\sum_{i}f(g\gamma_{i});$$
For $m$ prime to $N$, its effect on Fourier coefficients of forms with trivial character is described by
$$a(T(m)f, n)=\sum_{d|(m,n)}\N(d)^{k/2-1}a(f, mn/d^{2}).$$

When $m$ divides $N$, we can pick as double coset representatives the matrices $\gamma_{i}=\twomat{\pi_{m}} {c_{i}}{}{1}$ for $\{c_{i}\}\subset\widehat{\OO}_{F}$ a set of representatives for $\OO_{F}/m\OO_{F}$. Then the operator $T(m)$ is more commonly denoted $U(m)$ and we will usually follow this practice. It acts on Fourier coefficients of forms with trivial character by
$$a(U(m)f,n)=\N(m)^{k/2-1} a(f,mn).$$

Let $\mathbf{T}_{N}$ be the (commutative) subring of $\End S_{2}(K_{0}(N),\Z)$ generated by the $T(m)$ for $m$ prime to $N$. 
A form $f$ which is an eigenfunction of all the operators in $\mathbf{T}_{N}$ is called a Hecke \emph{eigenform}.  It is called a \emph{primitive} form if moreover  it is normalised by $a(f,1)=1$ and it is a newform (see \S \ref{sec:calS} below for the definition) of some level dividing $N$.

As usual (cf. \cite[Lemme 1.10]{PR}) we will need the following well-known lemma to ensure the modularity of certain generating functions.

\begin{lemm}\label{ismodular} Let $A$ be a $\Q$-algebra. For each linear form $$a\colon {\mathbf{T}}_N \to A$$  there is a unique   modular form in  $\oplus_{N'\vert N}S_{k}^{\rm new}(K_{0}(N'),A)$ whose Fourier coefficients are given by $a(T(m))$ for all $m$ prime to~$N$. 
\end{lemm}
\begin{proof} In  \cite[Corollary 3.18]{shouwu}, the result is stated and proved when $A=\C$ as a consequence  of the existence of a pairing $(T,f)\mapsto a_{1}(Tf)$ between ${\bf T}_{N}$ and the space of modular forms of interest; but this pairing is defined over $\Q$, hence the result is true for $A=\Q$ and by extending scalars for any $\Q$-algebra $A$. 
\end{proof}

\subsubsection{Atkin-Lehner theory} For any nonzero ideal $M$ of $\OO_F$, let $W_M\in \GL_2(\A^\infty)$ be a matrix with components
\begin{align}\label{WM}
W_{M,v}= \twomat {}{1}{-\pi_v^{v(M)}}{} \quad \textrm{if $v\vert M$,}\qquad W_{M,v}= \twomat 1 {}{}1 \quad\textrm{if  $v\nmid M$}
\end{align}
where $\pi_v$ is a uniformiser at $v$. 
We denote by the same name $W_M$ the operator acting on modular forms of level $N$ and trivial character by 
$$W_{M}f(g)=f(gW_{M});$$
it is self-adjoint for the Petersson inner product, and when $M$ is prime to $N$ it is proportional to the operator $[M]$ of \eqref{[d]}. On the other hand when $M$ equals $N$, or more generally $M$ divides $N$ and is coprime to $NM^{-1}$, the operator  $W_M$ is  an involution and its action is particularly interesting. In this case, extending the definition to forms of level $K_{1}(N)$ and character\footnote{Notice that a decomposition of $\psi$ as described is only unique up to class group characters (that is, Hecke characters of level one). We will only be using the operator $W_{M}$ for $M$ a proper divisor of $N$ in a case in which a decomposition is naturally given.} $\psi=\psi_{(M)}\psi_{(NM^{-1})}$ with $\psi_{(C)}$ of conductor dividing $C$, we have 
\begin{gather}\label{WMop}
W_M f(g)=\psi_{(M)}^{-1}(\det g) \psi_{(M)}^{-1}(\pi_{M}) f(gW_M)
\end{gather} 
where $\pi_{M}$ is the id\'ele with nontrivial components only at $v|M$ and given there by $\pi_{v}^{v(M)}$. It is easy to check that this definition is independent of the choice of uniformisers. 
The effect of the $W_{M}$-action on newforms is described  by Atkin--Lehner theory; we summarise it here (in the case $M=N$), referring to \cite{casselman} for the details.

Let $\pi$ be an irreducbile infinite-dimensional automorphic representation of $\GL_2(\A_F)$ of central character  $\psi$. Up to scaling, there is a unique \emph{newform} $f$ in the space of $\pi$. It is characterised by either of the equivalent properties: (a) it is fixed by a subgroup $K_{1}(N)$ with $N$ minimal among the $N'$ for which $\pi^{K_{1}(N')}\neq0$; (b) its Mellin transform is (a multiple of) the $L$-function $L(\pi,s)$ of $\pi$. In the case of a holomorphic cuspform, this is equivalent to requiring that it belongs to the space of newforms defined in  \S\ref{sec:calS} below. There is a functional equation relating the $L$-function  $L(s,\pi)$  of $\pi$  and  the $L$-function $L(1-s,\wtil{\pi})$ of the contragredient representation; as $\wtil{\pi}\cong\psi^{-1}\cdot \baar{\pi}$, it translates into the following description of the action of $W_N$ on newforms. Suppose that the eigenform $f\in S_k(K_1(N),\psi)$  is a newform in the representation $\pi$ it generates, then we have
\begin{gather}\label{atkin-lehner}
W_N f(g)= (-i)^{[F:\Q]k}\tau(f) f^\rho(g)
\end{gather}
where $f^{\rho}$ is the form with coefficients
\begin{align}\label{frho}
a(f^{\rho},m)=\baar{a(f,m)}
\end{align}
and  $\tau(f)=\tau(\pi)$ is an algebraic number of complex absolute value one.

\subsubsection{Trace of a modular form}
The \textbf{trace} of  a modular form $f$ of level $ND$ and trivial character is the form of level $N$
$$\mathrm{Tr}_{ND/N}(f)(g)=\sum_{\gamma\in K_0(N)/K_0(ND)} f(g\gamma).$$
It is the adjoint of inclusion of forms of level $N$ for the rescaled Petersson product:
\begin{align*}
\langle f, {\rm Tr}_{ND/N} g\rangle_{N}=\langle f, g\rangle_{ND}
\end{align*}
if $f$ has level $N$ and $g$ has level $D$.

Suppose  that $D$ is squarefree  
and prime to $N$, in which case we can write ${\rm Tr}_{D}={\rm Tr}_{ND/N}$ without risk of ambiguity.  A set of coset representatives for $K_0(N)/K_0(ND)$ is given by elements $\gamma_{j,\delta}$ for $\delta|D$, $ j\in \OO_{F,v}/\delta\OO_{F,v}$,
having components 
$$\gamma_{j,\delta,v}=\begin{pmatrix}1 & j\\ & 1 \end{pmatrix}\begin{pmatrix} & 1\\-1 &\end{pmatrix}={1\over \pi_v} \begin{pmatrix}\pi_v & j\\ & 1 \end{pmatrix}\begin{pmatrix} & 1\\-\pi_v &\end{pmatrix}$$ 
at places $v | \delta$, and $\gamma_{j,\delta,v}=1$ everywhere else.
From the second decomposition given just above, if $f$ has weight~$2$ we obtain 
\begin{align}\label{fouriertrace}
a(\mathrm{Tr}_{D}(f),m)=\sum_{\delta|D}a(U(\delta)f^{(\delta)},m)=\sum_{\delta|D}a(f^{\delta)}, m\delta)
\end{align}
where $f^{(\delta)}(g)=f(g W_\delta)$ with $W_\delta$ as in \eqref{WM}.

\begin{rema}  If $D$ is prime to $p$, the various trace operators $\mathrm{Tr}_{NDp^r/Np^r}$ extend to a continuous operator $\mathrm{Tr}_{ND/N}$ on $p$-adic modular forms of  tame level $ND$.  
Similarly the operators $[m]$, $T(m)$  and $W_{m}$ for $m$ prime to $Np$ extend to continuous operators on $p$-adic  modular forms of tame level $N$.
\end{rema}

\subsubsection{Ordinary projector} Let $L$ be a complete subfield of $\C_{p}$. Following Hida (see e.g. \cite[\S 3]{Hi}) we can define  for each $\wp\vert p$ an indempotent 
$$e_{\wp}=\lim_{n\to \infty} U_{\wp}^{n!}\colon {\bf S}_{N}(L)\to S_{N\wp}(L)$$
which is surjective onto $S_{N\wp}^{\wp{\rm -ord}}(L)$, the subspace of $S_{N\wp}(L)$ spanned by $U_{\wp}$-eigenforms with unit eigenvalue. 

 Let $P=\prod_{\wp\vert p}\wp$. Then we similarly  have a surjective idempotent
 $$e=\prod_{\wp \vert p} e_{\wp}\colon  {\bf S}_{N}(L)\to S_{NP}^{\rm ord}(L),$$
where $S_{NP}^{\rm ord}(L)$ is the subspace of $S_{NP}(L)$ spanned by  simultaneous $U_{\wp}$-eigenforms with unit eigenvalue.

\subsection{Fourier coefficients of old forms}\label{sec:calS} As we will study modular forms through their Fourier coefficients, we give here  a criterion for recognising the coefficients of certain old forms.\footnote{Cf. \cite[\S4.4.4]{shouwu}.} Let $N$, $P$ be coprime ideals of $\OO_F$. The space $S_{NP}^\textrm{$N$-old}\subset S_{NP}$ is the space spanned by forms $f=[d]f'$ for  some $1\neq d\vert N$  and some cuspform $f'$ of level $N'P$ with $N'\vert d^{-1}N$. In the case $P=1$, we define the space of \textbf{newforms} of level dividing $N$ to be the orthogonal to the space of old forms for the Petersson inner product. 
We denote by $\mathbf{S}^\mathrm{old}_N\subset \mathbf{S}_N$ the closed subspace generated by the image of $S_{Np^\infty}^{\textrm{$N$-old}}$ in~$\mathbf{S}_N$.
(The coefficient ring will always be either a finite extension of $\Q_{p}$ or $\C_{p}$, as understood from context when not present explicitly in the notation).

\medskip

Let now $\calS$ be the space of functions $f\colon \N_F\to \calA$ modulo those for which there is an ideal $M$ prime to $p$ 
such that $f(n)=0$ for all $n$ prime to $M$.  A function $f\in\calS$ is called \emph{multiplicative} if it satisfies\footnote{This relation and the following are of course to be understood to hold in $\calS$.} $f(mn)=f(m)f(n)$ for all $(m,n)=1$. For $h$ a multiplicative function, a function $f$ is called an \emph{$h$-derivative} if it satisfies $f(mn)=h(m)f(n)+h(n)f(m)$ for all $(m,n)=1$.

Let  $\sigma_1$ and $r$ be the multiplicative elements of $\calS$ defined by 
$$\sigma_1(m)=\sum_{d\vert m}\N (d), \qquad r(m)=\sum_{d\vert m} \eps_{E/F}(d)$$
(where $E$ is a totally imaginary quadratic extension of $F$ of discriminant prime to $p$).\footnote{We will see below that $\sigma_1$ and $r$ are the Fourier coefficients of weight $1$ Eisenstein series and theta series.} Let $P=\prod_{\wp\vert p }\wp \subset \OO_{F}$. We define a subspace $\mathcal{D}_N\subset \calS$ to be generated by $\sigma_1$, $r$, $\sigma_1$-derivatives, $r$-derivatives, and Fourier coefficients of forms in ${S}^{N\rm{-old}}_{NP}$.

\begin{lemm}\label{inj} The $q$-expansion map $S_{NP}^{\rm ord}/S_{NP}^{N-{\rm old}}\to \calS/\mathcal{D}_{N}$ is injective.
\end{lemm}
\begin{proof}
First notice that it is enough to show this when the coefficient ring is a number field $L$ over which $S_{NP}^{\rm ord}$ is defined (it suffices for $L$ to contain all the eigenvalues of  the operators $T_{\ell}$   ($\ell\nmid Np$) and $U_{\wp}$ on  $S_{NP}(L)$).
By \cite[Proposition 4.5.1]{shouwu}, the kernel of $S_{NP}(L)/S_{NP}^{N{\rm -old}}(L)\to \calS/\mathcal{D}_{N}$ is at most generated by $S_{NP}^{p{\rm-old}}(L)=\sum_{\wp\vert p} S_{NP}^{\wp{\rm-old}}(L) $, the space of forms which are old at some $\wp\vert p$. To conclude, it suffices to show that for each $\wp\vert p$ we have $I:=S_{NP}^{\wp{\rm-old}}\cap S_{NP}^{\wp{\rm -ord}}=0$. The intersection $I$ is stable under the action of ${\bf T}_{NP}$ which decomposes it into spaces $I[f_{i}]\subset S_{NP}[f_{i}]$ corresponding to eigenforms $f_{i}$ of level $N'$ or $N'\wp$ for some $N'\vert NP\wp^{-1}$. If $f_{i}$ has level $N'\wp$ then $S_{NP}[f_{i}]$ does not contain any nonzero $\wp$-oldforms. If $f_{i}$ has level $N'$ with $\wp\nmid N'$ then $S_{NP}^{\ord}[f_{i}]$ is either zero or the line spanned by the ordinary $\wp$-stabilisation of $f_{i}$, whereas $S_{NP}^{\wp{\rm -old}}[f_{i}]$ is the line spanned by $[\wp]f_{i}$. We conclude that $I[f_{i}]=0$ in all cases.
\end{proof}

\begin{rema} The  operators $U_\wp$  for $\wp\vert p$ extend to operators on $\calS$ via $U_\wp f(m)=f(m\wp)$. The Hecke algebra ${{\bf T}}_{Np}$ acts on the image $\calS_{N}$ of ${\bf S}_{N}$ in $\calS/{\mathcal D}_{N}$.
\end{rema}

\subsection{The functional $\Lf$}\label{sec:Lf}
Recall from the Introduction that we have fixed an ordinary primitive Hilbert modular newform $f$ of level $K_0(N)$.
If $\alpha_{\wp}$ is the unit root of the $\wp^{\mathrm{th}}$ Hecke polynomial of $f$, $\beta_{\wp}$ is the other root,  and the operator $[\wp]$ is as in \eqref{[d]}, then the \emph{$p$-stabilisation} of $f$ is 
$$f_\alpha=\prod_{\wp\vert p} \left(1-{\beta_{\wp}}[\wp]\right)f,$$
a form of level $K_{0}(N\prod_{\wp\vert p}\wp)$ satisfying $U_{\wp}f_{\alpha}=\alpha_{\wp}f_{\alpha}$ for all $\wp\vert p$.

We define a functional, first introduced by Hida, which plays the role of projection onto the $f$-component. Both sides of our main formula will be images of $p$-adic modular forms under this operator. 

Let $P$ be an ideal of $\OO_{F}$ divisible exactly by the primes $\wp\vert p$. For a form  $g\in M_2(K_0(NP))$ with $r\geq 1$, let
$$\Lf(g)= \frac{\langle W_{NP} f_{\alpha}^\rho,\, g\rangle }{\langle W_{NP} f_{\alpha}^\rho,\, f_{\alpha} \rangle}.$$%_{\rm Pet}  volendo
Let $L\subset\baar{\Q}_{p}$ be the   extension of $\Q_{p}$ generated by  $a(f, m)$  for all ideals $m$ and  $\alpha_{\wp}$ for $\wp\vert p$.
\begin{lemm}[Hida]\label{Lf} The above formula defines  a  linear functional  
$$\Lf\colon  M_2(K_0(Np^\infty),L)\to L$$
 satisfying:
\begin{enumerate}
\item\label{firstitem} On $M_2(K_0(N), L)$ we have
$$\Lf=\prod_{\wp\vert p} \left(1-{\N\wp\over\alpha_\wp^2}\right)^{-1}\one_f$$
where $\one_f(g)=\langle f,g\rangle/\langle f,f\rangle$.
\item\label{killcrit} On $M_2(K_{0}(N\wp^r))$ we have, for each nonnegative $t\leq r-1$,
 $$\Lf\circ U_\wp^t=\alpha_\wp(f)^t\Lf.$$
\item\label{secondlastitem} If each $\iota_{p}(\alpha_{\wp})$ is a $p$-adic unit, $\Lf$ admits a continuous extension to $p$-adic modular forms still denoted  
$$\Lf\colon  \mathbf{M}_N(L)\to L.$$\end{enumerate}
\end{lemm}
\begin{proof} See \cite[Lemma 9.3]{Hi} where the well-definedness of the functional and its extension to $p$-adic modular forms are proved more generally for Hida families.  For part 1 the computation is the same as in the case  of  elliptic modular forms: see ~\cite{PRlondon} 
or \cite[\S4]{hidaI}. 
\end{proof}

\subsubsection{Some quotient spaces}
Let $\bcalS=\calS/\mathcal{D}_{N}$. The ordinary projection operator $e$ is not defined on all   arithmetic functions; however its kernel $\Ker(e)$ is a well-defined subspace of $\calS$. We define $$\bcalS^{\rm ord}:= \calS/\mathcal{D}_{N}+\Ker({e}).$$
 The quotient map $\bcalS\to \bcalS^{\rm ord}$ is clearly injective when restricted to the image of $S_{NP}^{\ord}$, where $P=\prod_{\wp\vert p}\wp$. Then we denote by $\bcalS_{N}^{\rm ord}$ the image of $S_{NP}^{\rm ord}$ in either $\bcalS$ or $\bcalS^{\rm ord}$.  It is also identified with the common image of $S_{NP}$  and ${\bf S}_{N}$ in $\bcalS^{\rm ord}$.
We denote by $\bcalS_{N}^{p{\rm -adic}}\subset \bcalS$ the image of ${\bf S}_{N}$. 

We obtain a commutative  diagram (where $L$ is as usual any sufficiently large finite extension of $\Q_{p}$):
\begin{equation}\label{quotient}
\xymatrix{
{\bf S}_{N}(L) \ar[r]\ar[d]^{e} & \bcalS_{N}^{p{\rm -adic}}(L) \ar[d] \ar@{^{(}->}[r] &  \bcalS^{\rm ord} \\
S_{NP}^{\rm ord}(L)/S_{NP}^{N{\rm -old}} \ar[r]^{\wtil{}} & \bcalS_{N}^{\rm ord}(L) \ar[r]^{\Lf} & L
}
\end{equation}
where the right-hand vertical map is the restriction of  the quotient  $\bcalS\to \bcalS^{\rm ord}$, and the bottom horizontal map is an isomorphism by Lemma  \ref{inj}.

\section{Theta measure}
We construct a measure on the Galois group of the maximal abelian extension of $E$ unramified outside $p$ with values in $p$-adic theta series, and compute its Fourier expansion.
\subsection{Weil representation} 
We first define the Weil representation. See \cite[\S4.8]{bump} for an introduction, and  \cite{wald} or \cite{yzz} for  our conventions on  the representation for  similitude groups. 
\subsubsection{Local setting}
Let $V=(V,q)$ be a quadratic space over a local field $F$ of characteristic not 2, with a quadratic form $q$; we choose a nontrivial additive character $\e$ of $F$. For simplicity we assume $V$ has even dimension. For $u\in F^\times$, we denote by $V_u$ the quadratic space $(V,uq)$. We let $\GL_2(F)\times \GO(V)$ act on the space $\mathcal{S}(V\times F^\times)$ of Schwartz functions as follows (here $\nu\colon \GO(V)\rightarrow\mathbf{G}_m$ denotes the similitude character):
\begin{itemize}
\item $r(h)\phi(t,u)=\phi(h^{-1}t,\nu(h)u)$ \quad for $h\in \GO(V)$;
\item$r(n(x))\phi(t,u)=\e(xuq(t))\phi(t,u)$ \quad for $n(x)=\begin{pmatrix} 1& x\\&1\end{pmatrix}\in \GL_2$;
\item $r\left(\begin{pmatrix}a&\\& d\end{pmatrix}\right)\phi(t,u)=\chi_{V}(a)|{a\over d}|^{\dim V\over 4}\phi(at, d^{-1}a^{-1}u)$;
\item $r(w)\phi(x,u)=\gamma(V_u)\hat{\phi}(x,u)$ for $w=\begin{pmatrix}&1\\-1 &\end{pmatrix}.$
\end{itemize}
Here $\chi_V$ is the quadratic character associated with $V$, $\gamma(V_u)$ is a certain square root of $\chi(-1)$, 
and $\hat{\phi}$ denotes the Fourier transform in the first variable
$$\hat{\phi}(x,u)=\int_{V} \phi(y,u)\e(-u\langle x,y\rangle)dy$$ 
where $\langle\ , \ \rangle$ is the bilinear form associated with $q$ and $dy$ is the self-dual Haar measure.

\subsubsection{Global setting} Given  a quadratic space $(V,q)$ over a global field $F$ of characteristic not 2 (and a nontrivial additive character $\e\colon F \backslash \A_F \rightarrow \C^\times$), the Weil representation is the restricted tensor product  $r$ of the associated local Weil representations, with  spherical functions $\phi_v(t,u)=\mathbf{1}_{\mathcal{V}_v\times \OO_{F,v}^\times}(x,u)$ for some choice of lattices $\mathcal{V}_v\subset V(F_v)$.

\medskip

The  case of interest to us is the following: $F$ is a totally real number field, $V=(E,\frakN)$ is given by a quadratic CM extension $E/F$ with the norm form $\frakN=N_{E/F}$ and the lattices $\OO_{E,v}\subset E_{v}$, and the additive character $\e$ is the standard one. We denote $G=\GL_2$, $H=\GO(V)$, two algebraic groups defined over $F$; we have $H\cong\Res_{E/F}\mathbf{G}_m$. In this case we have $$\chi_{V}=\eps_{E/F}=\eps,$$
where $\eps_{E/F}$ is the quadratic character of $F_\A^\times$ associated with the extension $E/F$. The self-dual measure on $E_{v}$ is the one giving $\OO_{E,v}$ volume $|\OO_{E,v}/{\frak D}_{v}|^{-1/2}$ where ${\frak D}_{v}$ is the relative  different  Moreover
 the constant  $\gamma$  can be explicitly described (see \cite[\S\S\ 38.6, 30.4,  23.5]{bh}): in the case $v\vert \Delta_{E/F}$, which is the only one we will be using, such description is in  terms of a local Gau{\ss} sum~$\kappa(v)$:
\begin{gather}\label{gamma}
\gamma(E_v, u\frakN)=\eps_v(u)\kappa(v)=\eps_v(u)|\pi_v\vert ^{1\over 2}\sum_{x\in(\OO_{F,v}/\pi_v\OO_{F,v})^\times} \eps(x/\pi_v)\ex_v(x/\pi_v).
\end{gather}
Notice that our $\kappa(v)$ is the \emph{inverse} of the quantity denoted by the same name in \cite[Proposition 3.5.2]{shouwu}.

\subsection{Theta series}\label{sec:thetaseries}

We define the \textbf{theta kernel} to be
$$\theta_\phi(g,h)=\sum_{(t,u)\in V\times F^\times} r(g,h)\phi(t,u) $$
which is an automorphic form for the group $\GL_2(F)\backslash \GL_2(\A_F)\times \GO(V)\backslash \GO(V_{\A_F})$.

If $\W$ is an automorphic function for $H$ which is trivial at infinity (which is the same thing as a linear combination of finite order Hecke characters of $E$), we define the \textbf{theta series}\footnote{The reason for taking $\W(h^{-1})$ rather than $\W(h)$ is that we want $\theta_{\phi}(\W)$ to be the series classically denoted $\Theta(\W)$ for a suitable choice of $\phi$ -- this will be clear from the computations below.}
$$\theta_\phi(\W)(g)=\int_{H(F)\backslash H(\A_F)} \W(h^{-1})\theta_\phi(g,h)\,dh$$
which is an automorphic form on $G$. Here the measure $dh$ is the product of the measure on $H(\A^\infty)$ which gives volume 1 to the compact $U_0=\widehat{\OO}_E^\times$, and any fixed measure\footnote{There will be no ambiguity since later we will choose $\phi_\infty$ to be again any fixed Schwartz function, whose integral over $H(\A_\infty)$ with respect to the chosen measure is a specified function $\baar{\phi}_\infty$.}  on $H(\A_\infty)$.

Let us explain how to explicitly compute the integral in our situation. For each open compact subgroup $U\subset H(\A^\infty_F)=E_{\A^\infty}$, we have  exact sequences
$$1\to \OO_{E,U}^\times \backslash U E_\infty^\times \to E^\times\backslash E_\A^\times \to E^\times U\backslash E_{\A^\infty}^\times\to 1$$
and
$$1\to \mu(U)\backslash U E_\infty^1 \to \OO_{E,U}^\times \backslash U E_\infty^\times \stackrel{\frakN_\infty}{\to} N(\OO_{E,U}^\times)\backslash F_\infty^+\to 1.$$
The notation used is the following: $\OO_{E,U}^\times=E^\times\cap U \supset \mu(U)=$ the subset of roots of unity, $\frakN_\infty\colon E_\infty^\times \to F_\infty^+$ is the norm map at the infinite places and $E_\infty^1$ is its kernel.

We can choose a splitting $\iota$ of the first sequence, for example $$\iota\colon E^\times U\backslash E_{\A^\infty}^\times\cong E^\times U \backslash (E_\A^\times)^{1,\parallel}\hookrightarrow  E^\times\backslash E_\A^\times,$$
where $(E_\A^\times)^{1,\parallel}$ denotes the set of id\`{e}les  of adelic norm 1 with infinity component $h_\infty=(h,\ldots,h)$ for some real number $h>0$ and the isomorphism is the unique one which gives the identity once composed with projection onto the finite part. 

We begin to expand the series, evaluating the integral as explained above and exploiting the fact that the action of $H(F_\infty)=E_\infty^\times$ on  $\phi(t,u)$ factors through the norm. We take $U$ to be small enough so that $\W$ and $\phi$ are invariant under  $U$, and denote
$$\baar{\phi}_v(t,u)=\int_{H(F_v)} r(h)\phi_v(t,u)\,dh  \quad\quad\textrm{ if } v\vert \infty$$
and $\baar{\phi}=\prod_{v\nmid \infty}\phi_v\prod_{v\vert \infty}\baar{\phi}_v$. A specific choice of $\baar{\phi}_v$ will be made shortly: for the moment we just record, and use in the following computation, that we will take $u\mapsto \baar{\phi}_v(t,u)$ to be supported on $\R^+$.

We have
\begin{gather*} 
\theta_\phi (\W)(g) =
\int_{E^\times\backslash E_\A^\times} \W(h^{-1})\theta_\phi(g,h)\,dh\\
=w_U^{-1}\int_U\int_{E_\infty^1} \int_{\frakN(\OO_{E,U}^\times)\bks F^+_\infty} \int_{E^\times U\bks E_{\A^\infty}^\times} \W(\iota(a)^{-1}) \sum_{(t,u)\in E\times F^\times} r(g,\iota(a)h)\phi(t,u)  \,da \, dh
\end{gather*}
Here  $w_U=|\mu(U)|$ and  $dh$ denotes  the measure on $U\times E_\infty^1\times F^+_\infty=U\times H(F_\infty)$. We partially collapse the integral over $\frakN(\OO_{E,U}^\times)\bks F^+_\infty$ and the sum over $u\in F^\times$  and use our choice of $\phi_\infty$ to get
\begin{equation}\label{eqn:expand}
\begin{split}
&=w_U^{-1}\vol(U)\int_{E^\times U\bks E_{\A^\infty}^\times} \W(\iota(a)^{-1})\sum_{u\in \frakN(\OO_{E,U}^\times)\bks F^+} \sum_{t\in E} r(g,\iota(a))\overline{\phi}(t,u) \,da \\
&= w^{-1}{h\over h_U} \int_{E^\times U\bks E_{\A^\infty}^\times} \W(\iota(a)^{-1})\, \nu_U\sum_{u\in \frakN(\OO_{E,U}^\times)\bks F^+} \sum_{t\in E} r(g,\iota(a))\overline{\phi}(t,u) \,da 
\end{split}
\end{equation}

Here in the last step we have defined  $\nu_U= [\frakN(\OO_E^\times): \frakN(\OO_{E,U}^\times)]$ and computed $\vol(U)=\vol(U_0)(h/h_U) (w_U/w)\nu_U^{-1}$, where $U_0=\widehat{\OO}_E^\times$,  $h_U=|E^\times U\bks E_{A^\infty}^\times|$, $h=h_{U_0}$, $w=w_{U_0}$. Recall that our measure satisfies $\vol(U_0) =1$.  
The remaining integral is just a finite sum.

The sum over $u$ is actually finite owing to the integrality constraints imposed by $\phi$ at finite places.\footnote{We will see this in more detail shortly. We are also using the definition of $\overline{\phi}_\infty$ in order to freely replace the sum over $u\in F^\times$ with a sum over $u\in F^+$ -- in fact a slight variation would be necessary when $\det g_\infty\notin F_\infty^+$, but this is a situation we won't encounter.}

\subsection{Theta measure}\label{measure}\label{sec:calG}

We define a measure with values in $p$-adic modular forms on the group 
$$\calG'=\Gal(E_{\infty}'/E)\cong \varprojlim \baar{E^\times U_{p^n}}\bks E_{\A^\infty}^\times$$
where the over line denotes closure and $E'_{\infty}$ is the maximal abelian extension of $E$ unramified outside $p$, that is, the union of the ray class fields of $E$ of $p$-power ray $U_{p^n}=\prod_v\{\textrm{units }\equiv 1 \mod  p^n\OO_{E,v}\}$ and the isomorphism is given by class field theory. The topology is the profinite topology.

Recall that a \textbf{measure} on a topological space $\calG$ with values in a $p$-adic Banach space  $\mathbf{M}$ is a ${\C}_{p}$-linear functional
$$\mu:\mathcal{C}(\calG, \C_{p})\to \mathbf{M}$$
on continuous $\C_{p}$-valued functions, which is continuous (equivalently, bounded) with respect to the sup norm on $\mathcal{C}(\calG,\C_{p})$. The linearity property will be called distributional property in what follows. The boundedness property  will in each case at hand be verified on the set of  $p$-adic characters of $\calG$, which in our cases generates the whole of $\mathcal{C}(\calG, \C_{p})$  (classically, the continuity  of $\mu$ goes under the name of  \emph{abstract Kummer congruences} for $\mu$).

When ${\bf M}={\bf M}_{0}\otimes_{\Q_{p}}\C_{p}$ for a $p$-adic Banach space ${\bf M}_{0}$ over $\Q_{p}$, the measure $\mu$ is said to be \emph{defined over $\Q_{p}$} if $\mu(\W)\in {\bf M}_{0}\otimes \Q_{p}(\W)$ whenever the function  $\W$ on $\calG$ has values in $\Q_{p}(\W)\subset \baar{\Q}_{p}\subset\C_{p}$.

\begin{defi}\label{deftheta} The \textbf{theta measure} $d\Theta$ on $\calG'$ is defined by 
$$\Theta(\W)=\int_{\calG'} \W(\sigma)\,d\Theta(\sigma):=\theta_\phi(\W),$$
for any function $\W\colon \calG'\rightarrow \baar{\Q}$ factoring through a finite quotient of $\calG'$,
where the function $\phi$ is chosen as follows:
\begin{itemize}
\item for $v\nmid  p\infty$, $\phi_v(t,u)=\mathbf{1}_{\OO_{E,v}}(t)\mathbf{1}_{d_{F_v}^{-1,\times}}(u)$;  
\item for $v\vert p$, $$\phi_v(t,u)= [\OO_{E,v}^\times:U_v']\mathbf{1}_{U'_v}(t)\mathbf{1}_{d_{F_v}^{-1,\times}}(u),$$ where $U_v'\subset \OO_{E,v}^\times$ is any small enough compact set -- that is,  $U_v'\subset U_v$ if $\W$ is invariant under $U=\prod_v U_v$, and the definition does not depend on the choice of $U_v$. (In practice, we will choose $U_v'=U_v$ if $U_v$ is maximal with respect to the property just mentioned.)
\item for $v\vert  \infty$, $\phi_v(t,u)$ is a  Schwartz function such that
$$\int_{H(F_v)} r(h)\phi_v(t,u)\,dh=\overline{\phi}_v(t,u)= 1_{\R^+}(u)\exp(-2\pi u N(t)).$$
(See \cite[4.1]{yzz} for more details on this choice.)
\end{itemize}
\end{defi}

In Corollary \ref{theo:thetameasure} below we will show that  this in fact defines a measure on $\calG'$ with values in $p$-adic Hilbert modular forms of weight one, tame level $\Delta_{E/F}$ and character $\eps$.

\subsection{Fourier expansion of the theta measure / I}\label{sec:expandtheta}

We compute the Fourier expansion of the theta measure on $\calG'$, carrying on the calculation started in \S\ref{sec:thetaseries}. 

In  the case where $g=\begin{pmatrix}y &x\\& 1\end{pmatrix}$ with $y_\infty>0$, the sum over $(u,t)$ in \eqref{eqn:expand} evaluates to 
\begin{gather}\label{eqn:expand2}
\eps(y) |y|^{1/2}  \sum_{u,t} \phi^\infty(a^{-1}yt,\frakN(a)y^{-1}u) \ex_\infty(iy_\infty uN(t)) \ex(x u \frakN(t)).
\end{gather}
Then we compute the sum of this expression over the finite quotient   $\calG'_U$ of $\calG'$, with $\calG'_U\cong E^\times U\bks E_{\A^\infty}^\times$. 

We assume $\W$ is a character so $\W(a^{-1})=\baar{\W}(a)$ where $\baar{\W}=\W^{-1}$.

First we pre-compute the product of all the constants appearing in the theta series of \eqref{eqn:expand}, including the one from $\phi$ -- we take $$\phi_v(t,u)= [\OO_{E,v}^\times:U_v]\mathbf{1}_{U_v}(t)\mathbf{1}_{\OO_F^\times}(u)%=[\OO_{E,v}^\times:U_v]\phi_v^U(t,v)
,$$ so:
\begin{align*}
w{h\over h_U} \nu_U [\OO_{E,v}^\times:U_v]=&w[ \OO_E^\times \bks\widehat{\OO}_{E,v}^\times: \OO_{E,U}^\times \bks U]^{-1} [\frakN(\OO_E^\times):\frakN(\OO_{E,U}^\times)]^{-1}  [\widehat{\OO}_E^\times : U] \\ 
=& w [\mu(\OO_E):\mu({U})]=w_U^{-1}.
\end{align*}

This computation together with \eqref{eqn:expand}, \eqref{eqn:expand2} gives 

\begin{align*}
\Theta(\W) =&  \eps(y) |y|^{1\over 2} w_U^{-1}  \sum_{a\in E^\times U\bks E_{\A^\infty}^\times} \baar{\W}(a) \sum_{t\in E, u\in \frakN(\OO_{E,U}^\times)\bks F^+} \phi^{p\infty}(a^{-1}yt,\frakN(a)y^{-1}u) \\
& \times \one_{\OO_{E,U,p}^\times}(a^{-1}yt) \one_{d_{F_p}^{-1,\times}}(\frakN(a)y^{-1}u) \,\ex_\infty(iy_\infty u\frakN(t)) \, \ex(x u \frakN(t))  \\
=&   \eps(y)\baar{\W}(y) |y|^{1\over 2} w_U^{-1}\sum_{a\in E^\times U\bks E_{\A^\infty}^\times} \baar{\W}(a)  \sum_{t\in E, u\in \frakN(\OO_{E,U}^\times)\bks F^+}  \one_{\widehat{\OO_{E,U}}\cap \OO_{E,U,p}^{\times}} (a^{-1}t )  \\
&\times  \one[\frakN(a)yu\OO_F= d_F^{-1}]\,\ex_\infty(iy_\infty u\frakN(t))\,  \ex(x u \frakN(t)) 
\end{align*}
where we have made the change of variable $a\to ay^{\infty}$.

Now we make the substitution $u\frakN(t)=\xi$ and observe that the contribution to the $\xi^\textrm{th}$ term  is equal to zero if $(\xi y d_{F},p)\neq 1$, and otherwise it equals $\baar{\W}(a)$ times the cardinality of the set
\begin{align*}
R_{a^{-1}}(\xi,y)= \Big\{  &(t,u)\in\OO_{E}\times F^+\,|\, t\in U_{p}, u\frakN(t)=\xi, \frakN(t/a)\OO_F=\xi y d_F \Big\}\Big/\frakN\left(\OO_{E,U}^\times\right),
\end{align*}
which admits a surjection  $\pi:(t,u)\mapsto a^{-1}t\OO_{E}$ to the set $\mathbf{r}_{a^{-1}}(\xi y d_F)$ of  ideals $\mathfrak{b}\subset \OO_{E}$ in the  $U$-class $a^{-1}$, whose norm is $ \frakN(\mathfrak{b})=\xi y d_F$. 
The fibres of $\pi$ are in bijection with $\OO_{E,U}^\times/\frakN(\OO_{E,U}^\times)$ which has cardinality $w_U$. We deduce  the following description of the Fourier coefficients of $\Theta(\W)$. 
\begin{prop}\label{prop:fouriertheta}
 The series $\Theta(\W)$ belongs to $S_{1}(K_{1}(\Delta(\W)),\eps\baar{\W}|_{F_{\A}^{\times}})$, where $\Delta(\W)=\Delta\frakN(\mathfrak{f}(\W))$. Its Fourier coefficients are given by
$$a(\Theta(\W),m)=\sum_{\substack{\mathfrak{b}\subset\OO_{E}\\ \frakN(\mathfrak{b})=m}}\W(\mathfrak{b})=r_{\W}(m)$$
for $(m,p)=1$ and vanish for $(m,p)\neq 1$.
\end{prop}

\begin{coro}\label{theo:thetameasure} The functional $\Theta$ of Definition \ref{deftheta} is  a measure on $\calG'$ with values  in ${\bf S}_{1}(K_{1}(\Delta), \eps)$, defined over $\Q_{p}$. 
\end{coro}
\begin{proof} The distributional  property is obvious from the construction or can be  seen from the $q$-expansion given above, from which boundedness is also clear. Cf.  also \cite[Theorem 6.2 ]{HT}, where a slightly different theta measure is constructed.
\end{proof}

\begin{lemm}\label{tauW} Assume that $(D_{E}, D_{F}p)=1$.
The theta series admits a functional equation
$$W_{\Delta(\W)} \Theta(\W)=(-i)^{[F:\Q]}\W(d_{F}^{(p)})\baar{\W}(\mathfrak{D}_{E})\tau(\baar{\W})\Theta(\baar{\W})$$
where $\mathfrak{D}_{E}$ is the relative different, $d_{F}^{(p)}$ is the prime-to-$p$ factor of the different,   and $\tau(\baar{\W})=\prod_{v\vert p}\tau(\baar{\W}_{v})$ with
$$\tau(\baar{\W}_{v})=|\pi_{v}|^{-c/2}\int_{E_{v}^{\times}} \baar{\W}_{v}(h_{v})\e_{v}(-\Tr_{E_{v}/F_{v}}(h_{v}))\, dh_{v}$$
if the relative norm of the conductor of $\baar{\W}_{v}$ is $\pi_{v}^{c}\OO_{F,v}$.
\end{lemm}
\begin{proof}
 %fixme: forse expunge se abbiamo problems con W(d_{F}); nel caso riscrivi la part A-L solo per il caso  trivial central character: this is also needed for the part when we say that anticyclotomically you have zero=zero [that part should also be revised!]
Let
\begin{align*}
\phi_{\W}(g,t,u)&=\int_{H(F)\bks H(\A)} \W({h^{-1}})r(g,h)\phi(t,u)\, dh\\
\phi'_{\W}(g,t,u)&=\eps\W(\pi_{\Delta(\W)})\int_{H(F)\bks H(\A)} \W({h^{-1}})r(gW_{\Delta(\W)},h)\phi(t,u)\, dh
\end{align*}
for $(t,u)\in E_{\A}\times F_{\A}^{\times}$. The behaviour in $g$ is through the Weil representation. 

Then we have 
\begin{align*}
W_{\Delta(\W)}\Theta(\W)(g)&=\eps\W(\det(g))\sum_{(t,u)\in E\times{F^{\times}}} \phi'_{\W}(g, t,u)\\
\Theta(\baar{\W})(g) &= \sum_{(t,u)\in E\times{F^{\times}}} \phi_{\baar{\W}}(g, t,u)
\end{align*}
so that the lemma follows if we show that for all $(t,u)\in E_{\A}\times F_{\A}^{\times}$
\begin{gather}\label{bibi}
\eps\W(\det(g))\phi'_{\W}(g,t,u)=(-i)^{[F:\Q]}\baar{\W}(\mathfrak{D}_{E})\tau(\baar{\W})\eps\W(u)\phi_{\baar{\W}}(g,\baar{t}, u)
\end{gather}
where $\baar{t}$ is the conjugate of $t$ under the nontrivial automorphism of $E$ over $F$.  We write 
$$\wtil{\tau}(\baar{\W})=(-i)^{[F:\Q]}\W(d_{F}^{(p)})\baar{\W}(\mathfrak{D}_{E})\tau(\baar{\W})$$
for short.

We claim that it suffices to prove \eqref{bibi}  for $g=1$. Indeed it is clear that this implies the same result for all $g\in {\bf SL}_{2}(\A)$ by acting via the Weil representation on both  sides (viewed as functions of $(t,u)$). Then it suffices to verify it for the elements of the form $d(y)=\smalltwomat 1 {}{}y$:
\begin{multline*}
\eps(y)\W(y)r(d(y))\phi'_{\W}(1,t,u)=\wtil{\tau}(\baar{\W})\eps(y)\W(y)r(d(y))[\eps\W(u)\phi_{\baar{\W}}(1, \baar{t}, u)]\\
=\wtil{\tau}(\baar{\W})\eps\W(y)\eps\W(y^{-1}u)r(d(y))\phi_{\baar{\W}}(1,\baar{t},u)=\wtil{\tau}(\baar{\W})\eps(u)\W(u)\phi_{\baar{\W}}(d(y),\baar{t},u).
\end{multline*}

We now prove \eqref{bibi} for $g=1$, thus dropping $g$ from the notation. We can write 
\begin{multline*}
\phi'_{\W}(t,u)=\int_{H(F)\bks H(\A^{p\Delta})} \W({h_{0}^{-1}})r(1,h)\phi^{p\Delta}(t,u)\, dh_{0}\\
{\prod_{v|p\Delta}}\W(\pi_{v}^{c_{v}})\int_{H(F_{v})}\W(h_{v}^{-1})r(W_{\pi_{v}^{c_{v}}},1)\phi(h^{-1}_{v}t, \nu(h_{v})u)\, dh_{v}
\end{multline*}
were $c_{v}$ is the appropriate exponent.  We can rewrite this as 
$$\phi'_{\W}(t,u)=\phi'{}^{\Delta p}_{\W}(t,u)\prod_{v\vert \Delta p}\phi'_{\W, p}(t,u)$$ with obvious notation. A similar factorisation holds for $\phi_{\W}(t,u)$.

For $v\nmid \Delta p$ we  have, by the explicit description of $\phi_{v}$ (dropping the subscripts $v$):
\begin{multline*}r(h)\phi(t,u)=\phi(h^{-1}t, \nu(h)u)=\phi(\pi_{d_{F}}u\baar{h} t, \nu(h)u)\\
=\phi(\pi_{d_{F}}uh\baar{t}, \nu(h)^{-1}u^{-1}\pi_{d_{F}}^{-2})
= r((\pi_{d_{F}}uh)^{-1})\phi(\baar{t}, u).
\end{multline*}
A change of variable and integration over $H(F)\bks H(\A^{\Delta p})$ then gives
\begin{align}\label{bibitame}
\phi'{}^{\Delta p}_{\W}(t,u)=\eps^{\Delta p}(ud_{F})\W^{\Delta p}(ud_{F})\phi^{\Delta p}_{\baar{\W}}(\baar{t}, u).
\end{align}

For $v|\Delta$ we have by  \eqref{bibi'delta} below and the previous argument
\begin{multline*}
\eps(\pi) r(W_{\pi},h) (\pi)\phi(t,u)=\eps(u)\kappa(v)\phi(h^{-1} t\pi_{\mathfrak{D}}, \pi^{-1} \nu(h)u)
=\eps(u)\kappa(v) r(u^{-1}\pi_{d_{F}^{-1}}h^{-1}\pi_{\mathfrak{D}}) \phi(\baar{t}, u)
\end{multline*}
where $\pi_{\mathfrak{D}}\in \OO_{E,v}^{\times}$ is a generator of the local relative different of $E_{v}/F_{v}$. After change of variable and integration, we obtain 
\begin{align}\label{bibidelta}
\phi'_{\W,v}(t,u)=\kappa(v)\eps_{v}(u)\W_{v}(ud_{F})\baar{\W}_{v}(\mathfrak{D})\phi_{\baar{\W},v}(\baar{t}, u).
\end{align}

For $v|p$ we have
\begin{multline*}
\W(\pi^{c})\int_{H(F_{v})} \W(h^{-1}) r(h, w_{\pi^{c}})\phi(t, u) \, d^{\times}h\\
=|\pi|^{-c/2}\int_{E^{\times}}\int_{E}\baar{\W}(\pi^{-c}h) 
\e(-\pi^{-c}u\nu(h)\Tr(h^{-1}t\baar{\xi})) 	\phi(\xi, \pi^{-c}\nu(h)u) \, d\xi\, d^{\times}h
\end{multline*}
Using the fact that $\phi(\xi,u)\,d\xi=\phi(\xi,u)\, d^{\times }\xi$, and a change of variables $\zeta= \pi^{-c}uh\xi\baar{t}$, this equals 
$$\W(u)\tau(\baar{\W}) \int_{E^{\times}} \W(\xi\baar{t}) \phi(\xi, \nu(t\xi)u)		\, d^{\times}\xi$$
after integration, where the new second argument in $\phi$ gives the condition for the integral in $d\zeta$ to be nonzero. We observe that $\phi(\xi)=\phi(\xi^{-1})$ so that with the  the new variable $h'=\xi\baar{t}$, and reintroducing $v$ in the notation,  this can be rewritten as
$$\W_{v}(u)\tau(\baar{\W}_{v})\int_{E^{\times}_{v}}\baar{\W}_{v}(h'_{v})^{-1}\phi({h'}^{-1}\baar{t}, \nu(h'_{v})u)\, d^{\times}h'_{v}$$
so that
\begin{align}\label{bibip}
\phi'_{\W,v}(t,u)=\eps_{v}(d_{F})\eps_{v}(u)\W_{v}(u)\tau(\baar{\W}_{v})\phi_{\baar{\W},v}(\baar{t},u).
\end{align}

Putting together  \eqref{bibitame}, \eqref{bibidelta}, \eqref{bibip} and 
using the formula $\prod_{v|\Delta}\kappa(v)=(-i)^{[F:\Q]}\eps(d_{F})$ from \cite[p. 127]{shouwu},\footnote{Recall that our $\kappa(v)$ are the inverses of the $\kappa(v)$ of \emph{loc. cit.}.} we obtain \eqref{bibi} as desired.
\end{proof}

\subsection{Fourier expansion of the theta measure / II}\label{2.5}

For later use in computing the trace of the convolution of the theta measure with the Eisenstein measure (defined below), we need to consider the expansion of $\Theta(\W)^{(\delta)}(g)=\Theta(\W)(gW_\delta)$ for $g=\begin{pmatrix}y & x \\& 1\end{pmatrix}$; for such a $g$ we have 
$$\begin{pmatrix}y & x \\& 1\end{pmatrix}W_\delta= \begin{pmatrix}1 & x \\& 1\end{pmatrix} \begin{pmatrix}
y&\\& \pi_\delta \end{pmatrix} w_\delta$$
where $\pi_\delta$ is an id\`{e}le with components $\pi_v$ at $v\vert \delta$ and~$1$ everywhere else.  
Here $\pi_v$ is a uniformiser chosen to satisfy $\eps(\pi_v)=1$.

The modular form  $\Theta(\W)^\delta$  can be expanded in the same way as in \S\ref{sec:expandtheta}, except that for $v\vert \delta$ we need to replace $\phi_v(t,u)=\one_{\OO_{E,v}}(t)\one_{d_F^{-1,\times}}(u)$ by
\begin{equation}\label{bibi'delta}		
\begin{split}
W_\delta\phi_v(t,u)=&\eps_{v}(\pi_{v})\begin{pmatrix}1&\\&\pi_v\end{pmatrix}\gamma(u)\widehat{\one_{\OO_{E,v}}}(t)\one_{d_{F,v}^{-1,\times}}(u)\\
= &\eps_v(u)\kappa(v) \one_{\frakD_{v}^{-1}}(t)\one_{d_F^{-1,\times}}(\pi_\delta^{-1} u)
\end{split}
\end{equation}
Here  recall that $\frakD$ is the relative  different of $E/F$;  and that $w$ acts as Fourier transform in $t$ with respect to the quadratic form associated with $u\frakN$, with the normalising constant  $\gamma(u)=\gamma(E_v,u\frakN)$ as described in \eqref{gamma}. 

The computation of the expansion can then be  performed exactly as in \S\ref{sec:expandtheta}. We omit the details but indicate that the relevant substitution is now $a\to \pi_{\frakd}ay$, where $\frakd$ is an ideal of $\OO_{E}$ of norm $\delta$ and $\pi_{\frakd}\in\widehat{\OO_{E}}$ is a generator with components equal to~$1$ away from $\frakd$.

\begin{prop}\label{thetafw} The Whittaker-Fourier coefficients of the series $\Theta(\W)^{(\delta)}$ are given by
\begin{gather*}
\tilde{a}(\Theta(\W)^{(\delta)}, y)= \eps\W(y)|y|^{1/2} \kappa(\delta)\W(\frakd)\eps_{\delta}(y) r_{\W}(yd_{F}), 
\end{gather*}
where $\kappa(\delta)=\prod_{v\vert \delta}\kappa(v)$.
\end{prop}

\section{Eisenstein measure} In this section we construct a measure (cf. \S\ref{measure}) valued in Eisenstein series of weight one, and compute its Fourier expansion.
\subsection{Eisenstein series}
Let $k$ be a positive integer, $M$ an ideal 
of $\OO_F$,
 and $\varphi\colon F_\A^\times/F^\times\to \C^{\times}$ a finite order character of conductor dividing $M$ satisfying $\varphi_v(-1)=(-1)^k$ for $v\vert \infty$. Let 
\begin{align}\label{dirichlet}
L^{M}(s,\varphi)=\sum_{(m,M)=1} \varphi(m)\N(m)^{-s}
\end{align}
where the sum runs over all nonzero ideals of $\OO_{F}$. 

Let $B\subset \GL_2$ be the Borel subgroup of upper triangular matrices; recall the notation from  \S\ref{sec:hmf}, and the Iwasawa decomposition \eqref{iwadec}; the decomposition is not unique but the ideal  of $\widehat{\OO}_F$ generated by the lower left entry of the  $K_0(1)$-component is well-defined.

For $s\in\C$, define a function $H_{k,s}(g, \varphi)$ on $\GL_2(\A_F)$ by
\begin{align*}
H_{k,s}\left(g=qur(\theta); \varphi\right)=& \begin{cases}
 \left|{y_1\over y_2}\right|^s\varphi(y_1 a)\e_\infty(k\theta)&\textrm{if } u=\begin{pmatrix} a & b\\c & d \end{pmatrix} \in K_0(M)\\
0 &\textrm{if } u\in K_0(1)\setminus K_0(M_0).
\end{cases}
\end{align*}
where we have written $g=qur(\theta)$ with $q=\begin{pmatrix} y_1 & x\\ &y_2 \end{pmatrix}\in B(\A_F)$, $u\in K_0(1)$, $ r(\theta)\in~K_\infty$.

We define two \textbf{Eisenstein series}
\begin{align*} E^{M}_k(g,s;\varphi)= & L^{M}(2s,\varphi)\sum_{\gamma\in B(F)\bks \GL_{2}(F)} H_{k,s}\left(\gamma g; \varphi\right),\\
 \widetilde{E}_{k}^{M}(g,s;\varphi)=  &W_{M} E_k^{M}(g,s;\varphi)=\varphi^{-1}(\det g\pi_{M}) E_k^{M}(gW_{M},s;\varphi)
\end{align*}
which are absolutely convergent for $\Re s>1$ and continue analytically for all $s$ to (non-holomorphic) automorphic forms of level $M$, 
parallel weight $k$ and character $\varphi$ (for $E$) and $\varphi^{-1}$ (for $\wtil{E}$).
Here $W_M$ is as in \eqref{WMop}. The superscript  $M$ will be omitted from the notation when its value is clear from context.

\subsection{Fourier expansion of the Eisenstein measure}\label{3.2} We specialise to the case where $k$ is odd, 
$M=\Delta P$ with  $(\Delta, P)=1$,  $\varphi=\eps \phi$ with $\eps=\eps_{E/F}$ and $\phi$  a 
  character of conductor dividing $P$, trivial at infinity (in particular we have $\varphi_{v}(-1)=\eps_{v}(-1)\phi_{v}(-1)=-1$ as required).  We assume that $\Delta$ is squarefree. For $\delta\vert \Delta$ we compute\footnote{Cf. \cite[\S\S\ 3.5, 6.2 ]{shouwu}.} the Whittaker coefficients (c f. \S\ref{sec:fourier}; we suppress  $\varphi$, $M$ and $k$ from the notation) of $\widetilde{E}^{(\delta)}$;
$$c_s^\delta(\alpha, y)=D_{F}^{-1/2}\int_{\A_F/F} \widetilde{E}\left(\begin{pmatrix}
y & x\\ & 1\end{pmatrix} W_\delta,s\right) \ex(-\alpha x)\, dx $$
for $\alpha\in F$ and $\delta$ dividing $\Delta$;  since $c_{s}(\alpha,y)=c_{s}(1, \alpha y)$ for $\alpha\neq 0$, we can restrict to $\alpha=0$ or~$1$. The choice of uniformisers  $\pi_{v}$ at $v\vert \delta$ implicit in the above formula is made so that $\eps(\pi_{v})=1$ to save some notation.

\begin{prop}\label{prop:eis} In the case just described, the Whittaker coefficients $c_s^\delta(\alpha,y)$ of the Eisenstein series  $\widetilde{E}^{(\delta)}_k(g,s;\varphi)$ are given by 
\begin{align*}
c_s^\delta(0, y)=&
	\begin{cases} {1\over D_{F}^{1/2}\N(\Delta P)^{s}} \eps\phi(y)|y|^{1-s}V_{k, s}(0)^{[F:\Q]} L^{(P)}(2s-1, \eps\phi) 			& {\rm if\ \ } \delta=1\\
				0 & {\rm if\ \ } \delta\neq 1,
	\end{cases}\\
	&\\
c_s^\delta(1,y)=& { { \ \N(\delta)^{s-1/2}  } \over {D_F^{1/2} \N(\Delta P)^{s}}} \eps {\phi} (y) |y|^{1-s} \kappa(\delta) \phi(\delta) \eps_{\delta}(y)\phi_{\delta}(y^{\infty}d_{F})|y\pi_{\delta} d_{F}|_{\delta}^{2s-1}		 \sigma_{k,s,\eps\phi}(y) 
 \\
&\textrm{if $yd_F$ is integral, and $c_s^\delta(1,y)=0$ otherwise,}
\end{align*}
where  $\kappa(\delta)=\prod_{v\vert \delta}\kappa(v)$ with $\kappa(v)$ as in \eqref{gamma} and  
$$\sigma_{k,s,\varphi}(y)=\prod_{v\nmid\Delta M\infty} \sum_{n=0}^{v(yd_F)} \varphi_v(\pi_v)^n|\pi_v\vert ^{n(2s-1)}\prod_{v\vert \infty} V_{k,s}(y_v)$$
with $$V_{k,s}(y)=\int_\R {e^{-2\pi i y x}\over (x^2+1)^{s-k/2} (x+i)^{k}}\,dx.$$
\end{prop}

\begin{proof}
We use the Bruhat decomposition 
$$\GL_2(F)=B(F) \coprod B(F)wN(F)$$
with $w=\twomat{}{-1}{1}{}$ and the unipotent subgroup $N(F)\cong F$ via $N(F)\ni\begin{pmatrix} 1 &x\\ &1\end{pmatrix}\mapsfrom x\in F$, to get
\begin{align*} \eps\phi(y)\phi(\pi_{M/\delta})c_s^\delta(\alpha,y)&=L(2s,\varphi)D_F^{-1/2}\int_{\A_F/F} H_s\left(\begin{pmatrix}
y & x\\ & 1\end{pmatrix}W_{M/\delta}\right) \ex(-\alpha x)\, dx\\
&+L(2s,\varphi)D_F^{-1/2}\int_{\A_F}  H_s\left(w\begin{pmatrix}
y & x\\ & 1\end{pmatrix}W_{M/\delta}\right) \ex(-\alpha x)\, dx.
\end{align*}

At any place $v\vert M/\delta$, we have the decomposition
$$\begin{pmatrix}
y_v& x_v\\ & 1\end{pmatrix}W_{M/\delta,v}=\begin{pmatrix}y_v&\pi_v x_v\\&\pi_v\end{pmatrix}\twomat {}{1} {-1}{}$$
so that the first summand is always zero.

For the second integral, we use the identity
$$w\begin{pmatrix}
y & x\\& 1
\end{pmatrix}= \begin{pmatrix}
1 & \\ & y
\end{pmatrix} \begin{pmatrix}
&-1 \\1 & xy^{-1}
\end{pmatrix}$$
and the substitution $x\to xy$ to get
$$\int_{\A_F} H_s\left(w\begin{pmatrix}
y& x\\ & 1
\end{pmatrix}W_{M/\delta} \right) \ex(-\alpha x)\, dx=|y|^{1-s} \prod_v V_s^{M/\delta}(\alpha_v y_v)$$
where for $y\in F_v$,
\begin{gather}\label{eqn:Vs} V_s^M(y)=\int_{F_v} H_s\left(\begin{pmatrix} & -1 \\ 1 & x\end{pmatrix} W_{M,v}\right) \ex(-xy)\, dx.\end{gather}
\paragraph*{Archimedean places} As in \cite[Proposition 3.5.2]{shouwu}. 
\subsubsection*{Nonarchimedean places $v\nmid  M/ \delta$} If $v$ is a finite place, we have $\begin{pmatrix} & -1\\1 & x \end{pmatrix}\in \GL_2(\OO_{F,v})$ if $x\in\OO_{F,v}$, and otherwise we have the decomposition 
$$\begin{pmatrix} & -1 \\ 1 & x \end{pmatrix}= 
\begin{pmatrix}x^{-1} & -1 \\  & x \end{pmatrix}
\begin{pmatrix} 1&  \\
x^{-1} & 1 \end{pmatrix}.$$ 
Therefore 
\begin{gather}\label{Hs} H_{s,v}\left(\begin{pmatrix} & -1 \\ 1 & x \end{pmatrix}\right)=\begin{cases}
\baar{\varphi}_v(x)|x|^{-2s}& \textrm{if }v(x)\leq -1; \\
1 &\textrm{if } v\nmid M, v(x)\geq 0;\\
0 &\textrm{if } v\vert \delta, v(x)\geq 0.
\end{cases}
\end{gather}

\paragraph*{The case $v\nmid M$}
We deduce that
\begin{align*}
V_s^{M/\delta}(y)&=\int_{\OO_{F,v}} \ex(-xy)\,dx +\sum_{n\geq 1} \int_{\OO_{F,v}^\times} \baar{\varphi}_v(x\pi_v^{-n})|x\pi_v^{-n}|^{-2s} \ex(-xy\pi_v^{-n})\,d(\pi_v^{-n}x)\\
&=\one[y\in d_F^{-1}]+ \sum_{n\geq 1} \varphi_v(\pi_v)^n|\pi_v\vert ^{n(2s-1)}\int_{\OO_{F,v}^\times}\ex(-xy\pi_v^{-n})\, dx.
\end{align*}
The integral evaluates to $1-|\pi_v\vert $ if $v(yd_F)\geq n$, to $-|\pi_v\vert $ if $v(yd_F)= n-1$, and to zero otherwise. Therefore we have $V_s^M(y)=0$ unless $v(yd_F)\geq 0$ in which case if $y\neq 0$
\begin{align*} V_s^{M/\delta}(y)&= 1 +(1-|\pi_{v}|) \sum_{n=1}^{v(yd_F)} \left(\varphi_v(\pi_v)|\pi_v\vert ^{2s-1}\right)^n-|\pi_{v}|(\varphi_v(\pi_v)|\pi_v\vert ^{2s-1})^{v(yd_f)+1}\\
&= \left(1-\varphi_v(\pi_v)|\pi_v\vert ^{2s}\right)\sum_{n=0}^{v(yd_F)} \varphi_v(\pi_v)^n|\pi_v\vert ^{n(2s-1)}\\
&= L_v(2s,\varphi)^{-1} \sum_{n=0}^{v(yd_F)} \varphi_v(\pi_v)^n|\pi_v\vert ^{n(2s-1)};
\end{align*}
whereas for $y=0$, we have
\begin{align*} V_s^{M/\delta}(0)&= 1 +(1-|\pi_{v}|) \sum_{n=1}^{\infty} \left(\varphi_v(\pi_v)|\pi_v\vert ^{2s-1}\right)^n\\
&= 1+(1-|\pi_{v}|)(1-\varphi_{v}(\pi_{v})|\pi_{v}|^{2s-1})^{-1}(1-\varphi_{v}(\pi_{v})|\pi_{v}|^{2s})\\
&=L_{v}(2s, \varphi)^{-1}L_{v}(2s-1 ,\varphi)
\end{align*}
\paragraph*{The case $v\vert \delta$} Again by \eqref{Hs} we find
\begin{align*}
V_s^{M/\delta} (y)=\sum_{n\geq 1} \int_{\OO_{F,v}^\times} \baar{\varphi}_v(x\pi_v^{-n})|x\pi_v^{-n}|^{-(2s-1)} \ex(-xy\pi_v^{-n})\,dx.
\end{align*}
All the integrals vanish except the one with $n=v(yd_{F})+1$ which gives 
 $$\eps_v(y\pi_v^n){\phi_v}(y\pi_{d_{F},v}\pi_v)|y\pi_{d_{F},v}\pi_v\vert ^{2s-1}|\pi_v\vert ^{1/2}\kappa(v);$$
 therefore we have\footnote{Recall that we always choose $\pi_v$ so that $\eps_v(\pi_v)=1$.}
$$V_s^{M/\delta}(y)=\eps_v(y)\phi_v(y\pi_{d_{F},v}\pi_v)|y\pi_{d_{F},v}\pi_v\vert ^{2s-1}|\pi_v\vert ^{1/2}\kappa(v)$$
if $y\neq 0$ and $v(yd_{F})\geq 0$ and $V_s(y)=0$ otherwise. In particular, we see that if $\delta\neq 1$ then $V_{s}(0)=c_{s}(0,y)=0$. 

\paragraph*{Places $v\vert M/\delta$}
For  $ w \twomat 1x{}1\twomat {}{1}{-\pi_v^{v(M)}}{}= \twomat {\pi_v^{v(M)}} {}{-x\pi_v^{v(M)}} {1}$ we have the decompositions 
\begin{align*}
\twomat {\pi_v^{v(M)}} {}{-x\pi_v^{v(M)}}{1}=&\twomat {-\pi_v^{v(M)}} {}{} {1}\twomat {-1} {}{-x\pi_v^{v(M)}}{1}\\
=&\twomat {x^{-1}}{-\pi_v^{v(M)}} {}{x\pi_v^{v(M)}}\twomat{}1{-1}{x^{-1}\pi_v^{-v(M)}} :
\end{align*}
for $v(x)\geq 0$ we use the first one to find
$$H_s\left(\twomat {-\pi_v^{v(M)}} {}{x\pi_v^{v(M)}} {-1}\right)=\varphi_{v}(\pi_{v})^{v(M)}|\pi_{v}^{v(M)}|^{s};$$
for $v(x)<0$ the second decomposition shows that the integrand vanishes. We conclude that 
$$V^{M/\delta}_s(y)=\begin{cases}
\varphi_{v}(\pi_{v})^{v(M) }  |\pi_{v}^{v(M)}|^{s} &\textrm{ if $v(yd_{F})\geq 0$}\\\
 0 &\textrm{ otherwise}.
\end{cases}$$
The final formula follows from these computations. 
\end{proof}

We specialise to the case $s=1/2$,  and consider the rescaled holomorphic Eisenstein series:\footnote{Notice that these series do not depend on the ideal $P$ but only on its support.}
\begin{align}
\nonumber {\bf E}_{k,\eps\phi}^{\Delta P}(g)&={D_F^{1/2} \N( \Delta P)^{1/2}\over{ (-2\pi i )^{[F:\Q]} }}E_{k}^{\Delta P}(g,1/2; \eps\phi),\\
 \wtil{\bf E}^{\Delta P}_{k,\eps\phi}(g)&={D_F^{1/2} \N( \Delta P)^{1/2}\over{ (-2\pi i )^{[F:\Q]} }} \wtil{E}^{\Delta P}_{k}(g,1/2; \eps\phi). 
\label{tileis}
\end{align}

We further specialise to the case $k=1$.
\begin{coro}\label{eisfw} The Eisenstein series  $\wtil{\bf E}^{\Delta P}_{1,\eps\phi}$
belongs to $M_{1}(K_{1}(\Delta P), \eps\phi^{-1})$. The Whittaker-Fourier coefficients of $\wtil{\bf E}_{\eps\phi}^{(\delta)}=\wtil{\bf E}^{\Delta P, (\delta)}_{1,\eps\phi}$ for $\delta|\Delta$ are given by 
$$\tilde{a}^{0}( \wtil{\bf E}_{\eps\phi}^{(\delta)}, y)= \eps\phi(y)|y|^{1/2}{L^{(p)}(0, \eps\phi)\over 2^{g}}$$
if $\delta=1$ and $\tilde{a}^{0}( \wtil{\bf E}_{\eps\phi}^{(\delta)}, y)=0$ otherwise; and 
$$c^\delta(y)=\tilde{a}( \wtil{\bf E}_{\eps\phi}^{(\delta)}, y)=\eps\phi(y)|y|^{1/2} \kappa(\delta) \phi(\delta) \eps_\delta\phi_\delta(y^{\infty}d_{F}){\sigma}_{\eps\phi}(y^{\infty} d_{F}),$$
where for any integral ideal $m$ of $\OO_{F}[\Delta^{-1}P^{-1}]$,
$${\sigma}_{\eps\phi}(m)= \sum_{d|m} \eps\phi(d),$$
the sum likewise running over integral ideals of  $\OO_{F}[\Delta^{-1}P^{-1}]$.
\end{coro}
(If $m$ is an integral ideal of $\OO_{F}$ prime to $P$, then $\sigma_{\eps \one}(m)=r(m)$.) 
\begin{proof} This follows from Proposition \ref{prop:eis} together with the  evaluation 
$$V_{1,1/2}(t)=\begin{cases}
 0 & {\rm if\ } t<0\\
 -\pi i &{\rm if\ } t=0\\
 -2\pi i e^{-2\pi t} & {\rm if\ } t>0.
\end{cases}$$
which can be found in \cite[Proposition IV.3.3 (a), (d)]{GZ} (for the case $t=0$, this is deduced from (a) of \emph{loc. cit.} using $\lim_{s\to 0}{\Gamma(2s)\over \Gamma(s)}=1/2$).
\end{proof}

\begin{defi} Let $F_{\infty}'$ be the maximal abelian extension of $F$ unramified outside $p$, and let $\calG'_{F}=\Gal(F_{\infty}'/F)$. We define the {\bf Eisenstein (pseudo-) measure}\footnote{We \emph{do not} need to assume that $\Delta$ is squarefree when making the definition. See  after the definition for the meaning of the term pseudo-measure.} $\wtil{\bf E}_{\eps}$ on $\calG_{F}'$ 
by
$$\wtil{\bf E}_{\eps}(\phi)=\wtil{\bf E}_{\eps\phi}^{\Delta P}={D_F^{1/2} \N( \Delta P)^{1/2}\over{ (-2\pi i )^{g} }}\, {\wtil{E}^{\Delta P}_{\eps\phi}}$$
for any character $\phi$ of $\calG'_{F}$ of conductor dividing $P$ (it does not depend on the choice of $P$ once we require $P$ to satisfy $v\vert P\leftrightarrow v\vert p$). We denote with the same name the distribution induced on the group $\calG'$ of \S\ref{sec:calG} by 
$$\wtil{\bf E}_{\eps}(\W)=\wtil{\bf E}_{\eps}(\W|_{F_{\A}^{\times}}).$$

It has values in ${\bf M}_{1}(K_{1}(N\Delta ),\eps)$ and is defined over $\Q_{p}$. 
\end{defi}

To prove the soundness  of the definition, it is easy to see that the nonzero Fourier coefficients interpolate to a measure on $\calG_{F}'$, that is an element of $\Z_{p}\llbracket\calG_{F}'\rrbracket$. The $L$-values giving the constant term interpolate to the Deligne--Ribet $p$-adic $L$-function \cite{DR}; it is a pseudo-measure in the sense of Serre \cite{serre}, that is an element of the total quotient ring of $\Z_{p}\llbracket\calG_{F}'\rrbracket$ with denominators of a particularly simple form.

\section{The $p$-adic $L$-function}\label{sec:L}

\subsection{Rankin--Selberg convolution}\label{sec:rs} Let $f$, $g$ be modular forms of  
common level $M$, weights $k_{f}$, $k_{g}$, and characters $\psi_{f}$, $\psi_{g}$ respectively. We define a normalised Dirichlet series
 $$D^{M}(f,g,s)=L^{M}(2s-1,\psi_{f}\psi_{g})  \sum_{m}a(f,m)a(g,m)\N m^{-s},$$ 
 where the imprimitve $L$-function  $L^{M}(s,\varphi)$ of a Hecke character $\varphi$ of conductor dividing $M$ is as in \eqref{dirichlet}.

 When $f$ and $g$ are primitive forms of level $N_{f}$, $N_{g}$ (that is, normalised new eigenforms at those levels),  for a prime $\wp\nmid N_{f}$ denote by $\gamma_{\wp}^{(1)}(f)$, $\gamma_{\wp}^{(2)}{(f)}$ the two roots of the $\wp^{\mathrm{th}}$ Hecke polynomial of~$f$
$$P_{\wp,f}(X)=X^2-a(f,\wp) X+\psi_{f}(\wp)\N \wp^{k_{f}-1},$$
and  by  $\gamma_{\wp}^{(1)}(g)$, $\gamma_{\wp}^{(2)}{(g)}$  the analogous quantities for~$g$.
 Then the degree four Rankin--Selberg $L$-function $L(f\times g,s)$ with unramified Euler factors at~$\wp$ given by
 $$\prod_{i,j=1}^{2}\left(1-\gamma_{\wp}^{(i)}(f)\gamma^{(j)}_{\wp}(g)\N\wp^{-s}\right)^{-1}
$$
equals   the above Dirichlet series  
\begin{align*}
L(f\times g, s)=D^{N_{f}N_{g}}(f,g,s)
\end{align*}
if $N_{f}$ and $N_{g}$ are coprime.
 
Suppose now for simplicity that $k_{f}=2$, $k_{g}=1$, and $f$ is a cusp form (not necessarily primitive). The Rankin--Selberg convolution method\footnote{See \cite{shimura} or   \cite[Ch. V]{jacquet} for general treatments; our setting and normalisations are the same as in  \cite[Lemma 6.1.3]{shouwu} (where $g$ is a specific form, but the same calculation works in general to prove \eqref{eqn:rs}).} gives
\begin{align}\label{eqn:rs}
\langle f^{\rho}, g E^{M}_{1}(s;\psi_{f}\psi_{g})\rangle_{M}=D_{F}^{s+1}\left[{\Gamma(s+1/2)\over (4\pi)^{s+1/2} } \right]^{[F:\Q]} D^{M}(f\times g, s+1/2),
\end{align}
where $\langle\ , \ \rangle_{M}$ is the Petersson inner product \eqref{pet}.

\subsection{Convoluted measure and the $p$-adic $L$-function in the ordinary case} 
Consider the convolution pseudo-measure $\Theta*\widetilde{\bf E}_{\eps, N}$ on $\calG'$ defined by 
$\Theta*\widetilde{\bf E}_{\eps, N}(\W)= \Theta(\W)\widetilde{\bf E}_{\eps,N}(\baar{\W})$ for any character $\W\colon \calG'\to \Z_{p}^{\times}$,  where  $\widetilde{\bf E}_{\eps, N}=[N]\widetilde{\bf E}_{\eps}$. We deduce from it the (pseudo-)measure
\begin{gather}\label{Phidef}
\Phi(\W)={\rm Tr}_{\Delta}[\Theta * \widetilde{\bf E}_{\eps, N}({\baar{\W}})]={\rm Tr}_{\Delta}[\Theta(\W)\cdot[N]\widetilde{\bf E}_{\eps}(\baar{\W})] 
\end{gather}
on $\calG'$, which is a kind of $p$-adic   kernel of the Rankin--Selberg $L$-function as will be made precise below. It is valued in ${\bf M}_{2}(K_{0}(N),\C_{p})$. Notice that while $\Phi(\W)$, like $\widetilde{\bf E}_{\eps,N}$, is not a measure, we can see that, for any $\wp\vert p$,
$$U_{\wp}\Phi(\W)$$
is. Indeed its Fourier coefficients are the Fourier coefficients of $\Phi(\W)$ at ideals $m$ divisible by $\wp$, hence sums of coefficients of the   theta and Eisenstein series at pairs of ideals $(m_{1}=nm, m_{2}=(1-n)m)$ for some $n\in F$; since the coefficients of the theta series are zero at ideals $m_{1}$ divisible by $\wp$, only those pairs $(m_{1},m_{2})$ with $m_{1}, m_{2}$ both prime to $\wp$ contribute. In particular, the constant term of the Eisenstein series does not contribute to the Fourier expansion of $U_{\wp}\Phi$, which therefore belongs to $\Z_{p}\llbracket \calG'\rrbracket\otimes {\bf S}_{2}(K_{0}(N), \C_{p})$.

Thanks to  this discussion and the identity $\Lf=\alpha_{\wp}^{-1}\Lf\circ U_{\wp}$, the following definition makes sense.
\begin{defi}\label{def:prs} The $p$-adic Rankin--Selberg $L$-function is the element of $\OO_{L}\llbracket\calG\rrbracket\otimes L$ defined by 
$$L_{p}(f_{E}, \W)= D_{F}^{-2 } H_{p}(f) \Lf (\Phi(\W))$$
for any character $\W\colon\calG\to \OO_{L}^{\times}$, where
\begin{gather}\label{Hp}
H_{p}(f)= \prod_{\wp\vert p}\left(1-{1\over\alpha_{\wp}(f)^{2}}\right) \left(1-{\N\wp\over\alpha_{\wp}(f)^{2}}\right).
\end{gather}
\end{defi} 

\subsubsection{Functional equation}
The $p$-adic $L$-function admits a functional equation; we prove it in the case of anticyclotomic characters which is the only one we shall need. 

\begin{prop}\label{prop:few} Suppose that $\W$ is an anticyclotomic character of $\calG$, i.e., $\W|_{F_{\A^{\times}}}=1$. Then there  are  functional equations for the $p$-adic $L$-function
\begin{gather}\label{feq}
L_{p}(f_{E})(\W)=(-1)^{g}\eps(N)L_{p}(\W)
\end{gather}
and  for the analytic kernel
\begin{gather}\label{feq2}
\Phi(\W)=(-1)^{g}\eps(N) \Phi(\W).
\end{gather} 
In particular,  if $\eps(N)=(-1)^{g-1}$, we have 
$$\Phi(\W)=L_{p}(f_{E})(\W)=0.$$
\end{prop}
\begin{proof}
The functional equation for $L_{p}$ is implied by the functional equation for $\Phi$. We prove the latter by comparing the coefficients on both sides. From \eqref{b(m)} below,\footnote{Which does not use the present result. The formula \eqref{b(m)} is stated in the case when the anticyclotomic part $\W^{-}=\one$ but the very same calculation gives the result in general.}
  the coefficients of $\Phi(\W)$ are given by 
\begin{equation*}
b(m)= \sum_{\delta|\Delta}\sum_{\substack{n\in F \\ 0<n<1}} \eps_\delta((n-1)n)    r_{\W^{-}}((1-n)m\delta)\sigma_{\eps\one}(nm/N). 
\end{equation*}
(We use the notation $\one$ for the character of ideals defined by $\one(m)=1$ if $(m, p)=1$ and $\one(m)=0$ otherwise.) We rewrite this as $b(m)=\sum_{\delta, n} b_{\delta, n}(m)$ with, using $\eps_{\delta}(x)=\eps^{\delta}(x)$ for $x\in F^{\times}$ and writing in columns to highlight the factors:
\begin{align*}
b_{\delta,n}(m)
=&\eps_{\delta}(-1) &		=& (-1)^{g}\eps_{\Delta/\delta}(-1)  &\\
\cdot&\eps_{\Delta/\delta}((1-n)m)\eps_{\Delta/\delta}(nm) &		\cdot&\eps_{\Delta/\delta}((1-n)m)\eps_{\Delta/\delta}(nm) &\\
\cdot& \eps^{\Delta}((1-n)m) r_{\W}((1-n)m\delta) &		\cdot & r_{\W}((1-n)m\Delta/\delta)&\\
\cdot&\eps(N)  & & \cdot \eps(N) &\\
\cdot &\eps^{\Delta}(nm/N)  \sigma_{\eps\one}(nm/N)  & 		\cdot & \sigma_{\eps\one}(nm/N) &
= (-1)^{g}\eps(N)b_{\Delta/\delta, n}.
\end{align*}

Here we have used the following facts. In the first line,   $\eps_{\Delta}(-1)=\eps_{\infty}(-1)=(-1)^{g}$. In the third line, we have that $r_{\W}(m)=1$ if $m$ is divisible only by ramified primes in $E$, since in that case $m=\mathfrak{m}^{2}$ is a square and $\W(\mathfrak{m})^{2}=\W(m)=1$ -- this implies $\W(\mathfrak{m})=\pm 1$, hence $\W(\mathfrak{m})=1$ since $\W$, which is a character of $\calG\cong \Z_{p}^{1+g+\delta}$, has values in $1+p\Z_{p}$. Finally, in the third and fifth line one can observe that  if $q=\sigma_{\eps\one}$ or $q=r$, then $\eps^{\Delta}(m)q(m)=q(m)$; indeed this is trivial if $\eps^{\Delta}(m)=1$, while both sides are zero if $\eps^{\Delta}(m)=-1$. 
\end{proof}

\subsection{Interpolation property}\label{sec:pLf} 
We manipulate the definition to show that the $p$-adic $L$-function  $L_{p}(f_{E})(\W)$ of Definition \ref{def:prs} interpolates the special values of the complex Rankin--Selberg $L$-function $L(f_{E},\W,s)=L(f\times \Theta(\W),s)$ defined in the Introduction.

We will need a few technical lemmas.

\begin{lemm}\label{period}  Let $P$ be an ideal of $\OO_{F}$ such that $v\vert P$ if and only if $v\vert p$. We have
$$\langle W_{NP}f_{\alpha}^{\rho}, f_{\alpha}\rangle_{NP} =  \alpha_{P}(f)(-1)^{g}\tau(f)   H_{p}(f)\langle f, f\rangle_{N}$$
with $H_{p}(f)$ as in \eqref{Hp} 
and
$$\alpha_{P}(f)=\prod_{\wp\vert p}\alpha_{\wp}(f)^{v_{\wp}(P)}$$
\end{lemm}
\begin{proof} When $P=P_{0}:=\prod_{\wp\vert p}\wp$, this is the direct generalisation of  \cite[Lemme 27]{PRlondon}, and it is proved in the same way. In general, we can write $P=P_{0}P_{1}$, and then
$$W_{NP}f_{\alpha}^{\rho}=\N (P_{1}) [P_{1}] W_{NP_{0}}f_{\alpha}^{\rho}$$
Observing  that $[\wp]$ is the adjoint of $U_{\wp}$ for the Petersson inner product and that $\N (P_{1})=[K_{0}(NP):K_{0}(NP_{0})] $, we deduce
$$\langle W_{NP}f_{\alpha}^{\rho}\rangle_{NP}= [K_{0}(NP):K_{0}(NP_{0})] \langle  W_{NP_{0}}f_{\alpha}^{\rho}, U(P_{1})f_{\alpha}\rangle_{NP}=\alpha_{P_{1}} \langle  W_{NP_{0}}f_{\alpha}^{\rho},f_{\alpha}\rangle_{NP_{0}}.$$ The lemma then follows from this and the special case $P=P_{0}$. 
\end{proof}

For the next lemma, let $M, N$ be coprime; then  we define the space of  \emph{weakly $N$-old}  form of level $NM$ to be the subspace of $M _{k}(K_{1}(MN))$ spanned by by forms $f=[d]f'$ for some $d\vert N$ and some modular form $f'$ of level $N'M$ with $N'\vert d^{-1}N$. (This is often called simply the space of $N$-old form forms, but we have reserved that name for the span of forms $[d]f'$ as above with $d\neq 1$.)
\begin{lemm}\label{junk} For a character $\varphi$ of conductor dividing $M$ and an ideal $N$ prime to $M$, let ${ E}_{\varphi}^{M}=E_{1}^{M}(g, 1/2;\varphi)$, $\wtil{ E}_{\varphi}^{M}=W_{M}{ E}_{\varphi}^{M}$.  We have
$$W_{M} [N] \wtil{E}^{M}_{\varphi}=  { E}^{MN}_{\varphi}+ E^{\rm old}$$
where the form $E^{\rm old}$ is weakly old at $N$ (in particular, $E^{\rm old}$ is orthogonal to newforms of exact level $N$ and so is its product with any other form of level prime to $N$).
\end{lemm}
\begin{proof} It is easy to see that $W_{M} [N]\wtil{E}^{M}_{\varphi}= [N]{ E}^{M}_{\varphi}$.
Then we are reduced to showing that   
$$ [N] { E}^{M}_{\varphi}=  {E}^{MN}_{\varphi}+ E^{\rm old}.$$
In fact we have more generally and more precisely  that
$$ \N(N)^{s-}E^{M}_{\varphi}(g\smalltwomat 1 {}{}{\pi_{M}},s)=\sum_{d\vert N} {\varphi(d)\over \N(d)^{2s} }E_{\varphi}^{MN/d}(g,s) ;$$
this is  \cite[Lemma 6.1.4]{shouwu} with $\eps$ replaced by $\varphi$. The lemma then holds  with 
\begin{align*}
E^{\rm old}=\sum_{d\vert N, d\neq 1} {\varphi(d)\over \N(d) }E_{\varphi}^{MN/d}.\qedhere
\end{align*}
\end{proof}

\begin{lemm}\label{nek5.9}
With  notation as in \S\ref{sec:rs}, we have
$$D([\Delta] f, \Theta({\W}), 1)={\W}({\frak D}) D(f, \Theta(\W), 1).$$
\end{lemm}
The proof is as in  \cite[\S\,I.5.9]{nekovar}. 

\begin{theo}\label{theo:interpolate}  Let $\W\colon \calG' \to \baar{\Q}^\times$ be a finite order character of conductor ${\frak f}$ divisible only by primes above $p$. 
Then we have
$$L_{p}(f_{E})(\W)=\frac{\W(d_{F}^{(p)}) \tau(\baar{\W})   \N(\Delta(\W))^{1/2}   V_{p}(f,\W) \baar{\W}(\Delta)}
{  \alpha_{\frakN(\mathfrak{f}(\W))}(f) \Omega_{f}} 
 L(f_{E},\baar{\W},1),
$$
where $\Omega_{f}=(8\pi^{2})^{g}\langle f, f \rangle_{N}$,  $\tau(\baar{\W})$ is as in Lemma \ref{tauW}, and
\begin{align}\label{eq:Vp}
V_{p}(f,\baar{\W})=\prod_{\wp\vert p}\prod_{\frak{p}|\wp}\left(1-{\baar{\W}(\frak{p})\over\alpha_{\wp}(f)}\right).
\end{align}
\end{theo}

\begin{proof} Denote $P=\frakN(\mathfrak{f}(\W))$, $\Delta(\W)=\Delta P$,  $\phi=\W|_{F_{\A}^{\times}}$. We suppose that $\W$ is ramified at all places $v\vert p$ (in this case, we have $V_{p}(f, \W)=1$). Then the result follows from the definition and the following calculation.
\begin{align*}
\Lf(\Phi(\W))&=\frac{ \langle  W_{NP} f_{\alpha}^\rho,{\rm Tr}_{\Delta}[ \Theta(\W) \widetilde{\bf E}_{\eps, N}(\baar{\W})]\rangle_{NP}}{ \langle W_{NP}f_{\alpha}^{\rho}, f_{\alpha}\rangle_{NP}}    \\
{\rm (L.\, \ref{period})} &=  \frac{  \langle  W_{N\Delta} f_{\alpha}^\rho, W_{\Delta(\W)} \Theta(\W) W_{\Delta(\W)} \widetilde{\bf E}_{\eps\phi^{-1}, N}^{\Delta(\W)}\rangle_{N\Delta(\W)} }
{  \alpha_{P}(f) (-1)^{g}\tau(f)   H_{p}(f) \Omega_{f}} 
\\
{\rm (L.\, \ref{junk}, L.\, \ref{tauW})} &=\frac{ (-i)^{g}\W(d_{F}^{(p)}) \tau(\baar{\W})\baar{\W}(\frakD)  D_{E}   }
{  \alpha_{P}(f) (-1)^{g}\tau(f) H_{p}(f)\Omega_{f}} 
\langle  W_{N}[\Delta]f_{\alpha}^\rho, \Theta(\baar{\W}){\bf E}^{N\Delta(\W)}_{\eps\phi^{-1}}\rangle_{N\Delta(\W)} \\
 &= \frac{ (-i)^{g}\W(d_{F}^{(p)})\tau(\baar{\W})\baar{\W}(\frakD)  }
 {  \alpha_{P}(f)  H_{p}(f)  \Omega_{f}}  
\langle [\Delta]f_{\alpha}^\rho, \Theta(\baar{\W}){\bf E}^{N\Delta(\W)\baar{\W}(\frakD)}_{\eps\phi^{-1}}\rangle_{N\Delta(\W)} \\
{\rm   (  eq.\, \eqref{eqn:rs})} &= \frac{\W(d_{F}^{(p)}) \tau(\baar{\W})    D_{F}^{2}  \N(\Delta(\W))^{1/2}  }
	{  \alpha_{P}(f)  H_{p}(f) \Omega_{f}}
	D^{N\Delta(\W)}([\Delta]f_{\alpha}, \Theta(\baar{\W}), 1) \\
{\rm (L.\, \ref{nek5.9})} &=\frac{ \W(d_{F}^{(p)})\tau(\baar{\W})   {D_{F}^{2}}  \N(\Delta(\W))^{1/2}   %V_{p}(f,\W)
 \baar{\W}(\Delta)}
{  \alpha_{P}(f) H_{p}(f) \Omega_{f}} 
 L(f_{E},\baar{\W},1)
\end{align*}
where we have used various results from \S\ref{sec:hecke}, and the fact that in our case $f^{\rho}=f$ as $f$ has trivial character. 

The previous calculation goes through in general with $\Delta(\W)$ replaced by $\Delta(\W)'={\rm l.c.m.}(\Delta(\W),\prod_{\wp\vert p}\wp)$; then one further needs to compare the imprimitive Dirichlet series $D^{\Delta(\W)'}(f_{\alpha}, \Theta(\baar{\W},1)$ with the $L$-value $L(f_{E}, \baar{\W}, 1)$. 
This is done in the same way as in the case of elliptic modular forms, see \cite[Lemme 2.3 (i) and \S 4.4 (III)]{PRlondon}.\footnote{Notice that, as in \emph{op. cit}, our $\Theta(\W)$ is not the primitive theta series when $\W$ is unramified at some $\wp\vert p$; in general we have
$$\Theta(\W)=\left(\prod_{\frakp\vert \wp\vert p}(1-\N(\wp)^{1/2}\W(\frakp)[\wp])\right)\Theta(\W)^{\rm prim}$$
if $\Theta(\W)^{\rm prim}$ is the primitive theta series (i.e. the normalised newform in its representation). This replaces the second-last equation on \cite[p. 21]{PRlondon}, whose $\Theta(\W)$ (respectively $\Theta(\W'')$ is our $\Theta(\W)^{\rm prim}$ (respectively our $\Theta(\W)$).
The factor  $V_{p}(f, \baar{\W})$ comes from the analogue of \cite[Lemme 23]{PRlondon}.} We omit the details since no new phenomena appear in our context and, strictly speaking, we do not need to use the precise form of the interpolation  result except in the ramified case (which already determines $L_{p}(f_{E})$ uniquely).
\end{proof}

\subsection{Factorisation} The $p$-adic analogue of the standard $L$-function of $f$ has been studied by several authors (Manin, Dabrowski, Dimitrov, \ldots). Let $\calG_{F}=\Gal(F_{\infty}/F)$ where $F_{\infty}$ is the maximal $\Z_{p}$-extension of $F$ unramifed outside $p$. 
\begin{theo} There is a a $p$-adic $L$-function $L_{p}(f)\in \OO_{L}\llbracket \calG'_{F}\rrbracket\otimes_{\OO_{L }}L$ uniquely determined by the following property: for each finite order character $\chi\colon \calG_{F}\to \baar{\Q}^{\times}$ of conductor ${\frak f}(\chi)$ divisible by all the  primes $\wp\vert p$, we have
$$L_{p}(f,\chi)= \chi(d_{F}^{(p)})  {\tau(\baar{\chi})\N(\mathfrak{f}(\chi))^{1/2}\over \alpha_{{\frak f}(\chi)}}
{L(f,\baar{\chi},1)\over \Omega_{f}^{+} }$$
where $\Omega_{f}^{+}\in \C^{\times}$ is a suitable period and $\tau(\baar{\chi})=\prod_{v\vert p}\tau(\baar{\chi}_{v})$ with 
$$\tau(\baar{\chi}_{v})= |\pi_{v}|^{-c/2}\int_{F_{v}^{\times}} \baar{\W}_{v}(x_{v})\e_{v}(-x_{v})\, dx_{v}$$
if $c=v({\frak f}(\chi))$.
\end{theo}

Similarly, we have $L_{p, \eps\alpha}(f_{\eps})$ and a period $\Omega_{f_{\eps}}^{+}$ satisfying
$$L_{p, \eps\alpha}(f_{\eps},\chi)=\chi(d_{F}^{(p)}){\tau(\baar{\chi})\N(\mathfrak{f}(\chi))^{1/2}\over \eps({\frak f}(\chi))\alpha_{{\frak f}(\chi)}} {L(f_{\eps},\baar{\chi},1)\over \Omega_{f_{\eps}}^{+} }$$
for ramified finite order characters $\chi$. (In fact we have $\eps({\frak f}(\chi)=1$ under our assumptions).)

For the proof of the existence of $L_{p}(f)$ we refer to  \cite{dimitrov}: notice that our $L_{p}(f, \chi)$ equals  $\chi(d_{F}^{(p)})L_{p}(\pi_{f}, \chi^{-1})$ in \emph{op. cit}, where moreover the notation $\tau(\chi_{v})$  refers to \emph{un}normalised Gau\ss\ sums. The definition  and properties of the period $\Omega_{f}^{+}$ and of a related period $\Omega_{f}^{-}$ (both of which are a priori defined up to an $M_{f}^{\times}$-ambiuguity), are given in \S\ref{sec:periods} below; here we need the following: $\Omega_{f}^{+}\Omega_{f}^{-}\sim \Omega_{f}$, and $\Omega_{f_{\eps}}^{+}\sim D_{E}^{-1/2} \Omega_{f}^{-}$ where $\sim$ denotes equality in $\C^{\times}/M_{f}^{\times}$.
Then from the complex factorisation $L(f_{E}, \chi\circ\frakN,s)=L(f, \chi, s)L(f_{\eps}, \chi,s)$ and the interpolation properties satisfied by each factor we obtain
\begin{gather}\label{factoris2}
L_{p}(f_{E},\chi\circ\frakN)
=  \chi(\Delta)^{2}{\Omega_{f}^{+}\Omega_{{f}_{\eps}}^{+}\over D_{E}^{-1/2}\Omega_{f}} 
   L_{p}(f,\chi)L_{p}(f_{\eps},\chi),\end{gather}
where the period factor is in $M_{f}^{\times}$ (in particular, it  is algebraic).

\subsection{Fourier expansion of the analytic kernel}\label{4.5} 
Consider the restriction of $\Phi$ to $\calG$, the Galois group of the maximal $\Z_{p}$-extension of $E$ unramified outside $p$. Any character $\W$ of $\calG$ decomposes uniquely as $\W=\W^{+}\W^{-}$ with $({\W}^{+})^{c}=\W$, $({\W}^{-})^{c}={\W}^{-1}$ (we say that $\W^{+}$ is cyclotomic and $\W^{-}$ is anticyclotomic or dihedral). Since we are interested in the case  $\eps(N)=(-1)^{g-1}$ in which   $\Phi$ is zero on the  anticyclotomic characters, we study the restriction of $\Phi$ to the cyclotomic characters. We can write $\W^{+}=\chi\circ\frakN$ for a Hecke character  $\chi\colon F^{\times} \bks F_{\A}^{\times}\to 1+p\Z_{p}$, and we denote
$$\Theta_{\chi}=\Theta(\chi\circ\frakN), \qquad \Phi_{\chi}=\Phi(\chi\circ\frakN).$$

From now on we  assume that $(\Delta,2)=1$ and all primes $\wp\vert p$ are split in $E$.
\begin{prop}\label{fourierPhi} The Fourier coefficients $b(m)=a_{p}(\Phi_{\chi},m)$ of the $p$-adic modular form $\Phi_{\chi}$ are given by 
\begin{align*}
 b(m)&= \sum_{\substack{n\in F \\ 0<n<1\\n\in Nm^{-1}\Delta^{-1}}}\chi((1-n)m) 
 \prod_{v\vert \Delta} \left[\one[v(nm)=0]+\eps_v((n-1)n)  \chi_v^{-2}(nm\wp_v/N)\right] & \\
 &\qquad\qquad \qquad\qquad\qquad\qquad\qquad \cdot r((1-n)m\Delta)\sigma_{\eps\chi^{-2}}(nm/N). 
 \end{align*}
\end{prop}

\begin{proof}
By  \eqref{fouriertrace}, the Fourier coefficient $b(m)$ of $\Phi_{\chi}={\rm Tr}_{\Delta}[\Theta_{\chi}\wtil{\bf E}_{\eps\chi^{2},N}]$ is given by 
$$b(m)=\sum_{\delta|\Delta}b^{\delta}(m\delta)$$
with
\begin{align*} b^{\delta}(m)=a(\Phi_{\chi}^{(\delta)}, m)=|y|^{-1}\tilde{a}(\Phi_{\chi}^{(\delta)}, y)&=|y|^{-1}\sum_{n\in F}\tilde{a}(\Theta_{\chi}^{(\delta)},(1-n)y)\,\tilde{a}(\wtil{\bf E}_{\eps\chi^{-2},N}^{(\delta)},ny)\\
&=|y|^{-1}\sum_{n\in F}\tilde{a}(\Theta_{\chi}^{(\delta)},(1-n)y)\,\tilde{a}(\wtil{\bf E}_{\eps\chi^{-2}}^{(\delta)},ny/\pi_{N})
\end{align*}
if $y\in F_{\A}^{\times}$ satisfies $y_{\infty}>0$ and $y^{\infty}d_{F}=m$.

Then by Proposition \ref{thetafw} and Corollary \ref{eisfw}, we have:\footnote{Recall that $\kappa(v)^{2}=\eps_{v}(-1)$.}
\begin{equation}\label{b(m)}
\begin{split}
b(m)= \sum_{\delta|\Delta}\sum_{\substack{n\in F \\ 0<n<1}} \eps_\delta((n-1)n)\chi^{-1}(\delta)\chi((1-n)m\delta)  \chi_\delta^{-2}(nm\delta/N)  \\
 \cdot r((1-n)m\delta)\sigma_{\eps\chi^{-2}}(nm/N) . 
 \end{split}
\end{equation}
 We interchange the two sums and notice that the term corresponding to $\delta$ and $n$ is nonzero only if $n\in Nm^{-1}\Delta^{-1}$ and $\delta_{0}|\delta$, where $$\delta_{0}=\delta_{0}(n)=\prod_{\substack{v\vert \Delta\\ v(nm)=-1}}\wp_{v}$$
  ($\wp_{v}$ being the prime corresponding to $v$). Now for each $n$ we can rewrite the sum over $\delta$ (omitting the factor $\chi((1-n)m)$ and those on the second line of \eqref{b(m)}, which  do not actually depend on~$\delta$) as
  \begin{align*}
&\eps_{\delta_{0}}((n-1)n)  \chi_{\delta_{0}}^{-2}(nm\delta_{0}/N) \sum_{\delta'|\Delta/\delta_{0}} \eps_{\delta'}((n-1)n) \chi_{\delta'}^{-2}(nm\delta'/N)  \\
&= \prod_{v\vert \delta_{0}} \eps_{\delta'}((n-1)n) \ \chi_{v}^{-2}(nm\wp_{v}) \prod_{v\vert \Delta/\delta_{0}}[1+\eps_{v}((n-1)n)\chi_{v}^{-2}(nm\wp_{v})].
\end{align*}
The asserted formula follows. 
\end{proof}

\begin{rema}\label{eps1}
If $v(nm)=-1$ then 
 $(n-1)\pi_{m}\pi_{v}\equiv n\pi_{m}\pi_{v}$ in $(\OO_{F,v}/\pi_{v}\OO_{F,v})^{\times}$, so that we actually have $$\eps_{v}((n-1)n)=\eps_{v}((n-1)\pi_{m}\pi_{v})\eps_{v}(n\pi_{m}\pi_{v})=1.$$ 
\end{rema}

We can now compute the Fourier coefficients of the analytic kernel giving the central derivative of the $p$-adic $L$-function in the cyclotomic direction. To this end, let
$$\nu\colon \Gal(\baar{\Q}/F)\to1+p\Z_p\subset \Q_p^\times$$
Since $\Lf$ is  continuous, we have
$${d\over ds}L_{p}( f_{E}, \nu^{s}\circ \frak N)={d\over ds}\Lf(\Phi(s))=\Lf({d\over ds}\Phi(s)).$$
In particular $L_{p;\nu\circ\frakN}'(f_{E},\one)=\Lf(\Phi'(0))$.

Let $\ell_{F}={d\over d s}_{|_{s=0}}{\nu^s}\colon F^{\times}\bks F_{\A^{\infty}}^{\times}\to \Q_{p}$ be the  $p$-adic logarithm associated with $\nu$.
\begin{prop}\label{Phi} Assume that $\eps(N)=(-1)^{g-1}$. Then  $\Phi(0)=0$ and the Fourier coefficients $b'(m)$ of 
$$\Phi_{\nu}'=\Phi'(0)={d\over d s}\Phi_{\nu^s}|_{s=0}$$ 
are nonzero only for $m$ integral and nonzero, in which case 
 $$b'(m)=\sum_v b'_{v}(m)$$
with the sum running over all finite places $v$ of $F$ and $b'_v(m)$ given for $(\prod_{\wp\vert p}\wp)\vert m$ by:
\begin{enumerate}
\item If $v=\wp$ is inert in $E$, then 
$$b_v'(m)= \sum_{\substack{n\in Nm^{-1}\Delta^{-1}\\ (p,nm)=1 \\ \eps_v((n-1)n)=1\ \forall v\vert \Delta \\ 0<n<1}}
2^{\omega_{\Delta}(n)} 
  r((1-n)m\Delta ) r(nm\Delta/N\wp)(v(nm/N)+1)\ell_{F,v}(\pi_{v}),   $$ 
where
 $$\omega_{\Delta}(n)=\#\{v\vert (\Delta, nm\Delta)\}.$$
\item If $v=\wp\vert \Delta$ is ramified in $E$, then
$$b_v'(m)=  \sum_{\substack{n\in Nm^{-1}\Delta^{-1}\\ (p,nm)=1 \\ \eps_v((n-1)n)=-1\\ \eps_w((n-1)n)=1 \ \forall v\neq w|\Delta \\ 0<n<1}}
 2^{\omega_{\Delta}(n)} 
 r((1-n)m\Delta) r(nm\Delta/N)(v(nm)+1)\ell_{F,v}(\pi_{v}).   $$
\item If $v$ is split in $E$, then
$$b'_{v}(m)=0.$$
\end{enumerate}
\end{prop} 
 \begin{proof}
 The vanishing of $\Phi(0)=\Phi_\one$ follows from the functional equation \eqref{feq2} and the sign assumption. 

By Proposition \ref{fourierPhi}, the Fourier coefficient $b_s(m)$ of $\Phi(s)=\Phi_{\nu^{s}}$ can be expressed as $b_{s}(m)=\sum_{n\in F} b_{n,s}(m)$ with
$$b_{n,s}(m)= \nu^{s}((1-n)m) r((1-n)m\Delta)\prod_{v\nmid p\infty} \sigma^n_{s,v}(m/N)$$
where, using Remark \ref{eps1}:
$$\sigma^n_{s,v}(m)=\begin{cases}
\dfrac{1-\eps(n m \wp )\nu( n m \wp)^{-2s} }{ 1-\eps(\wp)\nu(\wp)^{-2s}} &\textrm{ if } v=\wp\nmid\Delta;\\
1+\eps_v(n(n-1)) \nu( nm \wp)^{-2s} &\textrm{if }v=\wp\vert \Delta{\rm \ and \ } v(nm)=0;\\
\ \nu(nm\wp)^{-2s}  &\textrm{if }v=\wp\vert \Delta{\rm \ and\  } v(nm)=-1.
\end{cases}$$

Then $b'(m)=\sum_{n} b_n'(m) =\sum_{n}\sum_v b_{n,v}'(m)$ with $\sum_{n} b_{n,v}'(m)=b_v(m)$, and $b_n'(m)$ can be nonzero only if exactly one of the factors $\sigma^{n}_{s,v}$ vanishes at $s=0$. If this happens for the place $v_0$, then the set over which $n$ ranges accounts for the positivity and integrality conditions and the nonvanishing conditions at other places, 
whereas the condition $(p,nm)=1$ results from  observing that $\lim_{s\to 0} \nu^s(a)=\one[(p,a)=1]$.

The values of $b'_{n,v}$ can then be determined in each case from the above  expressions: for $v$ ramified this is straightforward. For $v=\wp$ inert, notice that if $v(nm/N)$ is odd then $r(nm\Delta/N\wp)=r((nm\Delta/N)^{(\wp)})$, where the superscript denotes prime-to-$\wp$ part; whereas if $v(nm/N)$ is even then $\sigma_{0,v}^n(m/N)$ does not vanish so $(n,v)$ does not contribute to $b'(m)$ and indeed $r(nm/N\wp)=0$.
\end{proof}

 \clearpage
  
 \part{Heights}
 
 \section{$p$-adic heights and Arakelov theory}

By the work of many authors (Schneider, Perrin-Riou, Mazur--Tate, Coleman--Gross, Zarhin, Nekov{\'a}{\v{r},\ldots)  there are  $p$-adic height pairings on the Mordell-Weil group of an abelian variety defined over  a number field. In this section, we first recall (\S\S\ref{sec:ls}-\ref{sec:height pairing}) a definition of the height pairing as a sum of local symbols following Zarhin \cite{zarhin} and Nekov{\'a}{\v{r} \cite{nekheights}, and explain how it  induces a pairing on degree zero divisors on curves. In \S\S\ref{sec:arakelov}-\ref{sec:arglobal} we explain how $p$-adic Arakelov theory allows to extend the height pairing for curves to a pairing on divisors of any degree.

 \subsection{Local symbols}\label{sec:ls}

Let $A$ be an abelian variety of dimension $g$ over a local field $E_{v}$, $A^{\vee}$ its dual abelian variety, and  and let $V=V_{p}A=T_{p}A\otimes_{\Z_{p}}\Q_{p}$ be the rational Tate module of $A$, a continuous $\Gal(\baar{E}/E)$-representation.\footnote{Nekov{\'a}{\v{r}} \cite{nekheights} defines height pairings for Galois representations in much greater generality than described here.}
  Let $\ell_{v}:E_{v}^{\times}\to \Q_{p}$ be a homomorphism; we call $\ell$ a \emph{local $p$-adic logarithm} and assume that it is ramified, that is, $\ell_{v}:E_{v}^{\times}\to\Q_{p}$ does not vanish identically on $\OO_{E,v}^{\times}$. Let $D_{\mathrm{dR}}(V_{v})$ be the filtered $\Q_{p}$-vector spaces attached to $V_{v}$ by the theory of Fontaine. The comparison theorem 
 identifies $D_{\rm dR}(V_{v})$ with  $H^{1}_{\mathrm{dR}}(A^{\vee}/E_{v})$, equipped with the Hodge filtration; it is also identified with the filtered Dieudonn\'e module of the special fibre of the $p$-divisible group of $A$ (after an extension of scalars if  $E_{v}$ is ramified over $\Q_{p}$; see \cite{fontaine}) .
Let $L$ be a finite extension of the coefficient field $\Q_{p}$, and if $v\vert p$, let  $W_{v}\subset D_{\mathrm{dR}}(V_{v})\otimes L$ be a splitting of the Hodge filtration, that is, a complementary subspace to $\Omega^{1}(A^{\vee}/E_{v})\otimes L\subset D_{\mathrm{dR}}(V_{v})\otimes L$, which is isotropic\footnote{The isotropy condition ensures that the resulting height pairing is symmetric \cite[Theorem 4.1.1 (4)]{nekheights}} for the cup product. When $V_{v}$ is ordinary, there is a canonical choice for $W_{v}$, the ``unit root'' subspace (see e.g. \cite{iovita} for a nice discussion).
 
We proceed to define pairings, called  \emph{local N\'eron symbols}\footnote{The notation is a bit abusive: the subscript $W$ is meant to recall that the local pairing depends on the choice of $W_{v}$ when $v\vert p$; when $v\nmid p$, it has no meaning. Although the symbol also depends on $\ell$, we will usually omit it from the notation.}
$$\langle\ ,\ \rangle_{v,W} \colon ( \mathcal{D}_{0}(A)(E_{v}) \times Z_{0}(A)^{0}(E_{v}))_{e}\to L$$
on the subset of pairs with disjoint supports in the product of the group $\mathcal{D}_{0}(A)(E_{v})$ of divisors algebraically equivalent to zero defined over $E_{v}$ and the group $Z_{0}(A)^{0}(E_{v})$ of zero-cycles of degree zero defined over $E_{v}$. 

Let $\calA/\OO_{E,v}$,  $\calA^{\vee}/\OO_{E,v}$ be the N\'eron models of $A$  and $A^{\vee}$, and let  $\calA^{0}$ be the identity component of $\calA$. The rational equivalence class $[D]$ of $D\in\mathcal{D}_{0}(A)(E_{v})$ defines a point in $A^{\vee}(E_{v})=\calA^{\vee}(\OO_{E,v})={\rm Ext}^{1}_{\rm \emph{ fppf}}(\calA^{0},{\bf G}_{m})$, hence an extension
$$1\to {\bf G}_{m}\to \mathcal{Y}_{[D]}\to\calA^{0}\to 1$$
of  abelian \emph{fppf} sheaves on $\OO_{E,v}$, and $\mathcal{Y}_{[D]}$ is represented by a smooth commutative group scheme. On the generic fibre, $\mathcal{Y}_{[D]}\otimes E_{v}$ can be identified with the complement $Y_{D}$ of the zero section in the total space of the line bundle $\OO(D)$ on $A$, and thus the extension admits a section 
$$s_{D}\colon A\setminus |D|\to Y_{D}$$
which is canonical up to scaling.

Suppose given a morphism $\ell_{v,D,W}$ which makes the following diagram commute:

\[
\xymatrix{
0\ar[r] &{\OO_{E,v}^{\times}\hat{\otimes} L} \ar[r] \ar[d]& \mathcal{Y}_{[D]}(\OO_{E,v})\hat{\otimes} L \ar[r] \ar[d]& \calA^{0}(\OO_{E,v})\hat{\otimes} L \ar@{=}[d]\ar[r] & 0\\
0\ar[r] &{E_{v}^{\times}\hat{\otimes} L} \ar[r] \ar[d]^{\ell_{v}}& {Y}_{D}({E_v})\hat{\otimes} L \ar[r] \ar[d]^{\ell_{v,D,W}} & A(E_{v})\hat{\otimes} L \ar[r] & 0\\
{}&L \ar@{=}[r] & L. & &
}
\]
Then we can define the local pairing by
\begin{gather}\label{localhtdef}\langle D, z\rangle_{v,W}=\ell_{v,D,W}(s_{D}(z)),\end{gather}
where $s_{D}$ is extended to the divisor $z$ in the obvious way. Notice that since $z$ has degree zero, this is well-defined independently of the scaling ambiguity in $s_{D}$. 
\medskip

When $v\nmid p$, the logarithm $\ell_{v}$ vanishes on $\OO_{E,v}^{\times}$ for topological reasons and we can uniquely extend it to an $\ell_{v,D}$ as in the above diagram by requiring its restriction to $\mathcal{Y}_{[D]}(\OO_{E,v})$ to be trivial.  When $v\vert p$, given the splitting $W_{v}$ one can construct a section 
$$s_{v,D,W}\colon A(E_{v})\hat{\otimes} L\to Y_{D}(E_{v})\hat{\otimes}L$$
and define the extension $\ell_{v,D,W}$ by requiring it to be trivial on the image of $s_{v,D,W}$.  The standard construction is explained e.g. in \cite[\S 3.2]{kobayashi2}. In the ordinary case, when $W_{v}$ is chosen to be the unit root subspace, the crucial properties of the (canonical) local symbol are the last two in Proposition \ref{theo:ht} below; in this case the construction rests on the following result (see \cite{schneider} or \cite[\S 6.9]{nekheights}).
\begin{lemm}\label{unorm} Let $E_{v,\infty}$ be a totally ramified $\Z_{p}$-extension of $E_{v}$, and denote by $E_{v, n}$ its $n^{\rm th}$ layer. Let $e\in \End(A)\otimes\baar{\Q}$ be an idempotent. Assume that $eV$ is ordinary as a Galois representation. Then the module  of \emph{universal norms}
$$U(eA(E_{v}))
=\bigcap_{n} {\rm Im}\left[{\rm Tr}_{E_{v,n}/E_{v}}\colon eA(E_{v,n})\to eA(E_{v})\right]$$
has finite index in $eA(E_{v})$.
\end{lemm}

\begin{prop}\label{theo:ht} The $p$-adic local symbol
$$\langle\ ,\ \rangle_{v}= \langle\ ,\ \rangle_{v,W} \colon ( \mathcal{D}_{0}(A)(E_{v}) \times Z_{0}(A)^{0}(E_{v}))_{e}\to L$$
defined by \eqref{localhtdef} has the following properties (valid whenever they make sense).
\begin{enumerate}
\item\label{bil} It is bilinear.
\item If $h\in E_{v}(A)$ is a rational function, we have
$$\langle (h), z\rangle_{v}=\ell_{v}(h(z))$$
where if $z=\sum n_{P}P$, $h(z)=\prod h(P)^{n_{P}}$.
\item If $\phi\colon A\to A$ is a finite endomorphism, we have
$$\langle \phi^{*}D, z\rangle_{v}=\langle D, \phi_{*}z \rangle_{v}.$$
\item\label{cont} For any $D\in \mathcal{D}_{0}(A)(E_{v})$, $x_{0}\in A(E_{v})\setminus |D|$ the map from $A(E_{v})\setminus |D|\to L$ defined by 
$$x\mapsto \langle D, x-x_{0}\rangle_{v}$$
is continuous.
\item\label{comp} (Compatibility.) Let $E_{w}'/E_{v}$ be a finite extension. If $D\in \mathcal{D}_{0}(A)(E'_{w})$, $z\in Z_{0}(A)^{0}(E_{v})$, we have
$$\langle {\rm Tr}_{E_{w}'/E_{v}}( D'), z\rangle_{v}=\langle D', z\rangle_{w}$$
where $\langle\ , \ \rangle_{w}$ is the local pairing associated with $\ell_{w}=\ell_{v}\circ N_{E_{w}/E_{v}}$ and (if $v\vert p$) the splitting $W_{w}$ is induced from $W_{v}$.
\item\label{boundedness} (Boundedness.)  If $v\vert p$, let $E_{v,\infty}^{\ell}=\cup_{n} E_{v,n}^{\ell}$ be the ramified\footnote{Recall that we choose $\ell_{w}$ to be ramified.}  $\Z_{p}$-extension of $E_{v}$ determined by the isomorphism 
$$E_{v}^{\times}\supset\Ker(\ell_{v})\cong \Gal(E_{v,\infty}^{\ell}/E)\subset \Gal(E_{v}^{\rm ab}/E).$$
induced from class field theory.  In the  ordinary situation of Lemma \ref{unorm}, if $eW_{v}$ is the unit root subspace of $eV$, there is a nonzero constant $c\in\Z_{p}$ 
such that
$$\langle D, z\rangle_{v,n}\in {c}^{-1}\ell_{w}(E_{v,n}^{\ell,\times})$$
if $D\in e  \mathcal{D}_{0}(A)(E_{v,n})$, $z\in eZ_{0}(A)^{0}(E_{v,n})$ and $\langle\ , \ \rangle_{v,n}$ is the local pairing associated with the extension $E_{v,n}^{\ell}/E_{v}$ as in  \ref{comp}.
\end{enumerate}
If $v\nmid p$, the local symbol is characterised by properties \ref{bil}--\ref{cont}.
\end{prop}
We refer to \cite[II.1]{nekovar}, \cite[\S 4.2]{kobayashi} and references therein for the proof and more details on the construction. See also Proposition \ref{ar} below.

\medskip

\subsection{The $p$-adic height pairing}\label{sec:height pairing} 

Let $A$ be an abelian variety over a number field $E$. Let $\ell\colon E^{\times}\bks E_{\A}^{\times}\to \Q_{p}$ be a a homomorphism (which we call a \emph{global $p$-adic logarithm}), whose restrictions $\ell_{v}=\ell|_{E_{v}^{\times}}$ are ramified for all $v\vert p$. Let  $W_{v}$ be Hodge splittings at the places $v\vert p$ as in \S\ref{sec:ls}. Then  we can define a height pairing 
$$\langle \ , \  \rangle\colon A^{\vee}(E)\times A(E)\to L$$ 
as the sum of local height pairings
$$\langle x, y\rangle=\sum_{v} \langle \tilde{x},\tilde{y}\rangle_{v},$$
where $\tilde{x}$ is a divisor on $A$ whose class in $A^{\vee}(E)\cong \Pic^{0}(A)$ is $x$, and $\tilde{y}=\sum n_{P}[P]$ is a zero-cycle of degree zero on $A$ with support disjoint from the support of $\tilde{x}$, which satisfies $\sum n_{P}P=y$.  The result is independent of the choices of $\tilde{x}$ and $\tilde{y}$.

\medskip

Let $X$ be a (proper, smooth) curve over $E$ of genus $g\geq 1$, together with a degree~$1$ divisor class defined over $E$ inducing an embedding 
$$\iota\colon X\hookrightarrow J(X)$$
into its Albanese variety $J(X)$.\footnote{In our applications, we only have a rational divisor class, inducing a compatible system of maps $\iota_{E'}\colon X(E')\otimes \Q \to J(X)(E')\otimes\Q$ for $E'$ a finite extension of $E$, such that for some integer $n$, $(n\iota_{E'})_{E'}$ is induced from an $E$-morphism. This causes no extra difficulties.} Let  $\mathrm{Div}(X)$ be the group of divisors on $X$, $\mathrm{Div}^{0}(X)$ the subgroup of degree zero divisors, and similarly $\CH(X)=\mathrm{Div}(X)/\sim$, $\CH(X)_{0}=\mathrm{Div}^{0}(X)/\sim$ the Chow group of zero-cycles modulo rational equivalence and its subgroup of degree zero elements.  Then, given a $p$-adic logarithm and Hodge splittings for $V_{p}J(X)$, we can define local and global pairings on degree zero divisors on $X$ (denoted with a subscript $X$) from the above pairings on $J(X)$ (here denoted with a subscript $J(X)$). Let $D_{1}$, $D_{2}$ be divisors  of degree zero on $X$ defined over $E$ and with disjoint support.  The morphism $\iota$ induces an isomorphism $\iota^{*}\colon \Pic^{0}J(X)\cong \Pic^{0}(X)$, hence we can pick an algebraically trivial divisor $D_{1}'$ on $J(X)$ satisfying $D_{1}=\iota^{*}D_{1}'+(h)$ for some rational function $h\in E(X)$. If $D_{1}'$ is chosen so that its support is disjoint from the support of $\iota_{*}D_{2}$ and the support of $(h)$ is disjoint from the support of $D_{2}$, we can define 
$$\langle D_{1},D_{2}\rangle_{v,X}=-\langle D_{1}', \iota_{*}D_{2}\rangle_{v,J(X)}-\ell_{v}(h(D_{2})),$$
and 
$$\langle D_{1},D_{2}\rangle_{X}=\sum_{v}\langle D_{1},D_{2}\rangle_{v,X}.$$

The latter pairing descends to a \textbf{height pairing} on divisor classes
$$\langle \ ,\ \rangle\colon \CH(X)_{0}\times \CH(X)_{0}\to L.$$
There are various conventions in the literature for the normalisation of the signs of height pairings. Our choices are the same as those of Kobayashi \cite[\S 4.3]{kobayashi}, whose discussion we have followed and to which we refer for a comparison with other authors' choices.

\subsection{$p$-adic Arakelov theory -- local aspects}\label{sec:arakelov}
 
Here and in \S\ref{sec:arglobal} we summarise the main results of Besser \cite{besser}, who develops the $p$-adic analogue of classical Arakelov theory.

\subsubsection{Metrised line bundles} 
Let $X_{v}$ be a proper smooth variety over the finite extension  $E_{v}$ of $\Q_{p}$, and fix a ramified local $p$-adic logarithm $\ell_{v}:E_{v}^{\times}\to \Q_{p}$ which we extend to $\baar{\Q}_{p}^{\times}$ by $\ell_{v}|_{E_{v}'^{\times}}=\ell_{v}\circ N_{E_{v}'/E_{v}}$ for any finite extension $E_{v}'/E_{v}$.

A \emph{metrised line bundle} $\hatL=(\calL, \log_{\calL})$ on $X_{v}$ is a line bundle on $X_{v}$ together with a choice of  a \emph{log function} $\log_{\calL}$ on the total space of $\calL$ minus the zero section (which will also be viewed as a function on the nonzero sections of $\calL$). A log function is the analogue in the $p$-adic theory of the logarithm of a metric on the sections of a line bundle on a Riemann surface. It is a Coleman function having a certain analytic property\footnote{For which we refer to \cite[Definition 4.1]{besser}.} and the following algebraic property. If the $p$-adic logarithm $\ell_{v}$ factors as 
\begin{gather}\label{tv}
\ell_{v}=t_{v}\circ\log_{v}
\end{gather}
 for some $\log_{v}\colon E_{v}^{\times} \to E_{v}$ and some $\Q_{p}$-linear $t_{v}\colon E_{v}\to \Q_{p}$, then for any nonzero section $s$ of $\calL_{v}$ and rational function $f\in E(X_{v})$ we have
\begin{gather}\label{logv} 
 \log_{\calL,v}(f s)=\log_{v}(f)+\log_{\calL,v}(s).
\end{gather}
Adding a constant to a $\log$ function produces a new log function; this operation is called \emph{scaling}.

One can define a notion of $\baar{\partial}\partial$-operator on Coleman functions, and attach to any log function $\log_{\calL}$ on $\calL$ its \emph{curvature} $\baar{\partial}\partial\log_{\calL}\in H^{1}_{\rm dR}(X_{v})\otimes \Omega^{1}(X_{v})$; its cup product is the first Chern class of $\calL$.

Log functions on a pair of line bundles induce in the obvious way a log function on their tensor product, and similarly for the dual of a line bundle. If $\pi:X_{v}\to Y_{v}$ is  a morphism, then a log function on a line bundle  on $Y_{v}$ induces in the obvious way a log function on the pullback line bundle on $X_{v}$. If moreover $\pi$ is a finite Galois cover with Galois group $G$, and $\mathcal{L}$ is a line bundle on $X_{v}$ with log function $\log_{\mathcal{L}}$ and associated curvature $\beta$, then the norm line bundle $N_{\pi}\mathcal{L}$ on $Y_{v}$ with stalks 
$$(N_{\pi}\calL)_{y}=\underset{x\mapsto y}{\bigotimes}\calL_{x}^{\otimes e(x|y)}$$
has an obvious candidate log function $N_{\pi}\log_{\calL}$ obtained by tensor product. A delicate point is that it is not automatic that the latter is a genuine log function (i.e. it satisfies the analytic property alluded to above); cf. 
 \cite[Proposition 4.8]{besser} for a sufficient condition.

\subsubsection{The canonical Green function} Now let $X_{v}/E_{v}$ be a curve of genus $g\geq 1$ with good reduction above $p$. Choose a splitting $W_{v}\subset H^{1}_{\rm dR}(X_{v})\otimes L$ of the Hodge filtration as  in \S\ref{sec:ls}, which we use to identify $W_{v}\cong \Omega^{1}(X_{v})^{\vee}$; we then  define a canonical  element
$$\mu_{X_{v}}={1\over g}\id\in \End \Omega^{1}(X_{v})\cong W_{v}\otimes\Omega^{1}(X_{v})$$
and similarly for the self-product $X_{v}\times X_{v}$ (denoting $\pi_{1}$, $\pi_{2}$ the projections)
$$\Phi =\twomat {1\over g} {-1}{-1}{1\over g}\in  \End (\pi_{1}^{*}\Omega^{1}(X_{v})\oplus\pi_{2}^{*}\Omega^{1}(X_{v}))\hookrightarrow H^{1}_{\rm dR}(X_{v}\times X_{v})\otimes \Omega^{1}(X_{v}\otimes X_{v}).$$
The first Chern class of $\Phi$ is the class of the diagonal $\Delta\subset X_{v}\times X_{v}$.

 Let  $s_{\Delta}$ denote the canonical section of the line bundle $\OO(\Delta)$ on $X_{v}\times X_{v}$. Given any log function $\log_{\OO(\Delta)}$ on $\OO(\Delta)$ with curvature $\Phi$, we can consider the function $G$ on $X_{v}\times X_{v}$ given by
$$G(P,Q)=\log_{\OO({\Delta})}(s_{\Delta})(P,Q).$$
It is a Coleman function with singularities along $\Delta$; we call $G$ a \emph{Green function} for~$X_{v}$.

A Green function $G$ induces a log function on any line bundle $\OO(D)$ on $X_{v}$ by 
$$\log_{\OO(D)}(s_{D})(Q)=\sum n_{i} G(P_{i},Q)$$
if $D=\sum n_{i}P_{i}$ and $s_{D}$ is the canonical section of $\OO(D)$. A log function  $\log_{\calL}$ on the line bundle $\calL$ and the resulting metrised line bundle $(\calL, \log_{\calL})$ are called \emph{admissible} with respect to $G$ if for one (equivalently, any) nonzero rational section $s$ of $\calL$,  the difference  $\log_{\calL}(s)-\log_{{\rm div}(s)}$ is a constant. Such a constant is denoted by $\iota_{\log_{\calL}}(s)$, or $\iota_{\log_{v}}(s)$ in the case of the trivial line bundle with the log function $\log_{v}$. It is the analogue of the integral of the norm of $s$. It follows easily from the definitions that any isomorphism of admissible metrised line bundles is an isometry up to scaling. 

Let $\omega_{X_{v}}$ be the canonical sheaf on $X_{v}$. The  canonical isomorphism $\omega_{X_{v}}\cong \Delta^{*}\OO(-\Delta)$ gives another way to induce from $G$ a log function $\log_{\omega_{X_{v}}}^{G}$ on $\omega_{X_{v}}$, namely by pullback (and the resulting metrised line bundle has curvature  $(2g-2)\mu_{X_{v}}$). The requirement that this log function be admissible, together with a symmetry condition, leads to an almost unique choice of~$G$.

\begin{prop}[{\cite[Theorem  5.10]{besser}}]\label{green} There exists a unique up to constant symmetric Green function $G$ with associated curvature $\Phi$ such that $(\omega_{X_{v}}, \log_{\omega_{X_{v}}}^{G})$ is an admissible metrised line bundle with respect to $G$.
\end{prop}
In the following we will arbitrarily fix the constant implied by the Proposition. In our context, the canonical Green function thus determined is, in a suitable sense, defined over $E_{v}$ \cite[Proposition 8.1]{besser}.

\subsection{$p$-adic Arakelov theory -- global aspects}\label{sec:arglobal} Let $E$ be a number field with ring of integers $\OO_{E}$. Let $\X/\OO_{E}$ be an arithmetic surface with generic fibre $X$, that is, $\X\to\OO_{E}$ is a proper regular relative curve and $\X\otimes_{\OO_{E}}E=X$. We assume that $\X$ has \emph{good reduction} at all place $v\vert p$, and denote $X_{v}=\X\otimes E_{v}$. Fix choices of a ramified\ $p$-adic logarithm $\ell$ and Hodge splittings $W_{v}$ as in \S\ref{sec:arakelov}.

\subsubsection{Arakelov line bundles and divisors} An \emph{Arakelov line bundle} on $\X$ is a pair
 $$\widehat{\calL}=(\calL,(\log_{\calL_{v}})_{v\vert p})$$
  consisting of a line bundle $\calL$ on $\X$ together with admissible (with respect to the Green functions of Proposition \ref{green}) log functions $\log_{\calL_{v}}$ on $\calL_{v}=\calL|_{X_{v}}$. We denote by ${\rm Pic}^{\Ar}(\X)$ the group of isometry classes of Arakelov line bundles on $\X$.

The group $\Div^{\Ar}(\X)$ of \emph{Arakelov divisors} on $\X$ is the group of formal combinations
$$D=D_{\rm fin}+D_{\infty}$$
where $D_{\rm fin}$ is a divisor on $\X$ and $D_{\infty}=\sum_{v\vert p}\lambda_{v}X_{v}$ is a sum with coefficients $\lambda_{v}\in E_{v}$ of formal symbols $X_{v}$ for each place $v\vert p$ of $E$. To an Arakelov line bundle  $\hatL$ and a nonzero rational section $s$ of $\calL$ we associate the Arakelov divisor
$$\widehat{\rm div}(s)=(s)_{\rm fin}+(s)_{\infty}$$
where $(s)_{\rm fin}$ is the usual divisor of $s$ and $(s)_{\infty}= \sum_{v\vert p}\iota_{\log_{\calL_{v}}}(s_{v})X_{v}$. The group ${\rm Prin}^{\Ar}(\X)$ of \emph{principal} Arakelov divisors on $\X$ is the group generated by the $\widehat{{\rm div}}(h)$ for $h\in E(\X)^{\times}$. The Arakelov Chow group of $\X$ is 
$$\CH^{\Ar}(\X)= {\Div}^{\Ar}(\X)/{\rm Prin}^{\Ar}(\X),$$
and we have an isomorphism
$${\rm Pic}^{\Ar}(\X)\cong \CH^{\Ar}(\X)$$
given by  $\hatL\to [\widehat{\rm div}(s)]$ for any rational section $s$ of $\calL$.

\subsubsection{The $p$-adic Arakelov pairing} Most important for us is the existence of a pairing on $\CH^{\Ar}(\X)$, extending the $p$-adic height pairing of divisors of \S\ref{sec:height pairing}. Let $(\ , \ )_{v}$ denote the ($\Z$-valued) intersection pairing of cycles on $\X_{v}$ with disjoint support.
 \begin{prop}[Besser \cite{besser}]\label{ar} Let $\X/\OO_{E}$ be an arithmetic surface with good reduction above~$p$. For any choice of ramified $p$-adic logarithm $\ell\colon E^\times_\A/E^\times \to\Q_p$  and Hodge splittings $(W_{v})_{v\vert p}$  as above, there is a symmetric bilinear paring\footnote{The notation of \cite{besser} is $D_1\cdot D_2$ for $\langle D_1, D_2\rangle^\Ar$.}
 $$\langle\ , \ \rangle^\Ar\colon  \CH^\Ar(\X)\times \CH^\Ar(\X)\to L$$
 satisfying:
 \begin{enumerate}
\item\label{arakelov}  If $D_1$ and  $D_2$ are finite and of degree zero on the generic fibre, and one of them has degree zero on each special fibre of $\X$, then
$$\langle D_1,D_2\rangle^\Ar= \langle D_{1,E}, D_{2,E}\rangle,$$
where $D_{i,E}\in\mathrm{Div}^0(X)$ is the generic fibre of $D_i$ and $\langle\ , \ \rangle$ denotes the height pairing of Proposition \ref{theo:ht} associated with the same choices of $\ell$ and $W_{v}$.
\item\label{arakelovfin}  If  $D_{1,\fin}$, $D_{2,\fin}$ have disjoint supports on the generic fibre, then
$$\langle D_1, D_2\rangle^\Ar=\sum_v \langle D_1,D_2 \rangle^\Ar_v$$
where the sum runs over all finite places of $E$, and the local Arakelov pairings are defined by
$$\langle D_1,D_2 \rangle^\Ar_v= (D_1,D_2)_v \ell_v(\pi_v)$$
for $v\nmid p$ and below for $v\vert p$.

If moreover we are in the situation of \ref{arakelov}, then for each place $v$ we have 
$$\langle D_1,D_2\rangle^\Ar_v= \langle D_{1,E}, D_{2,E}\rangle_v.$$
\item\label{arakelovdiv} In the situation of \ref{arakelovfin}, if moreover $D_1=\widehat{\mathrm{div}}(h)$ is the Arakelov divisor of a rational function $h$, then 
$$\langle D_1,D_2\rangle_v^\Ar=\ell_v(h(D_{2,\fin}))$$
for all places $v$.
\end{enumerate}
 \end{prop}
 
For completeness, we give the description of the local pairing at $v\vert p$ of divisors with disjoint supports. If $\ell_{v}=t_{v}\circ\log_{v}$ as in \eqref{tv} and $G_{v}$ is the Green function on $X_{v}\times X_{v}$, we have $\langle  D,X_{w} \rangle_{v}^{\Ar}=0$ if $v\neq w$, $\langle X_{v},X_{v}\rangle_{v}^{\Ar}=0$, $\langle D,\lambda_{v} X_{v}\rangle_{v}^{\Ar}=(\deg D_{E})t_{v}(\lambda_{v})$ and if $D_{1}$, $D_{2}$ are finite divisors with images $D_{1,v}=\sum n_{i}P_{i}$, $D_{2,v}=\sum m_{j}Q_{j}$ in $X_{v}$ then
$$\langle D_{1}, D_{2}\rangle^{\Ar}_{v}=\sum_{i,j} n_{i}m_{j}t_{v}(G_{v}(P_{i},Q_{j})).$$

In fact, in \cite{besser} it is proved directly that the global Arakelov pairing and its local components at $p$ coincide with the global and local height pairings of Coleman--Gross \cite{CG}. The latter coincide with the Zarhin--Nekov{\'a}{\v{r}} pairings by \cite{besserCG}.

 \section{Heegner points on Shimura curves}
 
In this section we describe our Shimura curve and construct Heegner points on it, following \cite[\S\S1-2]{shouwu}, to which we refer for the details (see also \cite[\S 5]{asian}, and \cite{carayol} for the original source of many results on Shimura curves). We go back to our usual notation, so $F$ is a totally real number field of degree $g$, $N$ is an ideal of $\OO_F$, $E$ is a CM extension of $F$ of discriminant $\Delta$ coprime to $2Np$, and $\eps$ is its associated Hecke character. 

 \subsection{Shimura curves}\label{shimuracurves}
 Let $B$ be a quaternion algebra over $F$ which is ramified at all but one infinite place. Then we can choose an isomorphism $B\otimes\R\cong M_2(\R)\oplus \mathbf{H}^{g-1}$, where $\mathbf{H}$ is the division algebra of Hamilton quaternions.
There is an  action of $B^\times$ on $\frakh^\pm=\C-\R$ by M\"obius transformations via the map $B^\times\to\mathbf{GL}_2(\R)$ induced from the above isomorphism. For each   open subgroup $K$ of $\widehat{B}^\times=(B\otimes_F\widehat{F})^\times$ which is compact modulo $\widehat{F}^{\times}$ we  then have a  \textbf{Shimura curve}
$$M_K(\C)=B^\times\bks \frakh^\pm\times \widehat{B}^\times/K,$$
where $\frakh^{\pm}=\C\setminus\R$.
Unlike modular curves, the curves $M_{K}$ do not have a natural moduli interpretation. However, by the work of Carayol \cite{carayol}, $M_{K}(\C)$ has a finite map\footnote{Which is an embedding if $K\supset \widehat{F}^{\times}$.} to another (unitary) Shimura curve $M_{K'}'(\C)$ which, if the level $K'$ is small enough, has an interpretation as the moduli space of certain quaternionic abelian varieties. Namely, $M'_{K'}$ parametrises isomorphism classes of abelian varieties of dimension $4[F:\Q]$ with multiplication by the ring of integers $\OO_{B'}$ of $B\otimes_{F}F'$ and some extra structure (a polarisation and a $K'$-level structure, compatible with the quaternionic multiplication) \cite[Proposition 1.1.5]{shouwu}. 

We will usually denote a point of $M'_{K'}$ simply by $[A]$ where $A$ is the underlying abelian variety.   If $K'$ has maximal components at places dividing $m$, one can define a notion of  an \emph{admissible submodule} $D$ of level $m$ (see \cite[\S 1.4.3]{shouwu}): it is an $\OO_{B'}$-submodules of $A[m]$ satisfying a certain condition, which ensures that the quotient $A/D$ can be naturally endowed with the extra structure required by the functor $M'_{K'}$.  We denote  by $[A_{D}]$ the object whose underlying abelian variety is $A/D$, with the induced extra structure. 

As a consequence of the moduli interpretation, the curve $M_K(\C)$ has a canonical model $M_K$ defined over $F$ (it is connected but not, in general, geometrically connected), and a proper regular integral model\footnote{In the modular curve case $F=\Q$, $\eps(v)=1$ for all $v\vert N$, $M_{K}$ and $\mathcal{M}_{K}$ are proper only after the addition of finitely many cusps. (We caution the reader that Carayol \cite{carayol} uses the notation $\calM_K$  to denote instead the set of geometrically connected components of $M_K$.)} $\mathcal{M}_K$ over $\OO_F$; if $v$ is a finite place where $B$ is split, then $\mathcal{M}_K $ is smooth over $\OO_{F,v}$ if $K_v$ is a maximal compact subgroup of $B_v$ and $K^v$ is sufficiently small.  We denote  $\mathcal{M}_{K,v}=\mathcal{M}_{K}\otimes\OO_{F,v}$.

\subsubsection{Universal formal group and ordinary points} Assume that the level structure $K$ is maximal at $\wp$.  The curve $\mathcal{M}_{K,\wp}$ carries a universal $\wp$-divisible $\OO_{B,\wp}$-module $\mathcal{G}$ obtained from the $\wp$-divisible group $\mathcal{A}[\wp^{\infty}]$ of the universal abelian scheme $\mathcal{A}$ over $\mathcal{M}'_{K',\wp}$. More precisely, choosing an auxiliary quadratic field $F'$ which is split at $\wp$ and an isomorphism $j\colon \OO_{F',\wp}\cong \OO_{F,\wp}\oplus \OO_{F,\wp}$,  we have $$\mathcal{G}=\mathcal{A}[\wp^{\infty}]^{(2)}=e_{2}\mathcal{A}[\wp^{\infty}],$$ where $e_{2}$ is the idempotent in $\OO_{F',\wp}$ corresponding to $(0,1)$ under $j$.

Assume that $B$ is split at $\wp$. Then we denote by $\mathcal{G}^{1}$, $\mathcal{G}^{2}$ the images under the projectors corresponding to $\smalltwomat 1{}{}{0}$ and $\smalltwomat 0{}{}1$ under a fixed isomorphism $B\cong M_{2}(F_{\wp})$; they are isomorphic via the element $\smalltwomat{}1{-1}{}$. 

Let $x$ be a geometric point of the special fibre $\mathcal{M}_{K,\wp}$. Then the $\mathcal{G}^{i}_{x}$ are  divisible  $\OO_{\wp}$-modules of dimension $1$ and height $2$, hence isomorphic to either of:
\begin{itemize}
\item the direct sum $\Sigma_{1}\oplus F_{\wp}/\OO_{F,\wp}$, where $\Sigma_{1}$ is the unique formal $\OO_{F,\wp}$-module of height $1$ -- in this case $x$ is called \emph{ordinary};
\item the unique formal $\OO_{\wp}$-module of dimension~$1$ and height~$2$ -- in this case $x$ is called \emph{supersingular}.
\end{itemize}

Let $M_{K}(\baar{\Q}_{\wp})^{\rm ord}\subset M_{K}(\baar{\Q}_{\wp})$ be the set of points with ordinary reduction. Then the Frobenius map $\mathrm{Frob}_{\wp}$ admits a lift 
\begin{gather}\label{froblift} \varphi\colon M_{K}(\baar{F}_{\wp})^{\rm ord}\to M_{K}(\baar{F}_{\wp})^{\rm ord}\end{gather}
given in the moduli interpretation by $[A]\mapsto [A_{{\rm can}(A)}]$, where ${\rm can}(A)$ is the \emph{canonical submodule} of $A$, that is the sub-$\OO_{F,\wp}$-module of $A[\wp]$ in the kernel of multiplication by $\wp$ in the formal group of $A$.

\subsubsection{The order $R$ and the curve $X$}   Assume that $\eps(N)=(-1)^{g-1}$. Then the  quaternion algebra $\mathbf{B}$ over $\A_F$ ramified exactly at all the infinite places and the finite places $v\vert N$ such that $\eps(v)=-1$ is \emph{incoherent}, that is, it does not arise via extension of scalars from a quaternion algebra over $F$. On the other hand,  for any embedding  $\tau\colon F\hookrightarrow\R$, there is a \emph{nearby} quaternion algebra $B(\tau)$ defined over $F$ and ramified at $\tau$ and the places where $\mathbf{B}$ is ramified. Fix any embedding $\rho\colon E\to B(\tau)$, and let $R$ be an order of $\widehat{B}=\widehat{B}(\tau)$ which contains $\rho(\OO_E)$ and has discriminant $N$ (this is constructed in \cite[\S 1.5.1]{shouwu}). Then the curve $X$ over $F$  of interest to us is the (compactification of) the curve $M_K$ defined above for  the subgroup $K=\widehat{F}^\times\widehat{R}^\times\subset \widehat{B}$; that is, for each embedding $\tau\colon  F\to \C$, we have 
\begin{equation}\label{X}
X(\C)=B(\tau)^\times\bks \frakh^\pm\times \widehat{B}^\times/\widehat{F}^\times\widehat{R}^\times \cup\mathrm{\{cusps\}}.
\end{equation}
The finite set of cusps is nonempty only in the classical case where $F=\Q$, $\eps(v)=1$ for all $v\vert N$  so that $X=X_0(N)$. In what follows we will not burden the notation with the details of this particular case, which poses no additional difficulties and  is already treated in the original work of  Perrin-Riou \cite{PR}.

We denote by $\X$ the canonical model of $X$ over $\OO_{F}$, and by $\X_{v}$ its base change to $\OO_{F,v}$. We also denote by $J(X)$ the Albanese variety of $X$ and by $\mathcal{J}_{v}$ its N\'eron model over $\OO_{F,v}$.

\subsubsection{Hecke correspondences}  
Let $m$ be an ideal of $\OO_{F}$ which is coprime to the ramification set of $B$. Let $\gamma_{m}\in \widehat{\OO}_{B}$  be an element with components $1$  away from $m$ and such that $\det \gamma_{m}$ generates $m$ at the places dividing $m$. Then the Hecke operator $T(m)$ on $X$ is defined by 
$$T(m)[(z,g)]=\sum_{\gamma\in K \gamma_{m}K/K}[(z, g\gamma)]$$
under the complex description \eqref{X}. When $m$ divides $N$ we often denote the operator $T(m)$ by $U(m)$ or $U_{m}$.

Let ${\mathbf T}_{N}'$ be the algebra generated by the $T(m)$ for $m$ prime to $N$. Then by \cite[Theorem 3.2.1]{shouwu}, the algebra 
${\mathbf T}_{N}'$ is a quotient of the Hecke algebra on Hilbert modular forms ${\mathbf T}_{N}$ (hence the names $T(m)$ are justified). It acts by correspondences on $X\times X$, and taking Zariski closures of cycles on $\X\times \X$ extends the action to $\X$.

As in the classical case, the Hecke operators $T(m)$ admit a moduli interpretation, after base change to a suitable quadratic extension $F'$ and passing to the curve $X'$. Namely we have
$$T(m)[A]=\sum_{D}[A_{D}],$$
where the sum runs over the {admissible submodules} of $A$ of level $m$.

\subsection{Heegner points}\label{sec:heegner}
The curve $X$ defined above has a distinguished collection of  points defined over abelian extensions of $E$: we briefly describe it, referring the reader to \cite[\S 2]{shouwu}  for more details. 

  A point $y$ of $X$ is called a \textbf{CM point} with multiplication by $E$ if it can be  represented by $({x_0}, g)\in\frakh^+\times\widehat{B}^\times$ via \eqref{X}, where $x_{0}\in \mathfrak{H}^{+}$ is the unique point fixed by $E^{\times}$.
  The order 
$$\End(y)=g\widehat{R}g^{-1}\cap \rho(E)$$
in $E=\rho(E)$ is defined independently of the choice of $g$, and  $$\End(y)=\OO_E[c]=\OO_F+c\OO_E$$ for a unique ideal $c$ of $\OO_F$ called the \textbf{conductor} of $y$.  We say that the point $y=[({x_0},g)]$ has the \emph{positive orientation} if for every finite place $v$ the morphism $t\mapsto g^{-1}\rho(t)g$ is $R_v^\times$-conjugate to $\rho$ in $\Hom(\OO_{E,v},R_v)/R_v^\times$.\footnote{This set has two elements only if $v\vert N$ (the other element is called the negative orientation at $v$); otherwise it has one element and the condition at $v$ is empty. There is a group of Atkin-Lehner involutions acting transitively on orientation classes.} Let $Y_c$  be the  set of positively oriented CM points of conductor $c$. By the work of Shimura and Taniyama, it is a finite  subscheme of $X$ defined over $E$, and the action of $\Gal(\baar{\Q}/E)$ is given by 
$$\sigma([({x_0},g)])=[({x_0},\mathrm{rec}_E(\sigma)g)],$$
where $\mathrm{rec}_E\colon \Gal(\baar{E}/E)\to\Gal(\baar{E}/E)^\mathrm{ab}\wtil{\to} \baar{E^\times}\bks \widehat{E}^\times$ is the reciprocity map of class field theory. If $y=[({x_0},g)]$ has conductor $c$, then the action factors through
$$\Gal(H[c]/E)\cong  E^\times\bks \widehat{E}^\times/\widehat{F}^\times\widehat{\OO}_E[c]^\times$$
where $H[c]$ is the ring class field of $E$ of conductor~$c$; the action of this group on $Y_c$ is simply transitive.

For each nonzero ideal $c$ of $\OO_F$, let $u(c)=[\OO_E[c]^\times :\OO_F^\times]$ and define the divisor 
\begin{gather}\label{eta} \eta_c=u(c)^{-1}\sum_{y\in Y_c} y.\end{gather}
Let $\eta=\eta_1$. By the above description of the Galois action on CM points , each divisor $\eta_c$ is defined over $E$.

A \textbf{Heegner point} $y\in X(H)$ is a positively oriented CM point with conductor~$1$.   We can use the embedding $\iota\colon X\to J(X)\otimes\Q$ to define the point 
$$[z]=\iota(\eta)=[\eta]-h[\xi]\in J(X)(E)\otimes\Q$$
where $h$ is a number such that $[z]$ has degree zero in each geometrically connected component of $X$, and $[\xi]$ is the Hodge class of the Introduction (see below for more on the Hodge class).

 \subsubsection{Arakelov Heegner divisors} 
The Heegner divisor on $X$ can be refined to an Arakelov divisor $\hat{z}$ having degree zero on each irreducible component of each special fibre.  On a suitable Shimura curve $\wtil{X}\overset{\pi}{\to} X$ of deeper level away from $N\Delta_{E/F}$, we can give an explicit description of the pullback $\hat{\wtil{z}}$ of $\hat{z}$ and of the Hodge class as follows. 

As outlined in in \S\ref{shimuracurves}, after base change to a suitable quadratic extension $F'$ of~$F$, we have an embedding $\wtil{X}\hookrightarrow  \wtil{X}'$ of $\wtil{X}=M_{\wtil{K}}$ into the unitary Shimura curve $\wtil{X}'=M_{\wtil{K'}}'$ parametrising  abelian varieties of dimension $4g$ with multiplication by $\OO_{B'}$ and some extra structure. Then  by the Kodaira--Spencer map, we have an isomorphism $\omega_{\wtil{X}'}\cong \det \Lie \calA^{\vee}|_{\wtil{X}'}$, where $\calA\to \wtil{X}'$ is the universal abelian scheme and the determinant is that of an $\OO_{F'}$-module of rank~$4$ (the structure of $\OO_{F'}$-module coming from the multiplication by $\OO_{B'}$ on $\calA$). This gives a way\footnote{See \cite[\S4.1.3, \S1]{shouwu} for more details on this construction.} of extending the line bundle $\omega_{\wtil{X}'}$ to the integral model $\wtil{\X}'$ and to a line bundle $\calL$ on $\wtil{\X}$.  For each finite place $v\vert p$ we endow $\calL|_{\wtil{X}_{v}}$ with the canonical log functions $\log_{\calL,v}$ coming from the description $\calL|_{\wtil{X}_{v}}=\omega_{\wtil{X}_{v}}$ and a fixed choice of Hodge splittings on $\wtil{X}$. We define $[\hat{\wtil{\xi}}]\in \CH^{\Ar}(\wtil{\X})\otimes \Q$ to be the class of $(\calL, (\log_{\calL})_{v\vert p})$ divided by its degree, $[\wtil{\xi}]$ to be its finite part, and $\hat{\wtil{\xi}}$ to be any Arakelov divisor in its class. 

Then the Arakelov Heegner  divisor $\hat{\wtil{z}}\in \Div^{\Ar}(\X\otimes \OO_{E})$ is described by 
\begin{align}\label{zhat}
\hat{\wtil{z}}=\hat{\wtil{\eta}}-h\hat{\wtil{\xi}}+Z,
\end{align}
where $\hat{\wtil{\eta}}$ is the Zariski closure in $\X\otimes \OO_{E}$ of the pullback of $\eta$ to $\wtil{X}$, and $Z$ is a finite vertical divisor uniquely determined by the requirement that $\hat{\wtil{z}}$ should have degree zero on each irreducible component of each special fibre.

\subsection{Hecke action on Heegner points}\label{sec:heckeheegner} Recall from \S\ref{sec:calS}  the spaces of Fourier coefficients $\mathcal{D_N}\subset{\calS}$, the arithmetic functions $\sigma_1, r\in\mathcal{D}_N$, and the space $\bcalS=\calS/\mathcal{D}_N$. The action of Hecke operators on the  Arakelov Heegner divisor  is described as follows. 

\begin{prop}\label{heckeheegner} Let $m$ be an ideal of $\OO_F$ coprime to $N$. We have
\begin{enumerate}
\item $T(m)\eta=\sum_{c|m} r(m/c)\eta_{c}$.
\item Let  $\eta^0_c=\sum_{\OO_F\neq d|c}\eta_d$, and let $T^0(m)\eta=\sum_{c|m}\eps(c)\eta^0_{m/c}$. Then $\eta$ and $T^0(m)\eta$ have disjoint support and if $m$ is prime to $N\Delta$ then $T(m)\eta=T^0(m)\eta+r(m)\eta$.
\item\label{heckexi} $T(m)[\xi]=\sigma_1(m)[\xi]$ and  $m\mapsto T(m)\hat{\wtil{\xi}}$ is  zero in  $\baar{\calS}\otimes \Div^\Ar(\wtil{\X})$.
 % $T(m)[\xi]=\sigma_1(m)[\xi]$ and $m\mapsto T(m)[\hat{\wtil{\xi}}]$ is zero in  $\baar{\calS}\otimes \CH^\Ar(\wtil{\X})$.
\item The arithmetic function $m\mapsto T(m)Z$ is zero in $\baar{\calS}\otimes \Div^\Ar(\X)$.
\end{enumerate}
\end{prop}         
\begin{proof} Parts 1, 2 and 4 are proved in \cite[\S 4]{shouwu}. For part 3, we switch to the curve $\wtil{X}$. By definition  $[\hat{\wtil{\xi}}]$ is  a multiple of the class of the Arekelov line bundle $\calL=\det\Lie \calA^{\vee}$ on $\wtil{\X}$ with the canonical log functions on $\calL_{v}\cong\omega_{\wtil{X}_{v}}$, where $\calA\to\X$ is the universal abelian scheme. We view $T(m)$ as a finite  algebraic correspondence of degree $\sigma_{1}(m)$ induced by the subscheme $\wtil{\X}_{m}\subset \wtil{\X} \times \wtil{\X}$ of pairs $(A,A/D)$ where $D$ is an admissible submodule of $A$ of level $m$. If $p_{1}, p_{2}\colon\wtil{\X}_{m}\to\wtil{\X}$ are the two projections, then we have
$$T(m)\calL=N_{p_{1}}p_{2}^{*}\calL,$$
and the log functions $\log_{T(m)\calL_{v}}$ on $T(m)\calL|_{\wtil{X}_{v}}$ are the ones induced by this description. (That these are genuine log functions -- cf. the caveat in \S\ref{sec:arakelov} -- will be shown in the course of proving Proposition \ref{heckeheegner}.\ref{heckexi} 
below.)
 
Let $\pi\colon \calA_{1}\to \calA_{2}$ be the universal isogeny over $\wtil{\X}_{m}$. As $p_{i}^{*}\calL=\det\Lie\calA_{i}^{\vee}$, we have an induced map 
$$\psi_{m}=N_{p_{1}}\pi^{*}\colon T(m)\calL\to N_{p_{1}}p_{1}^{*}\calL=\calL^{\sigma_{1}(m)},$$
and \cite[\S 4.3]{shouwu} shows that $\psi_{m}(T(m)\calL)=c_{m}\calL^{\sigma_{1}(m)}$ where $c_{m}\subset \OO_{F}$ is an ideal with divisor $[c_{m}]$ on $\Spec\OO_{F}$ such that $m\to [c_{m}]$ is a $\sigma_{1}$-derivative (\S\ref{sec:calS}), hence zero in $\bcalS\otimes \Div\,(\Spec\OO_{F})\subset\bcalS\otimes\Div^{\Ar}(\X)$. In fact if the finite divisor $\hat{\wtil{\xi}}_{\rm fin}={\rm div}(s)$ for a rational section $s$ of $\mathcal{L}$, the same argument shows that $T(m)\hat{\wtil{\xi}}_{\rm fin}={\rm div}(T(m)s)=\sigma_{1}(m){\rm div}(s)+{\rm div}(c_{m})$, hence $m\mapsto \hat{\wtil{\xi}}_{\rm fin}$ is zero in $\bcalS\otimes \Div^{\Ar}(\X)$.

We complete the proof by showing that, for each $v\vert p$, the difference of log functions 
\begin{align}\label{diff}
\psi_{m}^{*}\log_{\calL_{v}^{\sigma_{1}(m)}}-\log_{T(m)\calL_{v}}
\end{align}
on the line bundle $T(m)\calL_{v}$ on $\wtil{X}_{v}$ is a constant on the total space of $\calL_{v}$, and it is a $\sigma_{1}$-derivative when viewed as a function of $m$.  (In particular this shows that $\log_{T(m)\calL_{v}}=\sigma_{1}(m)\psi_{m}^{*}\log_{\calL_{v}}+ {\mathrm{\ constant}}$ is a genuine log function.)

It is enough to show this after pullback via $p_{1}$ on $\wtil{X}_{m}$, where (denoting pulled back objects with a prime) the map $\psi'_{m}$ decomposes as 
$$\psi'_{m}=\otimes_{D}\pi_{D}^{*} \colon \otimes_{D}\det\Lie (\calA'/D)^{\vee}\to(\det\Lie \calA'^{\vee})^{\otimes \sigma_{1}(m)}$$
where the tensor product runs over admissible submodule schemes of level $m$ of $\calA'$ (since base change via $p_{1}$ splits the cover $p_{1}$, there are exactly $\sigma_{1}(m)$ of those). Now the difference \eqref{diff} is the sum of the $\sigma_{1}(m)$ differences 
$$(\pi_{D}^*)^{*}\log_{\calL}-\log_{\calL},$$
which are all the same since they are permuted by the Galois group of $p_{1}$.
As $\pi_{D}^{*}$ acts by multiplication by $(\deg\pi_{D})^{1/2}=\N(m)^{2}$,  by \eqref{logv} each of these differences is $2\log_{v}\N(m)$ so  that \eqref{diff} equals
$$2\sigma_{1}(m)\log_{v}\N(m)$$
which is indeed a $\sigma_{1}$-derivative.
 \end{proof}

 \section{Heights of Heegner points}\label{sec7}
 Let  $\Psi$ be the modular form of level $N$ with  Fourier coefficients given by the $p$-adic height pairing $\langle z, T(m) z\rangle$ (it is a modular form because of Lemma \ref{ismodular} and the fact that the quaternionic Hecke algebra ${\bf T}_{N}'$ is a quotient of ${\bf T}_{N}$, as explained at the end of  \S\ref{shimuracurves}). 
We will compute the heights of Heegner points, with the goal of showing (in \S\ref{proofmt}) that   $\Lf(\Phi')$ and $\Lf(\Psi)$ are equal  up to the action of some Hecke operators.  The main theorem will follow. 
 
 \medskip

The strategy is close to that of Perrin-Riou, namely we separate the local contributions to $\Psi$  from primes above $p$, writing $\Psi\sim \Psi_\mathrm{fin}+\Psi_p$; using the computations of \cite{shouwu, asian}  we find an explicit expression for   $\Psi_\mathrm{fin}$, which in \S\ref{proofmt} we will show to be   ``almost'' equal to the expression for $\Phi'$, while the contribution of $\Psi_p$ is shown to vanish. We circumvent the difficulties posed by the absence of cusps through the use of  $p$-adic Arakelov theory. 

It will be crucial to work in the quotient spaces $\bcalS$, $\bcalS^{\rm ord}$ introduced in \S\ref{sec:Lf}; for the convenience of the reader we copy here the diagram \eqref{quotient} which summarises the relations among them.   
\begin{equation*}
\xymatrix{
{\bf S}_{N}(L) \ar[r]\ar[d]^{e} & \bcalS_{N}^{p{\rm -adic}}(L) \ar[d] \ar@{^{(}->}[r] &  \bcalS^{\rm ord} \\
S_{NP}^{\rm ord}(L)/S_{NP}^{N{\rm -old}} \ar[r]^{\wtil{}} & \bcalS_{N}^{\rm ord}(L) \ar[r]^{\Lf} & L
}
\end{equation*}
We will abuse notation by using the same name for a modular form and its image in $\bcalS_N^{\rm ord}$.

The height pairings  $\langle\, ,\, \rangle$ (and the accompanying Arakelov pairings) on the base change of $X$ to $E$  that will be considered are the ones associated with  a ``cyclotomic'' $p$-adic logarithm given by $\ell=\ell_F\circ \frakN\colon E^{\times}\bks E_{\A^{\infty}}^{\times}\to\Q_{p}$ for some\footnote{In our application, we will take $\ell_{F}={d\over ds}|_{s=0}\nu^{s}$ for a character $\nu\colon \calG_{F}\to1+p\Z_{p}$.} 
$$\ell_{F}\colon  F^{\times}\bks F_{\A^{\infty}}^{\times}\to\Q_{p},$$
and to choices of Hodge splittings on $V_{v, L}=H^{1}_{\rm dR}(X_{v}/E_{v})\otimes L$ ($v\vert p$) such that on $e_{f}V_{v, L}\cong e_{f}M_{\mathcal{J}, L}$, the induced Hodge splitting is the unit root splitting.

As mentioned before, the Shimura curve $X$ and its integral model $\X$ may not be fine enough for the needs of Arakelov and intersection theory, so that we may need to pass to a Shimura curve  $\wtil{\X}\overset{\pi}{\to}\X$ of deeper level away from $p$ and consider the pullbacks $\wtil{\eta}$ of the divisors $\eta$, etc. Then  notation such as  $\langle \hat{\eta}, T^0(m)\hat{\eta}\rangle^\Ar$ is to be properly understood as $\langle \hat{\wtil{\eta}}, T^0(m)\hat{\wtil{\eta}}\rangle^\Ar/\deg\pi$.

\subsection{Local heights at places not dividing $p$}
 The next two results will be used to show the main identity.
 
\begin{lemm}\label{discard}    In the space $\bcalS$ we have $$\langle z, T(m) z\rangle=\langle {\hat z}, T(m) \hat{z}\rangle^\Ar\sim \langle {\hat \eta}, T^0(m)\hat{ \eta}\rangle^\Ar.$$
\end{lemm}
 \begin{proof}
 First observe that  by Lemma \ref{ismodular}, the first member is a modular form of level $N$, so it does indeed belong to $\bcalS_N$. 
 The first equality is a consequence of Propositon \ref{ar}.\ref{arakelov} and the construction of $\hat{z}$. The second part follows from expanding the second term for $m$ prime to $N\Delta$ according to  \eqref{zhat} and observing that the omitted terms  are zero in $\bcalS$ by Proposition \ref{heckeheegner}.
 \end{proof}
 
We can therefore write
\begin{equation}\label{finp} \Psi\sim\sum_{w}\Psi_{w}=\sum_{v}\Psi_{v}=\Psi_\fin+\Psi_p
\end{equation}
in $\bcalS$, with the first sum running over the finite places $w$ of $E$, the second sum running over the finite places $v$ of $F$, and
 \begin{gather*} 
 \Psi_w(m)= \langle {\hat \eta}, T^0(m)\hat{ \eta}\rangle^\textrm{Ar}_w,\quad
   \Psi_{v}=\sum_{w|v}\Psi_{w},\quad
  \Psi_{\rm fin}=\sum_{v\nmid p}\Psi_{v},\quad  \Psi_p=\sum_{v\vert p} \Psi_v.
 \end{gather*}
 (We are exploiting the fact that for $m$ prime to $N\Delta$ the divisors $\hat{\eta}$ and $T^0(m)\hat{\eta}$ have disjoint supports so that we can apply Proposition \ref{ar}.\ref{arakelovfin}.)

\medskip

For each prime $\wp$ of $F$ above $p$, we define an operator\footnote{This is different from the operator bearing the same name in \cite{PR}.} on $\calS$
$$\mathcal{R}_\wp=U_\wp-1,\quad\calR_p=\prod_{\wp\vert p}\calR_\wp.$$
We also define, for  integers $\mu_{\wp}\geq 1$,   operators
$$\calR_{\wp}^{(\mu_{\wp})}=U_{\wp}^{\mu_{\wp}}-1, \quad \calR_{p}^{(\mu)}=\prod_{\wp\vert p}\calR_{\wp}^{(\mu_{\wp})}
 .$$ 
\begin{prop}\label{Psi} In the space $\bcalS$ 
 we have
$$\Psi_\fin \sim\sum_{v\nmid p} \Psi_v + h,$$
where  $h$ is a  modular form which is killed by $\Lf$; 
the sum runs over the finite places of~$F$ and the summands are given by:
\begin{enumerate}
\item If $v=\wp$ is inert in $E$, then 
$$\Psi_v(m)=
 \sum_{\substack{n\in Nm^{-1}\Delta^{-1}\\  \eps_{v}((n-1)n)=1\ \forall v\vert \Delta \\ 0<n<1}}
 2^{\omega_{\Delta}(n)} 
 r((1-n)m\Delta) r(nm\Delta/N\wp)(v(nm/N)+1)\ell_{F,v}(\pi_{v}). $$ 
\item If $v=\wp\vert \Delta$ is ramified in $E$, then
$$\Psi_v(m)= \sum_{\substack{n\in Nm^{-1}\Delta^{-1}\\  \eps_v((n-1)n)=-1\\ \eps_w((n-1)n)=1\ \forall v\neq w|\Delta \\ 0<n<1}} 
 2^{\omega_{\Delta}(n)} 
 r((1-n)m\Delta) r(nm\Delta/N)(v(nm)+1)\ell_{v}(\pi_{v}).   $$
\item If $v$ is split in $E$, then
$$\Psi_{v}(m)=0.$$
\end{enumerate}
\end{prop} 
 \begin{proof} 
 %That $\Psi_\fin$ belongs to $\bcalS_N$ follows from \eqref{finp} and the previous proposition. ,
For $m$ prime to $N\Delta$ we have $\Psi_\fin(m)=\sum_{w\nmid p} \langle {\hat \eta}, T^0(m)\hat{\eta}\rangle^\textrm{Ar}_w$  (the sum running over all finite places  $w$ of $E$). By Proposition \ref{ar}.\ref{arakelovfin}, up to the factor $\ell_{F,v}(\pi_{v})$ (which equals $\ell_{w}(\pi_{w})$ or its half   for each place $w$ of $E$ above $v$),  each term is given by an intersection multiplicity $({\hat \eta}, T(m) \hat{\eta})_w$, which is computed by Zhang.

When $v(N)\leq 1$ for all $v$ which are not split in $E$,  the result is summarised in \cite[Proposition 5.4.8]{shouwu}; in this case, the values obtained there are equivalent to the asserted ones  by \cite[Proposition 7.1.1 and Proposition 6.4.5]{shouwu}, and there is no  extra term $h$. In fact (and with no restriction on $N$), these values also appear  as the local components ${}^{\C}\Phi'_{v}$ at finite places of a form ${}^{\C}\Phi'$ of level~$N$ which is a kernel of the Rankin--Selberg convolution for the central derivative $L'(f_{E},1)$ of the complex $L$-function. 

In general, \cite[Lemma 6.4.3]{asian}  proves that\footnote{We are adapting the notation to our case. In \cite{asian}, the form $f$ is denoted by $\phi$, the functions ${}_{v}h$ are denoted by ${}_{v}f$.}
\begin{align}\label{asiancomp} 
{\Psi_{v}\over \ell_{F,v}(\pi_{v})} \sim {  {}^{\C}\Phi_{v}'^{\sharp}\over \log \N(\wp_{v})} +{}_{v}h,
\end{align}
 where  ${}_{v}h$ is a modular form with zero projection onto the $f$-eigenspace (see the discussion at the very end of \cite{asian}; the forms ${}_{v}h$ come from intersections at bad places), and ${}^{\C}\Phi'^{\sharp}$ is a form of level $N\Delta$ which is a kernel for the complex Rankin--Selberg convolution in level $N\Delta$ (in particular, it is modular and ${\rm Tr}_{\Delta}({}^{\C}\Phi'^{\sharp})={}^{\C}\Phi' + h'$ where $h'$ is a modular form of level $N$ which is orthogonal to $f$). Applying the operator ${\rm Tr}_{\Delta}$ in \eqref{asiancomp} we recover the asserted formula.
\end{proof}

\subsection{Local heights at  $p$ / I}\label{sec:hp}\label{7.2}
The following is the key result concerning the local heights at places dividing $p$. We assume that all primes $\wp$ of $F$ dividing $p$ are split in $E$.
\begin{prop}\label{killp} The arithmetic function $\mathcal{R}^4_p\Psi_p$ belongs to $\bcalS_N^{\rm ord}\subset \bcalS^{\rm ord}$, and we have
$$\Lf(\mathcal{R}^4_p\Psi_p)=0.$$
\end{prop}

The modularity assertion follows, as in \cite{nekovar}, by difference from the modularity of $\Psi$ (hence of $\mathcal{R}_{p}^{4}\Psi$) and the modularity of $\mathcal{R}_{p}^{4}\Psi_{\rm fin}$ proved in Proposition \ref{psifinetc} below. 
%(In fact one could likely show the stronger result that $\calR_{\wp}^{3}\Psi_{v}$ is modular for each $v=\wp\vert p$, with a method similar to that of \cite[Lemme 5.8]{PR}.)

The proof of the vanishing of the $f_{\alpha}$-component will be completed in \S\ref{sec:hp2} using the results of the rest of this subsection.

\medskip

We start by fixing  for the rest of this section  a prime $\wp$ of $F$ dividing $p$.  
Fix an isomorphism $B_{\wp}=B\otimes_{F}{F_{\wp}}\cong M_{2}(F_{\wp})$ identifying the local order $R_{\wp}$ with $M_{2}(\OO_{F,\wp})$, and the field $E\subset B$ with the diagonal matrices in $M_{2}(F_{\wp})$. 
Let the divisors $\eta_c$ be as in \eqref{eta}, and denote
$$H_s=H[\wp^s],\qquad   u_s=u(\wp^s).$$
Let $y_{s}\in X(H_{s})$ be the CM point of conductor $\wp^{s}$ 
defined by 
$$y_{s}=\left[\left({x_0}, \, \iota_{\wp}\twomat{\pi^{s}}{1}{}{1}\right)\right],$$
where $\iota_{\wp}\colon \GL_{2}(F_{\wp})\to \widehat{B}^{\times}$ is the natural inclusion, and $\pi$ is a uniformiser at~$\wp$. 

Fix a place $w$ of $H$ above $\wp$; we still denote by $w$ the induced place on each $H_{s}$, and by $\frakp$ the prime of $E$ lying below $w$. Since $\wp$ splits in $E$, by \cite[\S2.2]{shouwu} the CM points $y_{s}=[A_{s}]$ are ordinary,  and their canonical submodules with respect to the reduction modulo $w$  are given by $A_{s}[\frakp]$.

\begin{prop}[Norm relations]\label{NR} Let $y_{s}$ be the system of CM points defined above.
\begin{enumerate}
\item\label{NR1} Let $m=m_0\wp^n$ be an ideal of $F$ with $m_0$ prime to $\wp N$.  We have 
\begin{gather*} [T(m\wp^{r+2})-2 T(m\wp^{r+1})+T(m\wp^r)](\eta) = u_{n+r+2}^{-1} T(m_0)\Tr_{H_{n+r+2}/E}(y_{n+r+2})\end{gather*}
as divisors on $X$.
\item\label{NR2} For all $s\geq 1$, we have
$$T(\wp)\, y_{s}= {\rm Tr}_{H_{s+1,w}/H_{s,w}}(y_{s+1})+y_{s-1}.$$
\item\label{NR3} For all $s\geq 1$, we have
$$\varphi(y_{s})=y_{s-1},$$
where $\varphi$ is the lift \eqref{froblift} of Frobenius with respect to the reduction modulo $w$.
\end{enumerate}
\end{prop}
\begin{proof} By the multiplicativity of Hecke operators it is enough to prove the statement of part \ref{NR1} for $m_0=1$. A  simple computation based on Proposition \ref{heckeheegner} shows that the left-hand side is equal to $\eta_{\wp^{n+r+2}}$.  Since the Galois action of $\Gal(H_{n+r+2}/E)$ is simply transitive on $Y_{\wp^{n+r+2}}$, the right-hand side is also equal to $\eta_{\wp^{n+r+2}}$. 

For part \ref{NR2},  use the notation $[g]$ to denote $[({x_0},\iota_{\wp}(g))]$. Then we have
$$T(\wp)\, y_{s}=\sum_{j\in \OO_{F,\wp}/\wp}\left[ \twomat {\pi^{s}}1{}1\twomat{\pi}j{}1\right]+ \left[ \twomat {\pi^{s}}1{}1\twomat1{}{}{\pi}\right].$$
The last term is identified as $y_{s-1}$ after acting by the diagonal matrix $\pi^{-1}{\rm id}$ (whose action is trivial on $X$). On the other hand, by local class field theory and the description of the Galois action on CM points of \S\ref{sec:heegner}, we have
$${\rm Tr}_{H_{s+1,w}/H_{s,w}}(y_{s+1})=\sum_{j\in\OO_{F,\wp}/\wp}\left[\twomat {1+j\pi^{s}}{}{}1\twomat {\pi^{s+1}}1{}1 \right],$$
which is the same as the above sum in $j$.

For part \ref{NR3}, which in fact is not needed in what follows, we switch to the moduli description,\footnote{As usual, after base change to a suitable quadratic extension $F'$.} so $y_{s}=[A_{s}]=[A_{D_{s}}]$ for an increasing sequence of admissible submodules $D_{s}$ of level $\wp^{s}$
 (this follows from part \ref{NR2}, together with a variant for $s=0$ that we omit, and the moduli description of Hecke correspondences).  
 %fixme: do not omit! and explain why follows [T_{p} increases the level of the submodule....]
 Now $D_{1}$ is different from $\can(A)=A[\frakp]$ since $[A_{A[\frakp]}]$ has conductor $1$; and it in fact  each $D_{s}$ does not contain $A[\frakp]$, since if it did then $[A_{s}]$ would be in the support of $T(\wp)^{s-1}[A_{A[\frakp]}]$ which is easily seen\footnote{By the following observation: if $y$ is a CM point of conductor $c$, then the support of $T(m)\, y$ consists of CM points of conductors dividing $cm$.} to consist of CM points of conductor dividing $\wp^{s-1}$. It follows that the point $\varphi([A_{s}])=[A_{D_{s}+\can(A_{s)}}]=[A_{D_{s}+A[\frakp]}]$ is in the support of $T(\wp)[A_{s}]$, but  it is not one of the Galois conjugates of $y_{s+1}$ since as just seen it has lower conductor; by part \ref{NR2}, it must then be $y_{s-1}$.
\end{proof}

\begin{lemm}\label{killdiv} Let $w$ a place of $E$ dividing $\wp$, and let $h\in E_w(X)$ be a rational function whose reduction at $w$ is defined and nonzero.  Let $\mu=\mu_{\wp}$ be the order of the ideal $\mathfrak{p}_{w}$ in the relative class group  of ${E}/F$. 
Then the arithmetic functions
$$ \mathcal{R}^2_\wp\calR_{\wp}^{(\mu)} \langle \widehat{\divisor}(h), T^0(m) \hat{\eta}\rangle^\Ar_w, \quad  \mathcal{R}^3_\wp \langle {\divisor}(h), T(m)z\rangle_w$$
belong to the kernel of the $\wp$-partial ordinary projection $e_{\wp}$.
\end{lemm}
\begin{proof}
We  show more precisely that
\begin{align}\label{quello}
v(U_\wp^s\mathcal{R}_\wp^3 \langle \widehat{\divisor}(h), T^0(m)\hat{\eta}\rangle^\textrm{Ar}_w)\geq v(\N \wp^{s})- C
\end{align}
 for a uniform constant $C$, where $v$ is the $p$-adic valuation. We may assume $m$ prime to $\wp N\Delta$.  
 
For  the second expression,  under our assumptions we have $T(m\wp^{s})\eta=T^{0}(m\wp^{s})\eta+r(m\wp^{s})\eta-h\sigma_{1}(m\wp^{s})\xi$, so  the analogue of    \eqref{quello} holds with the same proof 
together with the observation that $\calR_{\wp}^{2}r(m)=0  $ and  $v(\sigma_{1}(m\wp^{s}))=v(\N\wp^{s})$. (Cf. \cite[Lemme 5.4]{PR}.) 

As  $\calR_\wp^2 r(m)=0$, Proposition \ref{NR}.\ref{NR1} gives 
$$U_\wp^s\calR_\wp^2 {\eta}= u_{s+2}^{-1}\mathrm{Tr}_{H_{s+2}/E} y_{s+2}$$
where $y_{s+2}\in Y_{\wp^{s+2}}$; we make a compatible choice of $y_{s}$ such as the one described above Proposition \ref{NR}.

For $s$ large enough the divisor of $h$ is supported away from $y_s$ and its conjugates. Then by Proposition \ref{ar}.\ref{arakelovdiv} we have 
  \begin{align*} U_\wp^s\calR_\wp^2 \langle \widehat{\divisor}(h), T^0(m) \hat{\eta}\rangle^\textrm{Ar}_w &=u_{s+2}^{-1}\ell_{w}(h(T^0(m)y_{s+2}))\\
  &=u_{s+2}^{-1} \sum_{w'|w} \ell_{w} (N_{H_{s+2,w'}/E_{w}} h(y_{s+2})),
 \end{align*}
where $w'$ runs over the places  of $H$ above $w$ (which are  identified with the places of $H_{s+2}$ above $w$, since $H_{s+2}/H$ is totally ramified above $\wp$).

For any $w'|w$ we have
$$\calR_\wp^{(\mu)}\ell_{w} (N_{H_{s+2,w'}/E_{w}}h(y_{s+2}))=\ell_{w} \circ N_{H_{w'}/E_{w}}( N_{H_{s+2+\mu,w'}/H_{w'}}h(y_{s+2+\mu})/ N_{H_{s+2,w'}/H_{w'}}h(y_{s+2})).$$
Suppose that (for $s$ large enough)
\begin{equation}\label{indept} \textrm{the $w'$-adic valuation of $N_{H_{s,w'}/H_{w'}}(h(y_{s}))$ only depends on the residue class of $s$ $(\mod \mu)$}.\end{equation}
Then each $w'$-summand  in the expression of interest is  the product of $u_{s+2}^{-1}$ (which is eventually constant in $s$) and the $p$-adic logarithm of a unit which is a norm from an extension of $E_w$ whose ramification degree is a constant  multiple of $\N\wp^s$; hence its $p$-adic valuation is also at least a constant multiple of  the valuation of $\N\wp^{s}$, which proves the Lemma.

It remains to prove  \eqref{indept}. We have 
\begin{equation}\label{congru0}
w'(N_{H_{s,w'}/H_{w'}}(h(y_{s})))=[H_{s,w'}:H_{w'}](\underline{(h)}, {\underline{y_{s}}}),
\end{equation}
where the pairing in the right-hand side denotes the intersection multiplicity of the Zariski closures in the integral model. Now as in \cite[Lemme 5.5]{PR}, if $\pi_{s}$ denotes a uniformiser of $H_{s,w'}$ we can show that we have 
\begin{equation}\label{congru}
\underline{y_{s}}\equiv \underline{ y_{s-\mu}} \mod \pi_{s},\qquad \underline{y_{s}}\not\equiv \underline{ y_{s-\mu}}\mod\pi_{s}^{2}.
\end{equation}
In fact, we first check that the two points have the same reduction. By \cite[Lemma 5.4.2]{asian}, the set of   points in the special fibre $\X\times_{\OO_{F,\wp}}\baar{\bf F}_{\wp}$ having CM by $E$  (and thus being ordinary, as $\wp$ splits in $E$) is identified with 
\begin{align}\label{modp}
E^{\times}\bks (N(F_{v})\bks \GL_{2}(F_{\wp}))\times \B^{\wp\infty \times}/\widehat{F}^{\times}\widehat{R}^{\times},
\end{align}
(where $N$ is the group of upper triangular unipotent matrices)
in such a way that the reduction map sends the CM point $[(x_{0},g)]\in X(\C)$ to the class of $g$. Then if $\sim$ denotes the equivalence relation in \eqref{modp}, and $t\in E^{\times}$ is a generator of the ideal $\mathfrak{p}_{w}^{\mu}a$ for some ideal $a$ of $\OO_{F}$ with ad\`elic generator $\pi_{a}$, the reduction of $y_{s}$ is the class of 
\begin{multline*}
\iota_{\wp}\left(\twomat {\pi^{s}} {1}{}1\right)\sim \iota_{\wp}\left(\twomat {\pi^{s}} {}{}1\right)\sim \iota_{\wp}\left(\twomat {\pi^{\mu}} {}{}1\twomat {\pi^{s-\mu}} {}{}1\right)\\
\sim t \,\iota_{\wp}\left(\twomat {\pi^{s-\mu}} {}{}1\right)(t^{\wp\infty})^{-1}\pi_{a}^{-1}\sim \iota_{\wp}\left(\twomat {\pi^{s-\mu}} {}{}1\right)\sim\iota_{\wp}\left(\twomat {\pi^{s-\mu}} {}{}1\right),
\end{multline*}
which is the same as the reduction of $y_{s-\mu}$.

We can then verify the congruence relation \eqref{congru} on the completed local ring $\widehat{\OO}_{\X/W(\baar{\bf F}_{v}),\baar{y}}$  of the common reduction $\baar{y}$; here $W$ is the ring of integers in the completion of the maximal unramified extension of $F_{v}$. By  \cite[5.5 Proposition]{carayol} this is the universal deformation ring of the $p$-divisible module $\mathcal{G}^{1}_{\baar{y}}$ (with the notation of \S\ref{shimuracurves}). 
%
% which is  on the $\wp$-divisible groups $\mathcal{A}_{\underline{y_{s}}}[\wp^{\infty}]$ or equivalently on $\mathcal{G}_{s}=\mathcal{G}_{\underline{y_{s}}}$, or  on $\mathcal{G}_{s}^{1}$ (with the notation of \S\ref{shimuracurves}).
As the  point $\baar{y}$ is ordinary, such module is isomorphic to  the product $F_{\wp}/\OO_{F,\wp}\times  \Sigma_{1}$, where $\Sigma_{1}$ is the Lubin--Tate formal $\OO_{F, \wp}$-module of height one.
%Then 
% ordinary means that 
%  $\calG_{s}^{1}=\mathcal{G}^{1}_{\underline{y_{s}}}$ is an extension of $F_{\wp}/\OO_{F,\wp}$ by the 
%% 
%% by Honda--Tate theory, it suffices to check the congruence relation \eqref{congru} on the $\wp$-divisible groups $\mathcal{A}_{\underline{y_{s}}}[\wp^{\infty}]$ or equivalently on $\mathcal{G}_{s}=\mathcal{G}_{\underline{y_{s}}}$, or  on $\mathcal{G}_{s}^{1}$ (with the notation of \S\ref{shimuracurves}). That  $y_{s}$ is ordinary means that  $\calG_{s}^{1}$ is an extension of $F_{\wp}/\OO_{F,\wp}$ by the 
%% 
% Lubin-Tate formal module  $\Sigma_{1}$.
 Now its lifting  $\calG_{s}^{1}=\calG_{\underline{y_{s}}}^{1}$ is defined precisely over the ring of integers of $H_{s,w'}$, and so it is a quasi-canonical lifting of level $s$ of its reduction $\calG_{\baar{y}}$, in the sense of Gross \cite{quasi}.
  Then by \cite[\S 6]{quasi} (see also \cite{argos} for a detailed account),  $\calG_{s}^{1}$ is congruent to the canonical lifting modulo $\pi_{s}$ but not modulo $\pi_{s}^{2}$, whereas $\calG_{s-\mu}^{1}$ is congruent to the canonical lifting modulo $\pi_{s-\mu}=\pi_{s}^{\N\wp^{\mu}}$; this implies \eqref{congru}.
   Then for  each irreducible component  $\underline{a}$ in the support of $\underline{(h)}$, the sequence $[H_{s,w'}:H_{w'}](\underline{a}, \underline{y_{s}})$ stabilises to either $0$ or $1$, so that the expression \eqref{congru0} is indeed eventually constant along the arithmetic progression.
\end{proof}

\begin{lemm}\label{ismodular2} For each divisor  $D\in\Div^{0}(X)(E_{v})$ (respectively, $\widehat{D}\in \Div^{\Ar}(X)$),
the element of $\bcalS^{\rm ord}$ given by 
$$ m\mapsto \mathcal{R}_{\wp}^{2}\mathcal{R}_{\wp}^{(\mu)} \langle D, T(m)z\rangle_{w}\qquad {\textit{ (respectively } }m\mapsto  \mathcal{R}_{\wp}^{2}\mathcal{R}_{\wp}^{(\mu)} \langle \widehat{D}, T^{0}(m)\hat{\eta}\rangle_{w}^{\Ar}{\textit{)}}$$
is well-defined 
independently of the choice of $D$ in its class $[D]$ (of $\widehat{D}$ in its class $[\widehat{D}]$); it will be denoted by
$$ \mathcal{R}_{\wp}^{2}\mathcal{R}_{\wp}^{(\mu)} \langle [D], T(m)z\rangle_{w} \qquad  {\textit{ (respectively } }  \mathcal{R}_{\wp}^{2}\mathcal{R}_{\wp}^{(\mu)}\langle [ \widehat{D}], T^{0}(m)\hat{\eta}\rangle_{w}^{\Ar}\textit{)}.$$

If $\widehat{D}=D$, then the two elements coincide as elements of $\bcalS^{\rm ord}$; moreover, 
for the arithmetic function $\Psi_{w}\in\bcalS$ with $\Psi_{w}(m)=\langle \hat{\eta}, T^{0}(m)\hat{\eta}\rangle_{w}^{\Ar}$, we have
\begin{equation}\label{htvsarp}
 \mathcal{R}_{\wp}^{2}\mathcal{R}_{\wp}^{(\mu)}\Psi_{w}  \sim  \mathcal{R}_{\wp}^{2}\mathcal{R}_{\wp}^{(\mu)}\langle [z], T(m)z\rangle_{w} 
\end{equation}
in $\bcalS^{\rm ord}$. 
\end{lemm}
\begin{proof}

The first part follows from Lemma \ref{killdiv}.
For the second part we may argue as in the proof of  Lemma \ref{discard}: for example, in $\bcalS$ we have 
\begin{align*}
 \mathcal{R}_{\wp}^{2}\mathcal{R}_{\wp}^{(\mu)}\langle\hat{\eta}, T^0(m)\hat{\eta}\rangle_w^\Ar  &\sim  \mathcal{R}_{\wp}^{2}\mathcal{R}_{\wp}^{(\mu)}\langle\hat{\eta}+\widehat{\divisor}(h), T^0(m)\hat{\eta}\rangle_w^\Ar \\
 &\sim  \mathcal{R}_{\wp}^{2}\mathcal{R}_{\wp}^{(\mu)}\langle\hat{z}+\widehat{\divisor}(h), T(m)\hat{z}\rangle_w^\Ar= \mathcal{R}_{\wp}^{2}\mathcal{R}_{\wp}^{(\mu)}\langle z+\divisor(h), T(m)z\rangle_w.
\end{align*}
\end{proof}

\subsection{Local heights at $p$ / II}\label{sec:hp2}  
Here we prove the vanishing statement for $p$-adic local symbols asserted in Proposition \ref{killp}. In fact we will show the equivalent statement that 
$$\Lf(\mathcal{R}^4_p\mathcal{R}_{p}^{(\mu)}
\Psi_p)=\prod_{\wp\vert p }(\alpha_{\wp}^{\mu_{\wp}}-1) \Lf(\mathcal{R}^4_p\Psi_{p})=0,$$
where the integers  $\mu=(\mu_{\wp})_{\wp\vert p}$ are as in Proposition \ref{killdiv}.

Let $e_{f}\in {\bf T}_{Np}\otimes {M_{f}}$ be the maximal idempotent satisfying $T(m)\circ e_{f}= a(f, m) e_{f}$ for all $m$ prime to $Np$;\footnote{Recall that $M_{f}$ is the number field generated by the Fourier coefficients $a(m,f)$.} viewed as an endomorphism of $S_{N\prod_{\wp\vert p}\wp}$, it is the projector onto the subspace generated by $f$ and $[\wp]f$ for all the primes $\wp$ of $F$ dividing  $p$. With  $z_{f}=e_{f}[z]$, we have by \eqref{htvsarp}
$$e_{f}e\calR_{p}^{4}\mathcal{R}_{p}^{(\mu)}\Psi_{p}= \calR_{p}^{4}\mathcal{R}_{p}^{(\mu)}\langle z_{f},T(m)z\rangle_{p}$$
in $\bcalS_{N}^{\rm ord}$, where the left-hand side makes sense by the modularity part of Proposition \ref{killp} and the right-hand side makes sense by Lemma \ref{ismodular2}.
We also denote, for $w$ a place of $E$ above the $F$-prime $\wp\vert p$, and $i\geq 2$,
\begin{equation}\label{efpsip}
e_{f}e\calR_{\wp}^{i}\mathcal{R}_{\wp}^{(\mu_{\wp})}\Psi_{w}:=\calR_{\wp}^{i}\mathcal{R}_{\wp}^{(\mu_{\wp})}\langle z_{f},T(m)z\rangle_{w}
\end{equation}
where the right-hand side makes sense as an element of $\bcalS^{\rm ord}$ by Lemma \ref{ismodular2}. (As we have not shown that $\calR_{\wp}^{3}\Psi_{w}$ is modular, the left-hand side is not otherwise defined.) Then by definition we have 
\begin{equation}\label{separate}
e_{f}e \calR_{p}^{4}\mathcal{R}_{p}^{(\mu)}\Psi_{p}=\sum_{w\vert p} e_{f}e \calR_{p}^{4}\mathcal{R}_{p}^{(\mu)}\Psi_{w}.
\end{equation}

Now since $\Lf=\Lf\circ e_{f}=\Lf\circ e_{f}\circ e$,  by \eqref{separate} the desired result is implied by the following Lemma for all $\wp\vert p$.
 
\begin{lemm}\label{eigen}  Suppose that $f$ is ordinary at $\wp$. 
For each place $w$ of $E$ above $\wp\vert p$, the element $e_{f}\calR_{\wp}^{2} \mathcal{R}_{\wp}^{(\mu_{\wp})}\Psi_{w}$ is zero in  $\bcalS^{\rm ord}$.
\end{lemm}
\begin{proof}
The ordinarity assumption and Lemma \ref{unorm}
 (cf. \cite[Exemple 4.12]{PR})  imply that $z_f$ is ``almost'' a universal norm in the totally ramified  $\Z_p$-extension $E_{w,\infty}^{{\ell}}$ of $E_w$: that is, after perhaps replacing $z_{f}$ by an integer multiple, for each layer $E^{{\ell}}_{w,n}$ we have  
$$z_f=\Tr_n(z_n)$$ for some  $z_n\in e_{f}J(X)(E^{{\ell}}_{v,n})$, where $\Tr_n=\Tr_{{E^{\ell}_{w,n}/E_w}}$ . Then we have  
$$e_{f}\calR_\wp^{2} \mathcal{R}_{\wp}^{(\mu_{\wp})}\Psi_w(m)=\calR_{\wp}^{2} \mathcal{R}_{\wp}^{(\mu_{\wp})}\langle \Tr_{n}(z_n), T(m) z\rangle_w= \calR_{\wp}^{2} \mathcal{R}_{\wp}^{(\mu_{\wp})}\langle z_n, T(m) z\rangle_{w,n}.$$
where  $\langle\ , \ \rangle_{w,n}$ is the local height pairing on $\Div^0(X)(E^{{\ell}}_{w,n})$ associated with the logarithm $\ell_{n,v}=\ell_{w}\circ N_{E^{{\ell}}_{w,n}/E_{w}}$. By Proposition \ref{theo:ht}.\ref{comp}-\ref{boundedness},   the right-hand side above has image in $c^{-1}\mathrm{Im}(\ell_n)\subset\Z_p$ for a uniform nonzero constant $c\in\Z_p$. As the extension  $E^{{\ell}}_{w,n}/E_{w}$ has ramification degree $p^{n}$, we have  for some nonzero $c'\in\Z_p$
$$e_{f}\calR_{\wp}^{2}\mathcal{R}_{p}^{(\mu_{\wp})}\Psi_v(m)\in c^{-1}\mathrm{Im}(\ell_n)\subset c'^{-1}p^n\Z_p$$
for all $n$; therefore $e_{f}\calR_{\wp}^{2}\mathcal{R}_{p}^{(\mu_{\wp})}\Psi_w=0$. 
\end{proof}

\begin{part}{Main theorem and consequences}

\section{Proof of the main theorem}\label{proofmt}
In this section we prove Theorem \ref{theoremB}.

\subsection{Basic case}\label{8.1}

First we prove the formula when $\Delta_{E/F}$ is totally odd and each prime $\wp$ of $F$ dividing $p$  splits in $E$. 

Let $\Psi_{\W}\in \bcalS_{N}$ denote the modular form with coefficients $\langle [z], T(m)[z]\rangle_{\W}$, where $\W=\nu\circ\frakN$ and  $\langle \ , \ \rangle_{\W}$ is the height pairing on $J(X)(E)$ associated with the $p$-adic logarithm $\ell_{F}\circ\frakN$, with $$\ell_{F}={d\over ds}\nu^{s}|_{s=0}\colon F^{\times}\bks F_{\A^{\infty}}^{\times }\to \Q_{p}.$$

Recall that $\Lf$ is a continuous functional, so that it commutes with limits and
$$L_{p,\W}'(f_{E})(\one)=\Lf\left({d\over ds}\Phi(\W^{s})|_{s=0}\right)= \Lf(\Phi'_{\W}).$$
We compare the Fourier coefficients of $\Phi'_{\W}$ and $\Psi_{\W}=\Psi_{\W, {\rm fin}}+\Psi_{\W, p}$.

\begin{prop}\label{psifinetc} Suppose that all of the prime ideals $\wp$ of $F$ dividing $p$ are principal.
Then we have
$$ \left(\prod_{\wp\vert p}U_{\wp}^{4}-U_{\wp}^{2}\right)\Phi'_{\W}
\sim \left(\prod_{\wp\vert p}(U_{\wp}-1)^{4}\right) \Psi_{\W,{\rm fin}}$$
in the quotient space $\bcalS_{N}^{\rm ord}/\Ker(\Lf).$
\end{prop}
\begin{proof} We prove that the identity holds in $\bcalS/(\mathcal{D}_{N} +\Ker(\Lf)$, where $\Ker(\Lf)$ denotes the image in $\bcalS$ of classical modular form killed by $\Lf$. Then since the left-hand side belongs to $\bcalS_{N}^{p{\rm -adic}}$, so does the right-hand side (and after further quotienting by $\Ker(e)$ we descend to   $\bcalS_{N}^{\rm ord}/\Ker(\Lf)$).

The coefficients  of $\Psi_{{\rm fin}}=\Psi_{\W,{\rm fin}}$ are computed in Proposition \ref{Psi}. To lighten the notation, we write the explicit expression for $\Psi(m)=\sum_{v {\textrm{ non-split}}}\Psi_{v}(m)$ as
$$\Psi_{v}(m)=\sum_{\substack{n\in S_{v}([m]) \\ v_{\wp}(nm)\geq 0\, \forall \wp\vert p}} c_{v}([nm]) r((1-n)m\Delta)r(nm\Delta/N\wp_{v}^{\epsilon(v)}),$$
where the value $c_{v}([nm])$ only depends on the prime-to-$p$ part of the fractional ideal  $nm$, and the set $S_{v}([m])$ only depends on $v$ and the prime-to-$p$ part of $m$; here $\epsilon(v)=1$ if $v$ is inert and $\epsilon(v)=0$ if $v$ is ramified.

The coefficients of $\Phi'$ are computed in Proposition \ref{Phi}. They look ``almost'' the same, in that, up to the modular form $h$ of Proposition \ref{Psi}, which is in $\Ker(\Lf)$, we have, when $m$ is divisible by every $\wp\vert p,$ 
$$\Phi'_{\W}(m)=\sum_{v{\textrm{\ non-split}}} \Psi_{v}^{[p]}(m),$$
where for a product $P$ of some of the primes $\wp\vert p$, we denote
$$\Psi_{v}^{[P]}(m)=\sum_{\substack{n\in S_{v}([m]) \\ v_{\wp}(nm)\geq 0\, \forall \wp\vert p\\ v_{\wp}(nm)=0\, \forall  \wp\vert P}} c_{v}([nm]) r((1-n)m\Delta)r(nm\Delta/N).$$

Then it is enough to show that for each $v\nmid p$, each $\wp\vert p$, and each $\wp\nmid P$ with $P$ as above, we have 
$$(U_{\wp}^{4}-U_{\wp}^{2})\Psi_{v}^{[P\wp]}=(U_{\wp}-1)^{4}\Psi_{v}^{[P]}.$$
For the sake  of notation we write the computation when $v$ is ramified in $E$ and $P=\prod_{\wp'\neq \wp}\wp'$ (for more general $P$ one just needs more notation to keep track of $v_{\wp'}(nm)$ for the primes $\wp'\neq \wp$).

The right-hand side equals 
\begin{equation}\label{3.20}
\sum_{i=0}^{4}(-1)^{i}{4\choose i}\sum_{\substack{n_{i}\in S_{v}([m])\\v_{\wp'}(n_{i}m)= 0 \, \forall \wp'\neq \wp, \wp'\vert p\\v_{\wp}(n_{i}m\wp^{i}) \geq 0}} c_{v}([n_{i}m])r((1-n_{i})m\wp^{i}\Delta)r(n_{i}m\wp^{i}\Delta/N).
\end{equation}  
From the relation $r(m_{0}\wp^{t})=(t+1)r(m_{0})$, valid for $\wp\nmid m_{0}$, we deduce the relations 
\begin{align*}
2r(m)&=r(m\wp)+r(m\wp^{-1}),\\
2r(m)&=r(m\wp^{2})+r(m\wp^{-2} )\quad \textrm{if } \wp\vert m,\\
2r(m)&=r(m\wp^{2})-r(m)  \quad \textrm{if } \wp\nmid m,
\end{align*}
where we recall that $r(m)=0$ if $m$ is not an integral ideal.
Then we can pick a totally positive generator in $F$ for the ideal $\wp$, which abusing notation we will still denote by $\wp$, and  make the substitution $n_{i}=\wp^{t-i}n_{0}$ with $\wp^{t}||n_{i}m\wp^{i}$ to  write \eqref{3.20} as 
$$\sum_{t\geq 0} \sum_{\substack{n_{0}\in S_{v}(m)\\ v_{\wp'}(n_{0}m)=0\,\forall \wp\vert p}} c_{v}([n_{0}m])r((n_{0}m)^{(\wp)})(t+1)A_{t}$$
where we recall that for an ideal $m$ we denote $m^{(\wp)}=m\wp^{-v_{\wp}(m)} $, and 
\begin{align*}
A_{t}&=r(m\Delta\wp^{4}(1-n_{0}\wp^{t-4}))[t+1-2t+2(t-1)]\\
   &+r(m\Delta\wp^{2}(1-n_{0}\wp^{t-2}))\left[-2(t+2)+\begin{cases} 4(t+1)-2t &\textrm{if\ \ } t\geq 1\\ 3 &\textrm{if\ \ } t=0\end{cases}\right]\\
   &+r(m\Delta(1-n_{0}\wp^{t}))[t+3-2(t+2)+t+1].
\end{align*}
The three expressions in square brackets vanish when $t>0$ and yield, respectively, $1$, $1$, and $0$ when $t=0$. Substituting back $n_{4}=\wp^{t-4}n_{0}$ in the first line and $n_{2}=\wp^{t-2}n_{0}$ in the second line, we deduce that \eqref{3.20} equals 
$$(U_{\wp}^{4}-U_{\wp}^{2})\Psi_{v}^{[P\wp]}$$
as desired.\footnote{Cf. \cite[Proof of Proposition 3.20]{PR}.}
\end{proof}

Combining this Proposition with Proposition \ref{killp} which says
$$\Lf \left(\prod_{\wp\vert p}(U_{\wp}-1)^{4} \Psi_{\W,p}\right)=0,$$
we find for $\W=\nu\circ\frakN$:
\begin{align*}
D_{F}^{2}\prod_{\wp} (\alpha_{\wp}^{4}-\alpha_{\wp}^{2} )   L'_{p,\W}(f_{E}, \one)
&=\prod_{\wp}(\alpha_{\wp}^{4}-\alpha_{\wp}^{2})\left(1-{1\over \alpha_{\wp}^{2}}\right) \left(1-{\N \wp\over \alpha_{\wp}^{2}}\right)\Lf(\Phi'_{\W} )\\
&=\prod_{\wp}(\alpha_{\wp}-1)^{4}\left(1-{1\over \alpha_{\wp}^{2}}\right)\left(1-{\N \wp\over \alpha_{\wp}^{2}}\right)\Lf(\Psi_{\W})\\
&=\prod_{\wp}(\alpha_{\wp}-1)^{4}\left(1-{1\over \alpha_{\wp}^{2}}\right) \langle z_{f}, z_{f}\rangle_{\W}.
\end{align*}
Here, besides the definition of $L_{p}(f_{E})$ (Definition \ref{def:prs}) we have used various properties of the functional $\Lf$ from Lemma \ref{Lf} and the observation that the projection onto the $f$-component of the modular form $\Psi_{\W}\in S_{2}(K_{0}(N),\Q_{p})$ is $\one_{f}(\Psi_{\W})=\langle z_{f}, z_{f}\rangle_{\W}.$

\medskip

This completes the proof of Theorem \ref{theoremB} when $(\Delta_{E/F},2)=1$ and all primes $\wp\vert p$ split in $E$.
\subsection{Reduction to the basic case}\label{kobtrick}
The general case, where $E$ is only assumed to satisfy $(\Delta_{E/F}, Np)=1$, can be reduced to the previous one under the assumption 
$$L_{p,\W}'(f_{E}, \one)\neq 0$$
 by the following argument due to Kobayashi \cite[proof of Theorem 5.9]{kobayashi} using the complex Gross--Zagier formula (which is known with no restrictions on $\Delta$) and the factorisation $L_{p}(f_{E}, \chi\circ\frakN)\sim L_{p}(f, \chi)L_{p}(f_{\eps}, \chi)$.
 
 By the factorisation the orders of vanishing at the central point of the factors of $L_{p}(f_{E},\nu^{s}\circ \frakN)$ will be one (say for  $L_{p}(f)$) and zero (say for  $L_{p}(f_{\eps})$). Then, by the first part of Theorem \ref{theoremC},\footnote{Which can be proved by using the $p$-adic Gross--Zagier formula attached to a field $E'$ satisfying the assumptions of \S\ref{8.1}.} the orders of vanishing of $L(f,s)$ and $L(f_{\eps},s)$ at $s=1$ will also be one and zero. Moreover  the Heegner point $z_{f,E'}$ attached to $f$ and  any   $E'$  also satisfying  $L(f_{\eps_{E'/F}},1)\neq 0$   is non-torsion, and in fact its trace $z_{f,F}={\rm Tr}_{E'/F}(z_{f, E'})$ is non-torsion 
and $z_{f, E'}$ is up to torsion a  multiple of $z_{f,F}$ in $J(X)(E')\otimes\baar{\Q}$. 
Therefore, by  the complex and $p$-adic Gross--Zagier formulas for a suitable $E'$ satisfying the assumptions of \S\ref{8.1} and $L(f_{\eps_{E'/F}},1)\neq 0$, we have
\begin{align*} L_{p,\nu}'(f,\one)=\prod_{\wp\vert p}\left(1-{1\over \alpha_{\wp}}\right)^{2}
{L'(f,1)\over \Omega_{f}^{+}  \langle z_{f,F}, z_{f,F}\rangle}  \langle z_{f,F}, z_{f,F}\rangle_{\nu}
\end{align*}
where $\langle \ , \  \rangle_{\nu}$ is the $p$-adic height pairing on $J(X)(F)$ attached to $\nu$, and  $\langle \ , \  \rangle$ is the N\'eron--Tate height (the ratio appearing above belongs to $M_{f}^{\times}$ by the Gross--Zagier formula).   This allows us to conclude
\begin{align*}
L_{p,\W}'(f_{E}, \one) &= {\Omega_{f}^{+}\Omega_{{f}_{\eps}}^{+}\over D_{E}^{-1/2}\Omega_{f}} 
{ L_{p,\nu}'(f, \one)L_{p}(f_{\eps}, \one)}\\
&=D_{E}^{1/2}\prod_{\wp\vert p}\left(1-{1\over \alpha_{\wp}}\right)^{2}\left(1-{\eps(\wp)\over \alpha_{\wp}}\right)^{2}
{L'(f,1)L(f_{\eps}, 1)\over \Omega_{f} \langle z_{f,F}, z_{f,F}\rangle}  \langle z_{f,F}, z_{f,F}\rangle_{\nu}\\
&=D_{F}^{-2}\prod_{\wp\vert p}\left(1-{1\over \alpha_{\wp}}\right)^{2}\left(1-{\eps(\wp)\over \alpha_{\wp}}\right)^{2} {\langle z_{f,E}, z_{f,E}\rangle \over \langle z_{f,F}, z_{f,F}\rangle} \langle z_{f,F}, z_{f,F}\rangle_{\nu}\\
&=D_{F}^{-2}\prod_{\wp\vert p}\left(1-{1\over \alpha_{\wp}}\right)^{2}\left(1-{\eps(\wp)\over \alpha_{\wp}}\right)^{2} \langle z_{f,E}, z_{f,E}\rangle_{\W}.
\end{align*}
\begin{rema} It is natural to conjecture that when $L_{p,\W}'(f_{E}, \one)=0$ we should have $\langle z_{f}, z_{f}\rangle_{\W}=0$. However in this case the above argument fails because, without knowledge of the nontriviality of the $p$-adic height pairing, the vanishing of $L_{p}(f_{E},\W^{s})$ to order $\geq 2$ does not imply a similar high-order vanishing for $L(f_{E},s)$.
\end{rema}

\section{Periods and the Birch and Swinnerton-Dyer conjecture}\label{sec:periods}  
As seen in the Introduction, the application of our result to the Birch and Swinnerton-Dyer  formula rests on a conjectural relation among the periods of $f$ and the associated abelian variety $A$. Here we would like to briefly elaborate on this conjecture and its arithmetic consequences. (This section contains no new results or conjectures and is a very brief survey of work of Shimura and Yoshida \cite{yoshida}.) We retain the notation of the Introduction, and set $M=M_{f}$ and  $\dim\,A = [M:\Q]=d$.
\subsection{Real periods}\label{sec:realperiods}
The  conjecture on periods stated in the Introduction   can be refined to a conjecture on rationality rather than algebraicity.  First we need to  define  the automorphic periods $\Omega_{f^{\sigma}}^{+}$, for $\sigma\in \Hom(M,\C)$; they are naturally defined as elements of $\C^{\times}/M^{\times}$ (see \cite{ragh} for a modern exposition): one can choose them ``covariantly''  in the sense of \cite{yoshida} in order to have $\prod_{\sigma}\Omega_{f^{\sigma}}^{+}$ defined up to $\Q^{\times}$, or define directly the product as follows. Let $\mathcal{H}_{N}=Z(\A)\bks \GL_{2}(\A)/K_{0}(N)K_{\infty}$ be the open Hilbert modular variety   of level $N$. Then the perfect pairing of $\Q$-vector spaces
\begin{align}\label{periodpairing}
H_{g}(\mathcal{H}_{N},\Q)^{+}\times S_{2}(K_{0}(N), \Q) \to \C
\end{align}
(where ``$+$'' denotes the intersection of the  $+1$-eigenspaces for  the complex conjugations)
decomposes under the diagonal action of $\mathbf{T}_{N}$ into $\Q$-rational blocks {parametrised} by the Galois-conjugacy classes of eigenforms. Then 
$$\prod_{\sigma}\Omega_{f^{\sigma}}^{+}\in \C^{\times}/\Q^{\times}$$
is $(2\pi i)^{dg}$ times  the discriminant of the pairing on the rational block corresponding to $\{f^{\sigma}\}_{\sigma}$.  (The individual $\Omega_{f^{\sigma}}^{+}\in \C^{\times}/M^{\times}$ are defined as the discriminants of \eqref{periodpairing} on $\baar{\Q}$-rational $\mathbf{T}_{N}$-eigenblocks. One can similarly define periods $\Omega_{f^{\sigma}}^{-}$ by paring  with $H_{g}(\mathcal{H}_{N}, \Q)^{-}$, the $-1$-eigenspace for the complex conjugations.)
\begin{conj}\label{perQ} We have 
$$\Omega_{A}\sim \prod_{\sigma}\Omega_{f^{\sigma}}^{+}$$
in $\C^{\times}/\Q^{\times}$. 
\end{conj}
The conjecture is originally due to Shimura (see \cite{shimuraconj}, especially \S 11) and was refined by Yoshida \cite{yoshida}. When $A$ has complex multiplication, it has been proved by Blasius \cite{blasius}.
%\footnote{The conditions on the field of definition of the complex multiplication in \cite{blasius} are satisfied in our case by \cite{silverberg}.}
It is also known when $F=\Q$; before discussing that, let us translate it into a  language closer to  conjectures of Shimura.

For each $\tau\colon F\to \R$, let $f_{B(\tau)}$ be the Jacquet-Langlands transfer of $f$ to a rational\footnote{For consistency with the case in which $B(\tau)=\GL_{2}(\Q)$ and ``rational'' means ``rational $q$-expansion coefficients'', here $f_{B(\tau)}$ is considered $F$-rational for  the  structure $H^{0}(X/ F, \Omega_{X/ F})\otimes (2\pi i)^{-1} \Q\subset H^{0}(X/ F, \Omega_{X/ F})\otimes_{F, \tau}\C$ (where $X$ is the Shimura curve defined in the Introduction).} form on the quaternion algebra $B(\tau)/F$ defined in the Introduction (recall that $B(\tau)$ is ramified at all infinite places except $\tau$), and let $X$ be our Shimura curve. Then $A$ is (up to isogeny)  a quotient  $\phi$ of $J(X)$, and for each embedding $\tau$ we can write
$$\phi^{*}\omega_{A}=c_{\tau}\bigwedge_{\sigma} 2\pi i  f_{B(\tau)}^{\sigma}(z)\, dz$$
as forms in $H^{0}(J(X)(\C_{\tau}), \Omega^{d})$,  for some $c_{\tau}\in F^{\times}$ (since both are generators of a rank one $F$-vector space); here $z$ denotes the coordinate on the upper half-plane uniformising $X$.
Then we have 
$$\int_{A(\R_{\tau})}|\omega_{A}|_{\tau}\sim \prod_{\sigma} \Omega_{f^{\sigma}_{B(\tau)}}^{+} \textrm{  in }\C^{\times}/F^{\times},$$
where $ \Omega_{f^{\sigma}_{B(\tau)}}^{+}$ is $2\pi i$ times the discriminant of the ${f^{\sigma}_{B(\tau)}}$-part of the analogue of the pairing \eqref{periodpairing} on $X(\C_{\tau})$.
When choices are made covariantly in $\tau$, we then get $\Omega_{A}\sim  \prod_{\sigma,\tau} \Omega_{f^{\sigma}_{B(\tau)}}^{+} $ in $\C^{\times}/\Q^{\times}$.

Our conjecture, decomposed into its $\sigma$-constituents, can then be rewritten  as 
\begin{align}\label{eqper}
\Omega_{f}^{+}\sim \prod_{\tau}\Omega_{f_{B(\tau)}}^{+} \textrm{ in }\C^{\times}/(MF)^{\times}.
\end{align}
In this form, this is a stronger  version of Shimura's conjecture \cite{P-inv} on the factorisation of periods of Hilbert modular forms up to algebraic factors in terms of $P$-invariants. The reader is referred to \cite{yoshida} for a discussion of this point.

Notice that \eqref{eqper}  is nontrivial even when $F=\Q$: it asserts that the periods of the transfers of $f$ to any indefinite quaternion algebra have the same transcendental (or irrational) parts. However, in this case the conjecture is known by the work of  Shimura \cite{shimura31} (for the  algebraicity) and Prasanna \cite{prasanna} (for the  rationality).  

For general $F$, Shimura's conjecture on $P$-invariants is largely  proved by Yoshida \cite{yoshida2} under an  assumption of non-vanishing of  certain $L$-values.

\begin{rema} It is clear that our conjecture implies  that the Birch and Swinnerton-Dyer conjectural formula is true up to a nonzero rational factor when $A$ has analytic $M$-rank zero. By the complex (respectively, the  $p$-adic) Gross--Zagier formula, the conjecture for $f$ also implies the complex (respectively, the  $p$-adic) Birch and Swinnerton-Dyer formulas up to a rational factor when $A$ has ($p$-adic) analytic $M$-rank one.
\end{rema}

\subsection{Quadratic periods}
We can formulate a conjecture analogous to Conjecture \ref{perQ} for the periods of the base-changed abelian variety  $A_{E}=A\times_{\Spec F}\Spec E$.
\begin{conj}\label{quadper} We have
$$\Omega_{A_{E}}\sim \prod_{\sigma}\Omega_{f^{\sigma}}$$
in $\C^{\times}/\Q^{\times}$.
\end{conj}
Here  the period of $A_{E}$ is
$$\Omega_{A_{E}}=\prod_{\tau:E\to\C}\int_{A(\C_{\tau})} |\omega_{A_{E}}|_{\tau},$$
where for a differential form $\omega=h(z)dz_{1}\wedge\cdots\wedge dz_{k}$ we have $|\omega|_{\tau}=|h(z)|_{\tau}^{2}dz_{1}\wedge d\baar{z}_{1}\wedge\cdots\wedge dz_{k}\wedge d\baar{z}_{k}$.

As above, this conjecture can be ``decomposed'' into 
\begin{align}\label{eqperquad}
\Omega_{f}\sim \prod_{\tau}\Omega_{f_{B(\tau)}} \textrm{ in }\C^{\times}/(MF)^{\times}.
\end{align}
where $\Omega_{f_{B(\tau)}}$ is $\pi^{2}$ times the Pertersson inner product of  ${f_{B(\tau)}}$. This is essentially Shimura's conjecture on $Q$-invariants (see \cite{P-inv}). Up to algebraicity it has been proved by Harris \cite{harris} under a local condition   (a new proof of the same result should appear in forthcoming work of Ichino--Prasanna, yielding rationality and removing the local assumption).  As $\Omega_{f}=\Omega_{f}^{+}\Omega_{f}^{-}$,\footnote{See e.g. \cite[Theorem 4.3 (II)]{shimura}, where the assumption on the weight can now be removed thanks to the work of Rohrlich.} the factorisation \eqref{eqperquad}  is implied by \eqref{eqper} and its analogue for $\Omega_{f}^{-}$; thus Harris's result  can be seen as evidence for the conjecture on real periods. 

We take the opportunity to record an immediate  consequence of the conjecture on quadratic periods and the Gross--Zagier formulas. 

\begin{theo} If $A_{E}$ has complex (respectively, $p$-adic) analytic $M$-rank $\leq 1$, then the complex (respectively, the $p$-adic) Birch and Swinnerton-Dyer formula for $A_{E}$ is true up to a nonzero  algebraic factor. 
\end{theo}

\end{part}

\backmatter
\addtocontents{toc}{\medskip}

\end{document}